\documentclass[10pt]{article}
\usepackage{amsmath} 
\usepackage{amssymb} 
\usepackage{array} 
\usepackage{bbm} 
\usepackage{amsthm} 
\usepackage{xcolor} 
\usepackage[]{algorithm2e} 
\usepackage[shortlabels]{enumitem} 
\usepackage{booktabs}

\usepackage{gordon-se}
 
\usepackage{mathtools} 

\def\Law{\mathsf{Law}}
\def\SE{\mathrm{SE}}

\def\saddle{\mathsf{saddle}}

\newcommand\numberthis{\addtocounter{equation}{1}\tag{\theequation}}
\newtheoremstyle{myremark} 
    {\topsep}                    
    {\topsep}                    
    {\rm}                        
    {}                           
    {\bf}                        
    {.}                          
    {.5em}                       
    {}  

\newcommand\blfootnote[1]{%
  \begingroup
  \renewcommand\thefootnote{}\footnote{#1}%
  \addtocounter{footnote}{-1}%
  \endgroup
}

\newtheorem{theorem}{Theorem}[section]
\newtheorem{lemma}[theorem]{Lemma}

\newtheorem{proposition}[theorem]{Proposition}
\newtheorem{definition}{Definition}[section]
\theoremstyle{myremark}

\allowdisplaybreaks
\usepackage[top=1in,bottom=1in,left=1in,right=1in]{geometry}
\usepackage{geometry}
\usepackage[utf8]{inputenc}
\input{commands}

\begin{document}
\begin{center}
	
	{\bf{\LARGE{State evolution beyond first-order methods I: \\ Rigorous predictions and finite-sample guarantees}}}
	
	\vspace*{.2in}
	
	{\large{
			\begin{tabular}{ccc}
				Michael Celentano$^{\circ}$, Chen Cheng$^{\ast}$, Ashwin Pananjady$^{\star,\dagger}$, Kabir Aladin Verchand$^{\star, \ddagger}$
			\end{tabular}
	}}
	\vspace*{.2in}
	
	\begin{tabular}{c}
		Department of Statistics, UC Berkeley$^{\circ}$\\
		Department of Statistics, Stanford University$^{\ast}$ \\
                Statistical Laboratory, University of Cambridge$^\ddagger$ \\
        Schools of Industrial and Systems Engineering$^\star$ and Electrical and Computer Engineering$^\dagger$, \\
		Georgia Institute of Technology
	\end{tabular}
	
	\vspace*{.2in}

	\today
	
	\vspace*{.2in}
\begin{abstract}
	\blfootnote{Authors listed in alphabetical order. MC contributed to this work while he was a Miller fellow at UC Berkeley.} We develop a toolbox for exact analysis of iterative algorithms on a class of high-dimensional nonconvex optimization problems with random data.  While prior work has shown that low-dimensional statistics of (generalized) first-order methods can be predicted by a deterministic recursion known as state evolution, our focus is on developing such a prediction for a more general class of algorithms. We provide a state evolution for any method whose iterations are given by (possibly interleaved) first-order and saddle point updates, showing two main results. First, we establish a rigorous state evolution prediction that holds even when the updates are not coordinate-wise separable. Second, we establish finite-sample guarantees bounding the deviation of the empirical updates from the established state evolution.  In the process, we develop a technical toolkit that may prove useful in related problems.  One component of this toolkit is a general Hilbert space lifting technique to prove existence and uniqueness of a convenient parameterization of the state evolution.  Another component of the toolkit combines a generic application of Bolthausen’s conditioning method with a sequential variant of Gordon’s Gaussian comparison inequality, and provides additional ingredients that enable a general finite-sample analysis. 
\end{abstract}
\end{center}
\tableofcontents

\section{Introduction} \label{sec:intro}
Iterative algorithms underpin several methods in statistics and engineering.  For instance, they are employed in decoders for wireless communication~\citep{tse2005fundamentals,richardson2008modern}, as general purpose methods in convex and nonconvex optimization~\citep{boyd2004convex,nocedal2006numerical,jain2017non}, in the training of large-scale machine learning models~\citep{bottou2018optimization,lan2020first}, and as powerful methods in the study of average-case computational complexity~\citep{gamarnik2022disordered}.  From an analysis point of view, geometric properties inherent to the problem can be used to obtain rough quantitative bounds on the error dynamics of iterative algorithms even on nonconvex landscapes~\citep[e.g.][]{chi2019nonconvex,ge2016matrix,ge2017no}. However, these bounds may in some cases be (i)~overly conservative when the data is random, as is often the case in statistical learning and signal processing problems~\citep{blumensath2008iterative,agarwal2012fast}, or (ii)~unavailable in the absence of global problem structure~\citep{nemirovski1983problem,nesterov2018lectures}.  Motivated by the need for sharp average-case analysis of optimization algorithms, we study the \emph{precise} dynamics of iterative methods with random data, with the aim of producing deterministic recursions that accurately track low-dimensional functionals of the iterates. In the literature, it is common to refer to predictions of this flavor using the umbrella terminology of~\emph{state evolution}~\citep{donoho2009message}

To set the stage, let $n, d \in \mathbb{N}$ denote a sample size and parameter dimension, respectively, and let $\bX \in \mathbb{R}^{n \times d}$ have entries $(X_{ij})_{1 \leq i \leq n, 1 \leq j \leq d} \overset{\mathsf{iid}}{\sim} \mathsf{N}(0, 1/d)$. Iterative algorithms with random data can often be viewed as methods that make use of the  matrix $\bX$ to produce sequences $\{(\bu_k, \bv_k)\}_{k \in \mathbb{N}} \in \mathbb{R}^n \times \mathbb{R}^d$. 
A canonical update that produces such a sequence is a \emph{general first-order method} (GFOM), which we refer to throughout as just a first-order update. This update is specified by (potentially nonseparable) Lipschitz functions $f_{k + 1}^u: \mathbb{R}^n \rightarrow \mathbb{R}^n$ and $f_{k+1}^v: \mathbb{R}^d \rightarrow \mathbb{R}^d$ and given by
\[
\bu_{k+1} = f_{k+1}^u\bigl(\bX \bv_k\bigr) \quad \text{ and } \quad \bv_{k+1} = f_{k+1}^v\bigl(-\bX^{\top} \bu_{k}\bigr), \tag{\text{First-order}}
\]
where these functions can have additional dependence on previous iterates. See Section~\ref{sec:setup} for the precise definition of this general class of updates and examples. First-order updates encompass not just gradient-based methods on canonical loss functions, but can also capture other methods that are not necessarily based on minimizing a loss~\citep[e.g.][]{berthier2020state,rangan2011generalized,lou2025accurate}. 

The development of state evolution predictions for first-order updates dates back to approximate message passing (AMP) (see~\citet{donoho2009message,bayati2011dynamics} and \cite{feng2022unifying} for background). The reduction of general first-order methods to AMP was first given by~\citet{celentano2020estimation}.  
 Recent years have seen a flurry of work leveraging state evolution for first-order methods in a variety of different settings including gradient flows~\citep{celentano2021high,bordelon2023dynamics}, stochastic gradient methods~\citep{gerbelot2024rigorous,mignacco2020dynamical,mannelli2019passed}, inference along the trajectory~\citep{bellec2024uncertainty,han2024gradient}, and in sampling applications~\citep{fan2025dynamical,liang2023high}, to name a few.  While the aforementioned papers are based on either AMP state evolution, Gaussian conditioning arguments, and/or dynamical mean field theory, trajectory predictions for various first-order methods have also been derived via deterministic equivalents rooted in random matrix theory~\citep[see, e.g.,][and the references therein]{collins2024hitting,ferbach2025dimension} or via somewhat more direct arguments~\citep{tan2023online,arous2021online,arous2024high,abbe2023sgd}.

While many practical methods can be written as first-order updates, higher-order methods that are commonplace in numerical optimization do not execute this type of update.  
Consider for instance iterative methods for $M$-estimation problems in which we would like to minimize the objective $\ell(\bX \bv)$ over the decision variable $\bv \in \mathbb{R}^d$, where $\ell: \mathbb{R}^n \rightarrow \mathbb{R}$ is a smooth, strongly convex function.  In poorly conditioned problems, a numerically stable and preferred algorithm is the proximal point method~\citep[see, e.g.,][]{drusvyatskiy2018proximal,rockafellar1976monotone}, which forms the iterates $\{\bv_k\}_{k \in \mathbb{N}}$ by solving a regularized optimization problem
$\bv_{k+1} = \argmin_{\bv \in \mathbb{R}^d} \; \Bigl\{\ell(\bX\bv) + \frac{1}{2\alpha_k} \| \bv - \bv_{k} \|_2^2\Bigr\}$ at each iteration.
Evidently, such an iteration cannot be written as a first-order update. However, since the convex conjugate of a smooth, strongly convex function is also smooth and strongly convex, we may write each iteration as the min-max problem
\[
(\bu_{k+1}, \bv_{k+1}) = \underset{\bu \in \mathbb{R}^n, \bv \in \mathbb{R}^d}{\mathsf{saddle}}\; \Bigl\{\bu^{\top} \bX \bv - \ell^{\star}(\bu) + \frac{1}{2\alpha_k} \| \bv - \bv_k \|_2^2\Bigr\},
\]
where $\ell^{\star}$ denotes the Fenchel--Legendre conjugate of $\ell$. 
More generally, many complex methods --- for potentially nonconvex optimization problems --- can be written as \emph{saddle point updates} of the above form. In general, such an update can be written terms of two (potentially nonseparable) strongly convex and smooth penalty functions $\phi_{k + 1}^u: \mathbb{R}^n \rightarrow \mathbb{R}$ and $\phi_{k + 1}^v: \mathbb{R}^d \rightarrow \mathbb{R}$ as
\[
			(\bu_{k+1},\bv_{k+1})
=
\underset{\bv \in \mathbb{R}^p, \bu \in \mathbb{R}^n}{\saddle}
\Big\{
\bu^\top \bX \bv
-
\phi_{k+1}^u(\bu)
+
\phi_{k+1}^v(\bv)
\Big\} \tag{\text{Saddle}},
\]
where these functions can have additional dependence on previous iterates.
In Section~\ref{sec:setup}, we provide a precise definition of this general class of updates as well as additional examples.

In contrast to their first-order counterparts, higher-order and saddle updates have received relatively sparse attention from a state evolution perspective, aside from now-classical papers studying the one-shot saddle update for $M$-estimation problems~\citep{bean2013optimal,donoho2016high,thrampoulidis2018precise}. For iterative updates, a few papers have considered the simplified sample-split setting, in which independent, fresh data is used at every iteration. Deterministic trajectory predictions have recently been analyzed to obtain concrete convergence guarantees of algorithms such as alternating minimization~\citep{chandrasekher2023sharp, chandrasekher2024alternating}, stochastic prox-linear methods~\citep{lou2024hyperparameter}, and iteratively reweighted least squares~\citep{kaushik2024precise}. In an effort to go beyond sample-split algorithms, non-rigorous replica predictions for the dynamics of a similar higher-order method were provided in~\citet{okajima2025asymptotic}. A complementary body of recent work leverages random matrix theory to upper bound the dynamics of a higher-order optimization method termed ``Hessian descent"~\citep{montanari2025solving}; see also papers by~\citep{subag2021following,jekel2025potential}. Two-step variants of saddle updates were also considered by~\citet{celentano2021cad} and~\citet{celentano2023challenges} in the context of analyzing debiased statistical estimators.
Complex methods have also been considered in two-stage optimization problems in which either (i) a spectral method is used for initialization followed by a first-order method such as AMP to locally refine the estimates~\citep{mondelli2021approximate} or (ii) a first-order method is applied to localize the iterates and probe landscape properties of complicated random optimization problems~\citep{celentano2023local,celentano2024sudakov}.

\subsection*{Contributions and organization}

The main contribution of this paper is to develop a rigorous state evolution for iterative methods whose updates can be written as either a first-order update or a saddle point update, as well as interleaved variants of such updates. 
In Section~\ref{sec:setup}, we formalize the setting and provide several more examples.  We then present our two main results in Section~\ref{sec:main-results}.  Our first main result, Theorem~\ref{lem:fix-pt-exist-unique}, provides an explicit description of the state evolution, accompanied by a proof of uniqueness.  Our second main result, Theorem~\ref{thm:exact-asymptotics} provides finite-sample guarantees for the adherence of the empirical iterations to their state evolution predictions.  The proofs of the two theorems are then provided in Sections~\ref{sec:exist-unique} and~\ref{sec:proof-exact-asymptotics}, respectively. We conclude with a discussion in Section~\ref{sec:discussion}, and some ancillary and technical proofs are postponed to the appendices. In a companion paper, we provide an infinite-dimensional, canonical representation of state evolution in our setting, as well as concrete convergence consequences for particular iterative algorithms.

Let us highlight some technical takeaways from the present paper. To our knowledge, our parameterization of the state evolution is nonstandard but transparently allows --- in Theorem~\ref{lem:fix-pt-exist-unique} --- for a proof of existence and uniqueness in nonseparable settings. We believe that similar parameterizations and proof techniques can be applied to other related problems. In addition, our techniques for proving Theorem~\ref{thm:exact-asymptotics} involve several pieces that may be of potential independent interest. To handle successive saddle updates, we build on Bolthausen's conditioning argument~\citep{bolthausen2014iterative} and introduce a sequential variant of the convex Gaussian minmax theorem (CGMT) that may may find broader applications; the one-shot variant~\citep{gordon1985some,gordon1988milman,stojnic2013framework,thrampoulidis2015regularized} has been extensively used in the literature. Finally, our technique for establishing finite-sample concentration results relies on stability properties of the established state evolution, and is different from existing techniques based on directly handling AMP iterations. Details are discussed in Section~\ref{sec:main-results}.

\subsection*{Notation}
Throughout, we will let $\Lambda = n/d$ for convenience.
For a positive integer $m$, we use $[m]$ to denote the set of natural numbers less than or equal to $m$. We let $\mathbbm{1}\{\cdot\}$ denote the
indicator function. Denote by $\mathbb{B}_2(r; \btheta)$ the $\ell_2$ ball of radius $r$ centered at $\btheta$.  Let $\mathbb{S}^{k}_{\geq 0} \subseteq \mathbb{R}^{k \times k}$ denote the set of positive, semi-definite symmetric matrices.  For two sequences of non-negative reals $\{f_n\}_{n
	\geq 1}$ and $\{g_n \}_{n \geq 1}$, we use $f_n \lesssim g_n$ to indicate that
there is a universal positive constant $C$ such that $f_n \leq C g_n$ for all
$n \geq 1$.   Throughout, we use $(c, C, \ldots)$ to denote universal
positive constants, and their values may change from line to line. 

We denote by $\mathsf{N}(\bm{\mu}, \bm{\Sigma})$ a normal distribution with mean $\bm{\mu}$ and covariance matrix $\bm{\Sigma}$. For a covariance matrix $\bK \in \mathbb{S}^k_{\geq 0}$, a collection of random vectors $\{\bz_i\}_{i=1}^{k} \in \mathbb{R}^m$ is distributed as $\bz_1, \bz_2, \ldots, \bz_k \sim \mathsf{N}(0, \bK \otimes \bI_m)$ means that $\{\bz_i\}_{i=1}^k$ form a collection of centered Gaussian random vectors with pairwise covariance matrices $\EE[\bz_\ell \bz_{\ell'}^{\top}] = K_{\ell, \ell'} \bI_m$.  Equivalently, the matrix $\bZ_k \in \reals^{m \times k}$ having $i$-th column $\bz_i$ has i.i.d. rows each drawn from the Gaussian distribution $\normal(0, \bK)$.  We say that $X \overset{(d)}{=} Y$ for two random variables $X$ and $Y$ that are equal in distribution. 

Given a measurable space $(\Omega, \mathcal{F})$ and a measure $P$, we let $L^2(\Omega, \mathcal{F}, P)$ denote the Hilbert space of square-integrable functions on $(\Omega, \mathcal{F})$ with respect to the measure $P$.  When the context is clear, we abuse notation and simply write $L^2$.  When $\Omega$ itself is an inner product space, for two elements $X, X' \in L^2(\Omega, \mathcal{F}, P)$, we write $\langle X, X' \rangle_{L^2} = \mathbb{E}\langle X, X' \rangle$.  For a topological space $\mathbb{X}$, we will let $\mathcal{B}(\mathbb{X})$ denote its Borel $\sigma$-algebra.  We will use square brackets $[ \cdot ]$ to enclose the arguments of a linear functional $\mathsf{f}: L^2 \rightarrow \mathbb{R}$ and the usual round brackets $(\cdot)$ to enclose the argument of a function $f: \mathbb{R} \rightarrow \mathbb{R}$.

For a pair of matrices $\bA, \bB \in \mathbb{R}^{m \times k}$, we define the operation $\llangle \cdot, \cdot \rrangle: \mathbb{R}^{m \times k} \times \mathbb{R}^{m \times k}$ as $\llangle \bA, \bB \rrangle := \bA^{\top} \bB \in \mathbb{R}^{k \times k}$.  We will additionally write $\llangle \bA, \bB \rrangle_{L^2} = \mathbb{E}\llangle \bA, \bB \rrangle$ with expectations taken entry-wise.  For a matrix $\bA \in \mathbb{R}^{n \times d}$, we will let $\| \bA \|_{\infty}$ denote its maximum absolute entry and $\kappa(\bA):= \sigma_{\max}(\bA)/\sigma_{\min}(\bA)$ denote its condition number. 

For $\mu \in \mathbb{R}_{> 0}$ and a convex function $f: \mathbb{R}^d \rightarrow \mathbb{R}$, we define the proximal operator $\mathrm{prox}_{\mu f}: \mathbb{R}^d \rightarrow \mathbb{R}$ as 
\[
\mathrm{prox}_{\mu f}(x) := \argmin_{y \in \mathbb{R}^d} \Bigl\{ f(y) + \frac{1}{2\mu} \| y - x \|_2^2 \Bigr\}.
\]
We say that a multivariate function $h: \mathbb{R}^{p} \rightarrow \mathbb{R}^p$ is $L$-Lipschitz if $\| h(x) - h(x') \|_2 \leq L \| x - x' \|_2$ for all $x, x' \in \mathbb{R}^p$.
 Let $(\mathbb{X}, \| \cdot \|)$ be a normed space and $k \in \mathbb{N}$.  We say that a function $\psi: \mathbb{X}^{\otimes k} \rightarrow \mathbb{R}$ is order-$2$ pseudo-Lipschitz with constant $L$ if, for all $(x_1, \ldots x_k), (x'_1, \ldots, x'_k) \in \mathbb{X}^{\otimes k}$, 
\[
\lvert \psi(x) - \psi(x') \rvert \leq L \cdot \biggl(1 + \sum_{\ell=1}^{k} \| x_\ell \| + \| x'_{\ell} \| \biggr) \cdot \sum_{\ell=1}^{k} \| x_{\ell} - x'_{\ell} \| \quad \text{ and } \quad \lvert \psi(0, \ldots, 0) \rvert \leq L.  
\] 
  We adopt the convention that for the function $h: \mathbb{R}^n \rightarrow \mathbb{R}$ defined as $h(\bx) = \| \bx \|_2$, we have $\partial h(0) = 0$.

\section{Formal setting} \label{sec:setup}

In this section, we formally introduce the type of iterative algorithms considered in this paper, and provide several examples of such algorithms.

\subsection{Family of iterative algorithms}

As alluded to in Section~\ref{sec:intro}, we consider iterations that make use of the matrix $\bX$ with (a possible mix of) first-order and saddle updates, accommodating algorithms whose iterates may involve some form of auxiliary randomness. Let $\{\beps_{u, k}\}_{k \in \mathbb{N}} \overset{\mathsf{iid}}{\sim} \mathsf{N}(0, \bI_n/n)$ and $\{\beps_{v, k}\}_{k \in \mathbb{N}} \overset{\mathsf{iid}}{\sim} \mathsf{N}(0, \bI_d/d)$ denote auxiliary randomness independent of $\bX$.  We consider iterative algorithms which form sequences of pairs $\{(\bu_k, \bv_k)\}_{k \in \mathbb{N}} \in \mathbb{R}^{n} \times \mathbb{R}^d$, which are initialized as 
\begin{equation} \label{eq:initialization}
	\bu_1 = f_1^u(\beps_{u,1}),
	\quad \text{ and } \quad
	\bv_1 = f_1^v(\beps_{v,1}).
\end{equation}
Here $f_1^{u}: \reals^{n} \rightarrow \reals^{n}$ and $f_1^{v}: \reals^{d} \rightarrow \reals^{d}$ are, respectively, $L$-Lipschitz functions of their arguments.  Note that this allows for both random initialization as well as deterministic initializations --- in the latter case we can choose $(f_1^u, f_1^v)$ to be suitable constant functions so that $\bu_1 = f_1^u(0)$ and $\bv_1 = f_1^v(0)$. 

After $k \in \mathbb{N}$ iterations, we collect the previous iterates $\bu_1, \ldots, \bu_k$ and $\bv_1, \ldots, \bv_k$ into matrices
\[
\bU_k = [\bu_1 \; \vert \; \bu_2 \; \vert \; \cdots \; \vert \; \bu_k] \in \mathbb{R}^{n \times k} \qquad \text{ and } \qquad \bV_k = [\bv_1 \; \vert \; \bv_2 \; \vert \; \cdots \; \vert \; \bv_k] \in \mathbb{R}^{d \times k},
\]
and collect the auxiliary randomness $\{\beps_{u,\ell}\}_{1 \leq \ell \leq k}$ and $\{\beps_{v,\ell}\}_{1 \leq \ell \leq k}$ into matrices $\bE_{u, k} \in \mathbb{R}^{n \times k}$ and $\bE_{v,k} \in \mathbb{R}^{d \times k}$ similarly.  

At time $k + 1$, we then take one of two possible updates: a first-order or saddle-point update, defined below. 
\begin{subequations} \label{eq:updates-general}
	\begin{definition}[First-order update]
		\label{def:FO-update}
		Let $f_{k+1}^{u}: \mathbb{R}^{n} \rightarrow \mathbb{R}^{n}$ and $f_{k+1}^{v}: \mathbb{R}^d \rightarrow \mathbb{R}^d$ be $L$-Lipschitz functions.  A (generalized) first-order update at iteration $k + 1$ is given by
		\begin{align} \label{eq:first-order}
			\bu_{k+1}
			&=
			f_{k+1}^u(\bX \bv_k; \bU_k,\bE_{u,k}),
			\qquad \text{ and } \qquad 
			\bv_{k+1}
			=
			f_{k+1}^v
			(-\bX^\top \bu_k; \bV_k,\bE_{v,k}).
		\end{align}
	\end{definition}
	
	\begin{definition}[Saddle update]
		\label{def:saddle-update}
		Let $\phi^{u}_{k+1}: \mathbb{R}^{n} \rightarrow \mathbb{R}$ and $\phi_{k+1}^{v}: \mathbb{R}^d \rightarrow \mathbb{R}$ denote a pair of $\mu$-strongly convex and $L$-smooth functions.  A saddle update is given by 
		\begin{align}
			(\bu_{k+1},\bv_{k+1})
			&=
			\underset{\bv \in \mathbb{R}^p, \bu \in \mathbb{R}^n}{\saddle}
			\Big\{
			\bu^\top \bX \bv
			-
			\phi_{k+1}^u(\bu;\bU_k,\bE_{u,k})
			+
			\phi_{k+1}^v(\bv;\bV_k,\bE_{v,k})
			\Big\}. \label{eq:saddle-update}
		\end{align}
	\end{definition}
\end{subequations}
For brevity below, we often drop the dependence on the auxiliary randomness $(\bE_{u, k}, \bE_{v, k})$ in the definitions of the functions $(f^u_{k+1}, f^v_{k+1}, \phi^u_{k+1}, \phi^v_{k+1})$. Equipped with these definitions, let us now give several examples of iterative methods that fall under the umbrella of these updates. 

\subsection{Illustrative examples} 
To illustrate, we write two first-order methods in the form~\eqref{eq:first-order} and  two saddle methods in the form~\eqref{eq:saddle-update}; note that another example is given by the proximal point update from Section~\ref{sec:intro}. We note that the general form of iterations can also cover updates that interleave saddle and first-order updates~\citep[see, e.g.][]{mondelli2021approximate,celentano2023local}.

\paragraph{Example 1: Gradient methods:} Our first example is gradient descent for $M$-estimation in generalized linear models. Given a smooth (i.e. $C^1$) function $\ell$, we are interested in optimizing a loss $\mathcal{R}(\btheta) = \frac{1}{n} \sum_{i = 1}^n \ell(\inprod{\bx_i}{\btheta}; y_i)$ via gradient descent iterates
\begin{align} \label{eq:GD}
	\btheta_{t + 1} = \btheta_t - \frac{\eta}{n} \bX^\top  \ell'(\bX \btheta_t; \by).
\end{align}
Here we write $\ell'(\bX \btheta_t; \by) \in \reals^n$ to denote a vector whose $i$-th entry is given by $\ell'(\inprod{\bx_i}{\btheta}; y_i)$ and the derivative is taken with respect to the first argument.

To write Eq.~\eqref{eq:GD} as a first-order update according to Definition~\ref{def:FO-update}, we define $(\bu_\ell, \bv_{\ell})_{\ell \geq 0}$ as follows
\begin{align*}
	&\bu_{\ell} = \ell'(\bX \bv_\ell; \by)  \qquad &\text{ and } \qquad  &\bv_{\ell} = \bv_{\ell - 1} \quad &\text{ if $\ell$ odd}, \\
	&\bu_{\ell} = \bu_{\ell-1} \qquad &\text{ and } \qquad &\bv_{\ell} = \bv_{\ell - 1} - \eta \bX^\top  \bu_{\ell - 1} \quad &\text{ if $\ell$ even}.
\end{align*}
We then take $\btheta_k = \bv_{2(k-1)}$ with initialization $\bv_0 = \btheta_1$ and $\bu_0 = 0$. \hfill $\clubsuit$

\paragraph{Example 2: Approximate message passing (AMP):}
The approximate message passing algorithm in standard form~\citep{bayati2011dynamics,javanmard2013state,berthier2020state} for iterates $\bp_k \in \reals^d, \bq_k \in \reals^n$ and nonlinearities $\tau_k: \reals^{n} \to \reals^n$, $e_k: \reals^d \to \reals^d$ is given by the iteration\footnote{Note that we have an additional factor $\Lambda$ appearing in the Onsager terms since the typical AMP iteration is stated using a random matrix having entries of variance $1/n$, whereas our random matrix $\bX$ has entries of variance $1/d$.}
\begin{equation}
	\label{eq:amp}
	\begin{gathered}
		\bp_{k+1} = \bX^\top \tau_k(\bq_k) -\sa_{k} e_{k}(\bp_{k}),
		\qquad
		\bq_{k} = \bX e_k(\bp_k) -  \sb_{k - 1} \tau_{k-1}(\bq_{k-1}),
	\end{gathered}
\end{equation}
where
the scalars $\sa_k = \frac{1}{d} \mathsf{div} [\tau_k] (\bq_k)$ and $\sb_k = \frac{1}{d} \mathsf{div}[e_k](\bp_k)$ give rise to Onsager correction terms. Here $\mathsf{div}[f]$ is the divergence of $f$, given by the sum of its coordinate-wise derivatives. 

To write Eq.~\eqref{eq:amp} as a first-order update according to Definition~\ref{def:FO-update}, we introduce $(\bu_\ell, \bv_{\ell})_{\ell \geq -1}$ via
\begin{align*}
	&\bv_{4k} = e_k(\bv_{4k-1})  \qquad &\text{ and } \qquad \qquad   &\bu_{4k} = \bu_{4k - 1} , \\
	&\bv_{4k+1} = \bv_{4k} \qquad &\text{ and } \qquad \qquad &\bu_{4k+1} = \bX \bv_{4k} - \sb_{k-1}\tau_{k-1}(\bu_{4k}) ,\\
	&\bv_{4k+2} = \bv_{4k+1} \qquad &\text{ and } \qquad \qquad &\bu_{4k+2} =  \tau_{k}(\bu_{4k+1}) , \\
	&\bv_{4k+3} = \bX^{\top} \bu_{4k+2} - \sa_{k}e_k(\bv_{4k+2}) \qquad &\text{ and } \qquad \qquad &\bu_{4k+3} = \bu_{4k+2}.
\end{align*}
We then take $\bp_{k+1} = \bv_{4k+3}$ and $\bq_k = \bu_{4k+1}$ with initialization $\bv_{-1} = \bp_1$ and $\bu_{-1} = 0$. \hfill $\clubsuit$

\paragraph{Example 3: Tikhonov regularized nonlinear least squares for single-index models:} Let $\mathcal{Q}$ denote a distribution over latent variables and suppose we have a known link function $g: \mathbb{R} \times \mathcal{Q} \to \real{R}$. Consider a set of $n$ i.i.d. covariate-response pairs $\{(\bx_i, y_i)\}_{i=1}^{n} \in \mathbb{R}^d \times \mathbb{R}$, generated according to the (random) single-index model $y_i = g( \bx_i^{\top} \btheta_*; q_i) + \sigma \epsilon_i$, with $(\bx_i)_{i=1}^{n} \overset{\mathsf{i.i.d.}}{\sim} \mathsf{N}(0, \bI_d/d)$, $(q_i)_{i=1}^{n} \overset{\mathsf{i.i.d.}}{\sim} \mathcal{Q}$, $\epsilon_i \overset{\mathsf{i.i.d.}}{\sim} \normal(0, 1)$ and with a prior that $\btheta_* \sim \mathsf{N}(0, \bSigma)$ for some positive definite matrix $\bSigma$.  

A canonical family of algorithms to optimize the log-posterior in these problems~\citep{chandrasekher2023sharp} is a weighted least squares update, in which we choose a Tikhonov regularization proportional to $\| \btheta \|^2_{\bSigma^{-1}} = \btheta^{\top} \bSigma^{-1} \btheta$ and some weight function $w:\mathbb{R} \times \mathbb{R} \to \mathbb{R}$ and then execute the iterations
\begin{align} \label{eq:HO-iteration-example}
	\btheta_{k+1} = \argmin_{\btheta \in \mathbb{R}^d}\; \frac{1}{2} \sum_{i=1}^{n} \Bigl(w(\bx_i^\top \btheta_k, y_i) y_i - \bx_i^{\top} \btheta\Bigr)^2 + \lambda \| \btheta \|_{\bSigma^{-1}}^2 = \argmin_{\btheta \in \mathbb{R}^d}\; \frac{1}{2} \bigl \| w(\bX \btheta_k, \by) \odot \by -  \bX \btheta \bigr \|_2^2 + \lambda \| \btheta \|_{\bSigma^{-1}}^2
\end{align}
for some $\lambda > 0$,
where we stack the weights $w(\bx_i^\top \theta, y)$ to form the vector $w(\bX \btheta_k, \by)$ and $\odot$ denotes the Hadamard product.
Examples of such updates are alternating minimization for phase retrieval~\citep{gerchberg1972practical,fienup1982phase} --- in which we take $w(\bx^\top \btheta, y) = \mathsf{sgn}(\bx^\top \btheta)$ --- and expectation maximization (EM) for mixtures of linear regressions~\citep{dempster1977maximum,balakrishnan2017statistical} --- in which we take 
\[
w(\bx^\top \btheta, y) =\frac{ \phi((y - \bx^{\top} \btheta)/\sigma)}{ \phi((y - \bx^{\top} \btheta)/\sigma) +  \phi((y + \bx^{\top} \btheta)/\sigma)},
\]
where $\phi$ denotes the standard Gaussian density.  

The iterations~\eqref{eq:HO-iteration-example} can equivalently be written as
\[
(\bu_{k+1}, \btheta_{k+1}) = \underset{\btheta \in \mathbb{R}^d, \bu \in \mathbb{R}^n}{\saddle}\; \biggl\{ \bu^{\top}\bX \btheta - \bu^{\top} \by \odot w(\bX \btheta_k, \by) - \frac{1}{2}\| \bu \|_2^2 + \lambda \| \btheta \|_{\bSigma^{-1}}^2 \biggr\},
\]
which is an instance of the saddle-point updates~\eqref{eq:saddle-update} with $\bv_k = \btheta_k$ and functions $\phi_{k+1}(\btheta) = \lambda \| \bv \|_2^2$ and $\phi^u_{k+1}(\bu) = \| \bu \|_2^2/2 +  \bu^{\top} \by \odot w(\bX \bv_k, \by)$. Furthermore, notice that the KKT conditions of the saddle problems for each $2 \leq \ell \leq k$ yield that $\bX \bv_\ell = \by \odot w(\bX \btheta_{\ell - 1}, \by) + \bu_k$. Thus $\phi^u_{k+1}(\bu)$ depends on the past iterates only through the matrix $\bU_k$.  \hfill $\clubsuit$

\paragraph{Example 4: Prox-linear for nonlinear least squares problems:}
Suppose that given $n$ i.i.d. covariate response pairs $\{(\bx_i, y_i)\}_{i=1}^{n} \in \mathbb{R}^d \times \mathbb{R}$, we are interested in solving the nonlinear least squares problem
\[
\min_{\btheta \in \mathbb{R}^d}\; \sum_{i=1}^{n} \bigl\{y_i - g(\bx_i^{\top} \btheta) \bigr\}^2,  
\]
where $g: \mathbb{R} \rightarrow \mathbb{R}$ is assumed (for simplicity) to be strictly increasing and differentiable. Such loss functions are of interest in, for instance, nonlinear computed tomography~\citep[see, e.g.,][and the references therein]{fridovich2023gradient}.  A canonical iterative method to perform this minimization is the
prox-linear method~\citep{lewis2016proximal,drusvyatskiy2021nonsmooth}, whose iterates  $\{\btheta_{k}\}_{k \in \mathbb{N}}$ can be written according to the updates 
\[
\btheta_{k+1} = \argmin_{\btheta \in \mathbb{R}^d} \sum_{i=1}^{n}\bigl\{y_i - g(\bx_i^{\top} \btheta_k) - g'(\bx_i^{\top} \btheta_k) \bx_i^{\top} (\btheta - \btheta_k) \bigr\}^2 + \frac{1}{2\alpha_k} \| \btheta - \btheta_k \|_2^2.
\]
Note that taking the stepsize $\alpha_k \rightarrow \infty$, this reduces to the familiar  Gauss--Newton method~\citep[see, e.g.,][Chapter 10]{nocedal2006numerical}.  
Write the residuals as $r_i := y_i - g(\bx_i^{\top} \btheta_k)$ and collect these into the vector $\br_k = [r_1, \ldots, r_n]^{\top}$. Also define the diagonal matrix $\bD_k = \diag(g'(\bx_1^{\top} \btheta_k), \ldots, g'(\bx_n^{\top} \btheta_k))$ and make the change of variables $\bv = \btheta - \btheta_k$. Then we can more concisely write each iteration as
\[
\bv_{k+1} = \argmin_{\bv \in \mathbb{R}^d}\; \bigl \| \br_k - \bD_k \bX \bv \bigr\|_2^2 + \frac{1}{2 \alpha_k} \| \bv \|_2^2 = \argmin_{\bv \in \mathbb{R}^d} \sup_{\bz \in \mathbb{R}^n}\; \bz^{\top} \bD_k \bX \bv - \bz^{\top} \br_k - \frac{1}{2} \| \bz \|_2^2 + \frac{1}{2\alpha_k} \| \bv_k \|_2^2,  
\]
where in the second relation we used the fact that the Fenchel--Legendre conjugate of $\bx \mapsto \| \bx \|_2^2/2$ is itself.  Making the substitution $\bu = \bD_k \bz$ we thus have
\[
(\bu_{k+1}, \bv_{k+1}) = \underset{\bv \in \mathbb{R}^d, \bu \in \mathbb{R}^n}{\saddle}\; \biggl\{ \bu^{\top} \bX \bv - \bu^{\top} \bD_k^{-1} \br_k - \frac{1}{2} \bigl \| \bD_k^{-1} \bu \bigr \|_2^2 + \frac{1}{2\alpha_k} \| \bv_k \|_2^2 \biggr\},
\]
where we note that since $g$ is strictly increasing, the inverse $\bD_k^{-1}$ always exists.  \hfill $\clubsuit$

\section{Main results} \label{sec:main-results}

Having given several examples, we now state our main results for the general iterations that we set up in Section~\ref{sec:setup}. First, in Theorem~\ref{lem:fix-pt-exist-unique}, we provide a state evolution recursion that covers these methods, and show rigorously the existence and uniqueness of this state evolution. Second, in Theorem~\ref{thm:exact-asymptotics}, we establish finite-sample guarantees showing that low-dimensional functions of the empirical iterates~\eqref{eq:initialization}-\eqref{eq:updates-general} are predicted by deterministic functionals of the state evolution.

\subsection{State evolution}
Recall that $\Lambda = n/d$.  The state evolution specifies two joint distributions on different probability spaces, indexed by the iteration $k \in \mathbb{N}$.  The first is a joint distribution over random matrices $\Use_k \in \mathbb{R}^{n \times k}, \Uhatse_k \in \mathbb{R}^{n \times k}, \Hse_k \in \mathbb{R}^{n \times k}$ while the second is a joint distribution over random matrices $\Vse_k \in \mathbb{R}^{d \times k}, \Vhatse_k \in \mathbb{R}^{d \times k}, \Gse_k \in \mathbb{R}^{d \times k}$.  When $k=1$, these satisfy 
\begin{align}
	\label{eq:SE-base-case}
	\begin{gathered}
		\vse_1 =  \bv_{1}, \qquad \use_1 =  \bu_{1}, \qquad \uhatse_1 = 0, \qquad \vhatse_1 = 0,  \\
		\gse_1 \sim \normal(\bzero, \| \bu_1 \|_2^2 \cdot \id_d /d) \qquad \hse_1 \sim \normal(\bzero, \| \bv_1 \|_2^2 \cdot \id_n /n). 
	\end{gathered}
\end{align}
For further indices $k \in \mathbb{N}$ with $k > 1$, they satisfy
\begin{equation}
	\label{eq:se-guass-and-span}
	\begin{gathered}
		\gse_1, \ldots, \gse_k \sim \normal(\bzero, \bK^g_k \otimes \id_d / d),
		\quad \text{ and } \quad 
		\hse_1, \ldots, \hse_k \sim \normal(\bzero,\bK^h_k \otimes \id_n / n),
		\\
		\uhatse_k
		\in \spn\bigl( \{\use_{\ell} \}_{1 \leq \ell \leq k} \bigr), \quad \text{ and } \quad 
		\vhatse_k
		\in \spn\bigl( \{\vse_{\ell} \}_{1 \leq \ell \leq k} \bigr).
	\end{gathered}
\end{equation}
For first order updates at iteration $k$, they additionally satisfy 
\begin{equation}
	\label{eq:se-first-order}
	\use_k
	=
	f_{k}^u\bigl(\sqrt{\Lambda}\hse_{k-1} - \uhatse_{k-1};\Use_{k-1}\bigr),
	\quad \text{ and } \quad 
	\vse_k
	=
	f_{k}^v\bigl(\gse_{k-1}- \vhatse_{k-1}; \Vse_{k-1}\bigr),
\end{equation}
where we recall that $f_{k}^u, f_k^v$ were the Lipschitz update functions specified in Definition~\ref{def:FO-update}.  On the other hand, for saddle updates at iteration $k$, they satisfy
\begin{equation}
	\label{eq:se-saddle}
	\sqrt{\Lambda} \hse_k - \uhatse_k
	=
	\nabla \phi_k^u(\use_k; \Use_{k-1}), \quad \text{ and } \quad 
	\gse_k -\vhatse_{k}
	=
	\nabla \phi_k^v(\vse_k; \Vse_{k-1}).
\end{equation}
These random variables additionally satisfy the expectation relations
\begin{equation}
	\label{eq:fix-pt}
	\begin{aligned}
		\bK^g_k = \llangle \Gse_k, \Gse_k \rrangle_{L^2} &= \llangle \Use_k, \Use_k \rrangle_{L^2},
		\qquad&
		\bK^h_T = \llangle \Hse_k, \Hse_k \rrangle_{L^2} &= \llangle \Vse_k, \Vse_k \rrangle_{L^2},
		\\
		\llangle \Gse_k, \Vse_k \rrangle_{L^2} &= \llangle \Use_k , \Uhatse_k \rrangle_{L^2},
		\qquad&
		\llangle \sqrt{\Lambda} \Hse_k, \Use_k \rrangle_{L^2} &= \llangle \Vse_k, \Vhatse_k \rrangle_{L^2}.
	\end{aligned}
\end{equation}
Note that above, we abuse notation and write $L^2$ to mean one of two quantities: either 
\[
L^2\Bigl(\underbrace{\mathbb{R}^d \times \cdots \mathbb{R}^d}_{k \text{ times}}, \mathcal{B}\bigl( \underbrace{\mathbb{R}^d \times \cdots \mathbb{R}^d}_{k \text{ times}}\bigr), \mathsf{N}(0, \bI_d)^{\otimes k}\Bigr) \quad \text{ or } \quad L^2\Bigl(\underbrace{\mathbb{R}^n \times \cdots \mathbb{R}^n}_{k \text{ times}}, \mathcal{B}\bigl( \underbrace{\mathbb{R}^n \times \cdots \mathbb{R}^n}_{k \text{ times}}\bigr), \mathsf{N}(0, \bI_n)^{\otimes k}\Bigr).
\]
Note that since the functions $\{f_k^u, f_k^v\}_{k \in \mathbb{N}}$ and $\{ \nabla \phi^u_k, \nabla \phi^v_k\}_{k \in \mathbb{N}}$ are Lipschitz, all of the random vectors in the previous displays are elements of the appropriate $L^2$ spaces described above.  

Our first main result proves that the equations defining the state evolution have a unique solution.   
\begin{theorem}
	\label{lem:fix-pt-exist-unique}
	For each $k \in \mathbb{N}$, there is a unique pair of joint distributions $\Law(\Gse_k, \Vse_k, \Vhatse_k)$ and $\Law(\Hse_k, \Use_k, \Uhatse_k)$ that satisfy Eqs.~\eqref{eq:SE-base-case}, \eqref{eq:se-guass-and-span}, \eqref{eq:se-first-order}, \eqref{eq:se-saddle}, and \eqref{eq:fix-pt}.
\end{theorem}
We provide the proof of Theorem~\ref{lem:fix-pt-exist-unique} in Section~\ref{sec:exist-unique}.  

Our parameterization~\eqref{eq:SE-base-case}--\eqref{eq:fix-pt} of the state evolution may seem nonstandard, but as revealed in the proof of Theorem~\ref{lem:fix-pt-exist-unique}, it is directly useful in establishing existence and uniqueness guarantees. To illustrate the parameterization, we show that it recovers two existing state evolutions in the literature.    

Our first example consists of linear estimation with AMP~\citet[Section 2.1]{bayati2011dynamics} in a pure noise problem, in which the dynamics are given by
\[
\bp_{k+1} = \lambda_{k} \cdot \bigl(\bX^{\top} \bq_k + \Lambda \bp_k\bigr), \qquad \text{ and } \qquad \bq_{k} = \by - \bX \bp_k + \lambda_{k-1} \bq_{k-1},
\]
where $\lambda_k$ is a user-defined parameter and $\by \sim \mathsf{N}(0, \sigma^2 \bI_n)$.  For this simple iteration, the reduction to our first-order methods in Definition~\ref{def:FO-update} is simpler than that in Example 2.  Indeed, we take 
\begin{align}
	&\bv_{2k} = \bv_{2k-1} \qquad &\text{ and } \qquad \qquad &\bu_{2k} =  \by - \bX \bv_{2k-1} + \lambda_{k-1} \bu_{2k-1} , \nonumber\\
	&\bv_{2k+1} = \lambda_k \cdot \bigl(\bX^{\top} \bu_{2k} + \Lambda \bv_{2k}\bigr) \qquad &\text{ and } \qquad \qquad &\bu_{2k+1} = \bu_{2k}. \label{eq:AMP-linear-expanded}
\end{align}
In Appendix~\ref{sec:SE-AMP-example}, we show that the state evolution in Eqs.~\eqref{eq:SE-base-case}--\eqref{eq:fix-pt} recovers~\citet[Eq. (2.5)]{bayati2011dynamics} with $v$ in that equation set to $0$.  In particular, we show that $\| \use_{2k}\|_{L^2}^2 = \tau_k^2$ where $\tau_k$ satisfies the recursion
\[
\tau_{k}^2 = \sigma^2 + \Lambda \lambda_{k-1}^2 \tau_{k-1}^2.
\]

For our second example, we consider the asymptotic description of the behavior of $M$-estimators as in~\citep{thrampoulidis2018precise,bellec2023existence}, where one is interested in understanding asymptotic properties of
\[
\bv^{\mathrm{M}} = \argmin_{\bv \in \mathbb{R}^d}\; L(\bv) := \frac{1}{d} \rho\bigl(\by - \bX \bv\bigr) + \frac{\lambda}{d} f(\bv),
\]
where $\rho$ is a data-fitting term and $f$ is understood as a penalty.  To simplify the discussion, we suppose that $\by = 0$.  In Appendix~\ref{sec:appendix-m-est-gordon}, we show that the asymptotic properties of this problem are characterized by the fixed point equations
\begin{subequations} \label{eq:BK-SE}
	\begin{align}
		\alpha^2 &= \mathbb{E}\Bigl[\bigl\{\mathsf{prox}_{\lambda \nu^{-1} f}\bigl(\nu^{-1} \beta Z\bigr)\bigr\}^2\Bigr],\label{eq:example-gordon-norm-v}\\
		\Lambda^{-1} \beta^2 \kappa^2 &= \mathbb{E}\Bigl[\bigl\{\alpha W - \mathsf{prox}_{\kappa \rho}(\alpha W)\bigr\}^2\Bigr],\label{eq:example-gordon-norm-u}\\
		\delta^{-1} \nu \alpha \kappa &= \mathbb{E}\Bigl[W \cdot \bigl\{ \alpha W - \mathsf{prox}_{\kappa \rho}(\alpha W)\bigr\}\Bigr], \quad \text{and }\label{eq:example-gordon-uh}\\
		\kappa \beta &= \mathbb{E}\Bigl[Z \cdot \mathsf{prox}_{\lambda \nu^{-1} f}(\nu^{-1} \beta Z)\Bigr], \label{eq:example-gordon-vg},
	\end{align}
\end{subequations}
where $Z, W \sim \mathsf{N}(0,1)$ are independent of all other randomness. These equations coincide with the set of fixed point equations in~\citet[Eq. (1.11)]{bellec2023existence}.  Note that taking one saddle step in our state evolution and applying Theorem~\ref{lem:fix-pt-exist-unique} immediately implies the existence and uniqueness of a solution to the fixed point equations in~\eqref{eq:BK-SE}, hence recovering (under stronger assumptions) the guarantee of~\citet[Theorem 3.2, (2)]{bellec2023existence}.   

Let us remark briefly on our technique for proving Theorem~\ref{lem:fix-pt-exist-unique}.  Our proof is inspired by the variational approach of~\citet{montanari2025generalization}, which obtains existence and uniqueness of a solution to a system of nonlinear equations obtained post-scalarization of the auxiliary optimization generated by the convex Gaussian minmax theorem (CGMT). In particular, they construct an appropriate infinite-dimensional optimization problem whose KKT conditions coincide with the fixed point equations.  Our proof follows this general recipe inductively, but without explicitly scalarizing the auxiliary objective; note that this is important in our setting since the functions defining our saddle update are not necessarily separable. At each iteration, when the update is a saddle update, we construct an appropriate variational problem and show that its KKT conditions coincide with the state evolution.  On the other hand, if the update is first-order, we argue the existence and uniqueness of the state evolution from first principles.

\subsection{Finite-sample guarantees} 

We now show that the state evolution~\eqref{eq:SE-base-case}--\eqref{eq:fix-pt} is operationally useful, in that it can make predictions about the iterates $(\bu_k, \bv_k)_{ \geq 1}$.
We begin with some notation.  Since $\widehat{\bv}_k^{\mathrm{SE}} \in \spn( \{\vse_{k'} \}_{k' \leq k} )$ and $\widehat{\bu}_k^{\mathrm{SE}} \in \spn(\{\use_{k'} \}_{k' \leq k} )$, there must exist matrices $\bL_k^{v} \in \mathbb{R}^{k \times k}$ and $\bL_k^{u} \in \mathbb{R}^{k \times k}$ such that
\[
\widehat{\bv}_k^{\mathrm{SE}} = \sum_{\ell = 1}^{k} \bigl(L^{v}_k\bigr)_{k, \ell} \bv_{\ell}^{\mathrm{SE}} \quad \text{and} \quad \widehat{\bu}_k^{\mathrm{SE}} = \sum_{\ell = 1}^{k} \bigl(L^{u}_k\bigr)_{k, \ell} \bu_{\ell}^{\mathrm{SE}}.
\]
Using these coefficients, we further define, for $\ell \in [k]$, the vectors
\begin{align}
	\label{eq:PO-g-h-def}
	\bg_{\ell} :=  \sum_{j = 1}^{\ell} \bigl(L^{v}_\ell\bigr)_{j, \ell} \bv_{j} - \bX^{\top} \bu_{\ell} \quad \text{and} \quad \bh_{\ell} :=  \Lambda^{-1/2} \Big(\sum_{j = 1}^{\ell} \bigl(L^{u}_{\ell}\bigr)_{\ell, j}\bu_{j} + \bX \bv_{\ell} \Big),
\end{align}
and collect these into the matrices $\bG_k = [\bg_1 \; \vert \; \bg_2 \; \vert \; \cdots \; \vert \; \bg_k] \in \mathbb{R}^{d \times k}$ and $\bH_k= [\bh_1 \; \vert \; \bh_2 \; \vert \; \cdots \; \vert \; \bh_k] \in \mathbb{R}^{n \times k}$.
Before stating the theorem, we require the following regularity assumption.
\begin{assumption}\label{asm:regularity}
	There exist universal constants $(\mu, L, K)$ such that the following is true. For each $k \in \mathbb{N}$, the functions $\phi_{k}^v, \phi_k^{u}$ are $\mu$-strongly convex and $L$-smooth, and the functions $f^u_k, f^v_k$ are $L$-Lipschitz.  Moreover, we have $\| \nabla \phi_{k}^{u}(0) \|_2 \vee \| \nabla \phi_{k}^{v}(0) \|_2 \vee \| f^u_k(0) \|_2 \vee \| f^v_k(0)\|_2 \leq K$.
\end{assumption}

\begin{assumption}\label{asm:non-degeneracy}
	The matrices $\bK^g_k, \bK^h_k$ in Eq.~\eqref{eq:se-guass-and-span} are invertible for each $k$.
\end{assumption}
We note that Assumption~\ref{asm:non-degeneracy} is similar to the non-degeneracy assumption in~\citet[Section 5.2]{berthier2020state} and can be removed using similar techniques therein or by using approximate Cholesky decompositions as in~\citet{celentano2023challenges}.  We elect to state our theorem under this simplifying assumption as it allows to more easily expose the ideas underlying the proof.  
\begin{theorem}
	\label{thm:exact-asymptotics}
	Let $T \in \mathbb{N}$, suppose that Assumptions~\ref{asm:regularity} and~\ref{asm:non-degeneracy} hold, and that, for each, $k \in [T]$, the pair $(\bu_k, \bv_k)$ is obtained using either a first order update as in Definition~\ref{def:FO-update} or a saddle update as in Definition~\ref{def:saddle-update}.  Further, let $\bU_T\in \mathbb{R}^{n \times T}$ and $\bV_T \in \mathbb{R}^{d \times T}$ be defined as 
	\[
	\bU_T = [\bu_1 \; \vert \; \bu_2 \; \vert \; \cdots \; \vert \; \bu_T] \qquad \text{ and } \qquad \bV_T = [\bv_1 \; \vert \; \bv_2 \; \vert \; \cdots \; \vert \; \bv_T].
	\]
	Let $\psi_d: \mathbb{R}^{d \times T} \times \mathbb{R}^{d \times T}$ and $\psi_n: \mathbb{R}^{n \times T} \times \mathbb{R}^{n \times T}$ be order-$2$ pseudo-Lipschitz functions.  There exists a pair of constants $c_{\mathrm{SE}}, C_{\mathrm{SE}} > 0$, depending only on the state evolution and the tuple $(\mu, L, K)$ such that if $\delta \in (e^{-c_{\mathrm{SE}}n}, 1)$, then with probability at least $1 - \delta$, we have
	\begin{align} \label{eq:finite-sample}
		\Bigl \lvert \psi_d(\bV_T, \bG_T) - \EE\bigl[\psi_d(\Vse_T, \Gse_T)\bigr] \Bigr \rvert \; \vee \; \Bigl \lvert \psi_n(\bU_T, \bH_T) - \EE\bigl[\psi_n(\Use_T, \Hse_T)\bigr] \Bigr \rvert \leq C_{\mathrm{SE}} \bigl(T! \bigr)^2 \Bigl(\frac{\log(10T^2/\delta)}{n}\Bigr)^{2^{-T}}.
	\end{align}
\end{theorem}
We provide the proof of this theorem in Section~\ref{sec:proof-exact-asymptotics}. The key takeaway is that empirical quantities (e.g. $\psi_d(\bV_T, \bG_T)$) are accurately predicted by deterministic quantities (e.g. $ \EE\bigl[\psi_d(\Vse_T, \Gse_T)\bigr]$) that are computable as simple expectations of Gaussian functionals. In particular, these are themselves low-dimensional integrals provided the test and iteration functions are ``simple'' (e.g. separable). In that sense, the state evolution provides predictions of the behavior of the random iterates, and Theorem~\ref{thm:exact-asymptotics} additionally gives non-asymptotic fluctuation guarantees showing consistency in $n$. A few further remarks are in order.  

First, a careful inspection of the proof reveals that the RHS of the bound~\eqref{eq:finite-sample} can be improved to 
$C_{\mathrm{SE}} \bigl(T! \bigr)^2 \Bigl(\frac{\log(10T^2/\delta)}{n}\Bigr)^{2^{-(\tau \vee 1)}}$, where $\tau$ is the number of saddle updates. In particular,
if every update is first order, we obtain the sharper bound
\begin{align} \label{eq:FO-deviation}
	\Bigl \lvert \psi_d(\bV_T, \bG_T) - \EE\bigl[\psi_d(\Vse_T, \Gse_T)\bigr] \Bigr \rvert \; \vee \; \Bigl \lvert \psi_n(\bU_T, \bH_T) - \EE\bigl[\psi_n(\Use_T, \Hse_T)\bigr] \Bigr \rvert \leq C_{\mathrm{SE}} \bigl(T! \bigr)^2 \Bigl(\frac{\log(10T^2/\delta)}{n}\Bigr)^{1/2},
\end{align}
Since, by Stirling's inequality, $T! \leq e^2(T/e)^{T + 1}$, this implies that (ignoring logarithmic factors) we have a deviation bound scaling as $C_{\mathrm{SE}} T^{2T + 2} n^{-1/2}$.  If the constant $C_{\mathrm{SE}}$ is a universal constant (e.g. independent of $T, n, d$), then this bound implies that the state evolution is valid for $T = o(\log{n}/\log\log{n})$ many iterations.  Results of a similar flavor were shown by~\citet{rush2018finite} for AMP (see also~\citet{cademartori2024non}), and GFOMs (see~\cite{han2024entrywise}, who also provides universality guarantees). Our proof techniques, however, are distinct from these papers and arguably easier to generalize to other situations.

We remark that, in general, the appearance of $(T!)^2$ in the bound~\eqref{eq:FO-deviation} should not be thought of as very loose.  Indeed, the quantity $C_{\mathrm{SE}}$ itself can depend on $T$ through properties of the state evolution. Thus, at the level of generality of the first-order update~\eqref{eq:first-order}, one should not expect a tighter deviation bound than one scaling as $C^{T} n^{-1/2}$ and hence the state evolution can only be accurate for $o(\log n)$ iterations. To see this, set some constant $C > 1$ and consider the explosive iteration 
\begin{align*}
	\bu_{k + 1} = \sqrt{\frac{C}{\Lambda}} \bX \bv_{k} \quad \text{ and } \quad \bv_{k + 1} = \bX^\top \bu_{k},
\end{align*}
which is a first-order update of the form~\eqref{eq:first-order}.
With any fixed, unit-norm initializations $(\bu_1, \bv_1)$, it can be verified that $\| \bu_{T + 1} \|_2^2 = C^T \mathsf{trace} \Big((\frac{\bX^\top \bX}{\Lambda})^T \Big)$. It can also be verified that $\EE[\| \use_{T +1} \|_2^2] = \int (Cx)^T \nu(dx)$, where $\nu$ is the Marchenko--Pastur measure. By standard results in random matrix theory~\citep{bai2010spectral,couillet2022random}, we should expect the deviation to scale as
\[
\big| \| \bu_{T + 1} \|_2^2 - \EE[\| \use_{T + 1} \|_2^2] \big| = \left| C^T \mathsf{trace}\Big( (\bX^\top \bX) / \Lambda )^T \Big) - \int (Cx)^T \nu(dx) \right|  \asymp C^T \Big( \frac{1}{n}\Big)^{1/2}
\]
with constant probability. 

While the argument above shows the optimality of the $n^{-1/2}$ rate for first-order methods in their most general form, sharper concentration results may be obtainable by exploiting specific properties of the state evolution in particular methods (see, e.g.,~\cite{li2022non,li2024non}). We also note that for saddle updates, the exponent $2^{-T}$ in Eq.~\eqref{eq:finite-sample} is an artifact of our analysis. This arises because of a sequential application of Gordon's Gaussian comparison inequality, and each application typically yields deviation bounds of a suboptimal order~\citep[see, e.g.,][for further discussion]{chandrasekher2023sharp,miolane2021distribution}. 

As previously noted, our proof technique for establishing Theorem~\ref{thm:exact-asymptotics} may be of independent interest. We introduce a sequential variant of the convex Gaussian minmax theorem (CGMT) to accommodate successive saddle updates, and apply it in conjunction with Bolthausen's conditioning technique~\citep{bolthausen2014iterative}. To establish finite-sample concentration results, we leverage certain stability properties of the established state evolution in conjunction with a concise proof of sub-exponential concentration for pseudo-Lipschitz functions (see Lemma~\ref{lem:pseudo-Lipschitz-concentration}).

We now proceed to proofs of the two main theorems.


\section{Proof of Theorem~\ref{lem:fix-pt-exist-unique}: State evolution has unique solution}
\label{sec:exist-unique}

Since the state evolution at time $k+1$ depends on its definition at time $k$, our proof will proceed via induction. In Section~\ref{sec:SE-exists-basecase}, we present the base case $k = 1$. The induction hypothesis is set up formally in Section~\ref{sec:SE-exists-indhyp}. The bulk of the proof is in the induction step, which is presented in Section~\ref{sec:indstep-FO} for a first-order update and Section~\ref{sec:indstep-saddle} for a saddle update.

\subsection{Base case: $k = 1$} \label{sec:SE-exists-basecase} The pair $(\bu_1, \bv_1)$ is fixed independently of the randomness in the data $\bX$, and the tuple $(\gse_1, \hse_1, \uhatse_1, \vhatse_1)$ is chosen---in particular, uniquely---based on this pair. In particular, since both $\uhatse_1$ and $\vhatse_1$ are equal to zero, we have the trivial inclusions $\uhatse_1 \in \spn\{\use_1 \}$ and $\uhatse_1 \in \spn\{\use_1 \}$. Furthermore, the random vectors $\gse_1$ and $\hse_1$ are drawn from standard Gaussians scaled by $\| \bu_1 \|_2 / \sqrt{d}$ and $\| \bv_1 \|_2 / \sqrt{n}$, respectively, so that we have the scalar relations $\bK^g_1 = \| \bu_1 \|^2_2$ and $\bK^h_1 = \| \bv_1 \|^2_2$. It remains to verify the fixed point equations~\eqref{eq:fix-pt}, but this is also immediate since
\begin{align*}
	&\inprod{\gse_1}{\gse_1}_{L^2} = \| \bu_1 \|_2^2 = \inprod{\use_1}{\use_1}_{L^2}, \qquad  &\inprod{\hse_1}{\hse_1}_{L^2} = \| \bv_1 \|_2^2 = \inprod{\vse_1}{\vse_1}_{L^2} \\
	&\inprod{\gse_1}{\vse_1}_{L^2} = \inprod{\use_1}{\uhatse_1}_{L^2} = 0, \qquad  &\inprod{\hse_1}{\use_1}_{L^2} = \inprod{\vse_1}{\vhatse_1}_{L^2} = 0,
\end{align*}
which completes the base case.  We turn now to stating the induction hypothesis.

\subsection{Induction hypothesis} \label{sec:SE-exists-indhyp}
Suppose that the fixed-point equations up to iteration $k$ are uniquely defined. In other words, we have that for covariance matrices $\bK^g_{k} \in \mathbb{R}^{k \times k}$ and $\bK^h_{k} \in \mathbb{R}^{k\times k}$ and Gaussian random matrices $\Gse_k \sim \mathsf{N}(0, \bK^g_{k} \otimes \bI_d/d)$ and $\Hse_k \sim \mathsf{N}(0, \bK^h_{k} \otimes \bI_n/n)$, the system of equations
\begin{equation}
	\label{eq:fix-pt-up-to-t}
	\begin{aligned}
		\bK^g_{k} = \llangle \Gse_k, \Gse_k \rrangle_{L^2} &\overset{(a)}{=} \llangle \Use_k, \Use_k \rrangle_{L^2},
		\qquad&
		\bK^h_{k} = \llangle \Hse_k, \Hse_k \rrangle_{L^2} &\overset{(b)}{=} \llangle \Vse_k, \Vse_k \rrangle_{L^2},
		\\
		\llangle \Gse_k, \Vse_k \rrangle_{L^2}, &\overset{(c)}{=} \llangle \Use_k, \Uhatse_k \rrangle_{L^2},
		\qquad&
		\llangle \sqrt{\Lambda} \Hse_k, \Use_k \rrangle_{L^2} &\overset{(d)}{=} \llangle \Vse_k, \Vhatse_k \rrangle_{L^2},
	\end{aligned}
\end{equation}
has a unique solution. Specifically, we have
\begin{align*}
	\vse_ 1 = \bv_1, \qquad \use_1 = \bu_1, \qquad \vhatse_1 = 0, \qquad \uhatse_1= 0, \\
	\gse_1 \sim \normal(\bzero, \| \bu_1 \|_2^2 \cdot \id_d/d) \qquad \hse_1 \sim \normal(\bzero, \| \bv \|_2^2 \cdot \id_n/n),
\end{align*}
and for all $2 \leq \ell \leq k$, we have 
\begin{align*}
	\hbv^{\SE}_\ell \in \spn( \{ \vse_{\ell'} \}_{\ell' \leq \ell}), \qquad &\uhatse_{\ell} \in \spn( \{ \use_{\ell'} \}_{\ell' \leq \ell}),
\end{align*}
and
\begin{align*}
	\use_{\ell} = f^u_{\ell} \bigl(\sqrt{\Lambda} \hse_{\ell - 1} - \uhatse_{\ell - 1}, \Use_{\ell-1}\bigr), 
	\qquad &\vse_{\ell} = f^v_{\ell} \bigl(\gse_{\ell - 1} - \vhatse_{\ell - 1}, \Vse_{\ell-1}\bigr)  &\text{ if first-order update at time } \ell,\\
	\sqrt{\Lambda} \hse_{\ell} - \uhatse_{\ell}
	=
	\nabla \phi_{\ell}^u\bigl(\use_{\ell}; \Use_{\ell - 1}\bigr)
	,
	\qquad 
	&\gse_{\ell}-\vhatse_{\ell}
	=
	\nabla \phi_{\ell}^v\bigl(\vse_{\ell}; \Vse_{\ell - 1}\bigr) \; \; \; &\text{ if saddle update at time } \ell.
\end{align*}
Furthermore, the joint laws of the matrices $(\Gse_k, \Vse_k, \Vhatse_k)$ and $(\Hse_k, \Use_k, \Uhatse_k)$ are uniquely defined.

\subsection{Proof of induction step for first-order updates} \label{sec:indstep-FO}
Our goal is to establish that the state evolution has a unique solution at time $k+1$ if we take a first-order update.  We first form the upper left $k \times k$ blocks of the covariance matrices $\bK^g_{k+1} \in \mathbb{R}^{(k+1) \times (k+1)}$ and $\bK^h_{k+1} \in \mathbb{R}^{(k+1) \times (k+1)}$ by setting them equal to $\bK^g_{k}$ and $\bK^h_{k}$, respectively. We must show that for some $d$-dimensional random vector $\gse_{k+1}$ with i.i.d. entries that is jointly Gaussian with $\Gse_k$ and $n$-dimensional random vector $\hse_{k+1}$ with i.i.d. entries that is jointly Gaussian with $\Hse_{k}$, we can construct $d$-dimensional random vectors $\vse_{k+1}$ and $\vhatse_{k+1}$ and $n$-dimensional random vectors $\use_{k+1}$ and $\uhatse_{k+1}$ satisfying 

\begin{itemize}
	\item The following system of equations for all $\ell = k+1, \ell' \leq k+1$ or  $\ell' = k+1, \ell\leq k+1$:
	\begin{equation}
		\label{eq:fix-pt-FO-at-t}
		\begin{aligned}
			\E[ \inprod{\gse_{\ell}}{\gse_{\ell'}}] &\overset{(a)}{=} \E[\inprod{\use_{\ell}}{\use_{\ell'}}],
			\qquad&
			\E[ \inprod{\hse_{\ell}}{\hse_{\ell'}}] &\overset{(b)}{=} \E[\inprod{\vse_{\ell}}{\vse_{\ell'}}] 
			\\
			\E[ \inprod{\gse_{\ell}}{\vse_{\ell'}}] &\overset{(c)}{=} \E[\inprod{\use_{\ell}}{\uhatse_{\ell'}}], 
			\qquad&
			\E[ \sqrt{\Lambda} \inprod{\hse_{\ell}}{\use_{\ell'}}] &\overset{(d)}{=} \E[\inprod{\vse_{\ell}}{\vhatse_{\ell'}}].
		\end{aligned}
	\end{equation}

	\item The inclusions
	\begin{align} \label{eq:span-inclusions-induction-step-FO}
		\vhatse_{k+1} \in \spn( \{ \vse_{\ell} \}_{\ell \leq k+1}), \qquad &\uhatse_{k+1} \in \spn( \{ \bu^{\SE}_{\ell} \}_{s \leq t}).
	\end{align}
	
	\item The relations
	\begin{align} \label{eq:FO-cond-induction-step}
		\use_{k+1} = f^u_{k+1} \bigl(\sqrt{\Lambda}\hse_{k} - \uhatse_{k}; \Use_{k}\bigr), 
		\qquad &\vse_{k+1} = f^v_{k+1} \bigl(\gse_{k} - \vhatse_{k}; \Vse_k\bigr). 
	\end{align}
\end{itemize}
Furthermore, we would like the joint (conditional) laws of 
\[
(\gse_{k+1}, \vse_{k+1}, \vhatse_{k+1}) \mid \Gse_k, \Vse_k, \Vhatse_k \quad \text{ and } \quad (\hse_{k+1}, \use_{k+1}, \uhatse_{k+1}) \mid \Hse_k, \Use_k, \Uhatse_k
\]
to be uniquely defined.

We now construct the random vectors satisfying the above desiderata.  To this end, let $\Gse_k \sim \normal(0,\bK^g_{k}\otimes \id_d/d)$, $\Hse_k \sim \normal(0,\bK^h_{k}\otimes \id_n/n)$ and let $\Vse_k, \Vhatse_k, \Use_k, \Uhatse_k$ be defined inductively via~\eqref{eq:se-first-order} and~\eqref{eq:se-saddle}.  From these, we explicitly construct the tuples $(\gse_{k+1}, \vse_{k+1}, \vhatse_{k+1})$ and $(\hse_{k+1}, \use_{k+1}, \uhatse_{k+1})$ in a unique fashion, thereby ensuring that the joint laws of $(\Gse_{k+1}, \Vse_{k+1}, \Vhatse_{k+1})$ and $(\Hse_{k+1}, \Use_{k+1}, \Uhatse_{k+1})$ are unique. 

Note that, given the past, Eq.~\eqref{eq:se-first-order} uniquely specifies 
\[
\use_{k+1} = f_{k+1}^{u} \bigl(\sqrt{\Lambda} \hse_{k} - \uhatse_{k}; \Use_{k}\bigr) \quad \text{ and } \quad \vse_{k+1} = f_{k+1}^{v}\bigl(\gse_{k} - \vhatse_{k}; \Vse_k \bigr),
\]
by definition, so that Eq.~\eqref{eq:FO-cond-induction-step} is satisfied.  It follows that the matrices $\Use_{k+1}$ and $\Vse_{k+1}$ are uniquely specified.  We thus use these to (uniquely) define 
\[
\bK^{g}_{k+1} = \llangle \Use_{k+1}, \Use_{k+1} \rrangle_{L^2} \quad \text{ and } \quad \bK^h_{k+1} =  \llangle \Vse_{k+1}, \Vse_{k+1} \rrangle_{L^2}.  
\]
By joint Gaussianity, we can write $\gse_{k+1} = \alpha_{k+1} \bxi_g + \widetilde{\bg}_{k}^{\mathrm{SE}}$, where $\bxi_g \sim \mathsf{N}(0, \id_d/d)$, independently of everything else and $\widetilde{\bg}_k^{\mathrm{SE}} \in \spn(\{\gse_{\ell}\}_{1 \leq \ell \leq k})$ so that $\llangle \Gse_{k+1}, \Gse_{k+1} \rrangle_{L^2} = \bK^g_{k+1}$ as well as $\hse_{k+1} = \beta_{k+1} \bxi_h + \widetilde{\bh}_{k}^{\mathrm{SE}}$, where $\bxi_h \sim \mathsf{N}(0, \id_n/n)$, independently of everything else and $\widetilde{\bh}_k^{\mathrm{SE}} \in \spn(\{\hse_{\ell}\}_{1 \leq \ell \leq k})$, so that $\llangle \Hse_{k+1}, \Hse_{k+1} \rrangle_{L^2} = \bK^h_{k+1}$.  By construction, relations $(a)$ and $(b)$ in Eq.~\eqref{eq:fix-pt-FO-at-t} are satisfied.  Note further that, by the fundamental theorem of linear algebra, there exists a unique $\uhatse_{k+1} \in \spn(\{\use_{\ell}\}_{1 \leq \ell \leq k+1})$ such that 
\begin{equation} \label{eq:uhat-fix-FO}
	\< \uhatse_{k+1} , \use_{\ell} \>_{L^2} = \< \vse_{k+1} , \gse_{\ell} \>_{L^2}, \quad \text{ for all } \quad \ell \in [k+1],
\end{equation}
and similarly for $\vhatse_{k+1}$.  This shows that the relation~\eqref{eq:span-inclusions-induction-step-FO} holds for the uniquely defined $\uhatse_{k+1}$ and $\vhatse_{k+1}$ and that relations $(c)$ and $(d)$ of Eq.~\eqref{eq:fix-pt-FO-at-t} hold by construction. \hfill \qed

\subsection{Proof of induction step for saddle updates} \label{sec:indstep-saddle}
The saddle update is quite a bit more involved.  We begin by outlining the desiderata.  Similar to the first order step, our goal is to establish that the state evolution has a unique solution at time $k+1$ if we take a saddle update.  We first form the upper left $k \times k$ blocks of the covariance matrices $\bK^g_{k+1} \in \mathbb{R}^{(k+1) \times (k+1)}$ and $\bK^h_{k+1} \in \mathbb{R}^{(k+1) \times (k+1)}$ by setting them equal to $\bK^g_{k}$ and $\bK^h_{k}$, respectively. We must show that for some $d$-dimensional random vector $\gse_{k+1}$ with i.i.d. entries that is jointly Gaussian with $\Gse_k$ and $n$-dimensional random vector $\hse_{k+1}$ with i.i.d. entries that is jointly Gaussian with $\Hse_{k}$, we can construct $d$-dimensional random vectors $\vse_{k+1}$ and $\vhatse_{k+1}$ and $n$-dimensional random vectors $\use_{k+1}$ and $\uhatse_{k+1}$ satisfying 

\begin{itemize}
	\item The following system of equations for all $\ell = k+1, \ell' \leq k+1$ or  $\ell' = k+1, \ell\leq k+1$:
	\begin{equation}
		\label{eq:fix-pt-HO-at-t}
		\begin{aligned}
			\E[ \inprod{\gse_{\ell}}{\gse_{\ell'}}] &\overset{(a)}{=} \E[\inprod{\use_{\ell}}{\use_{\ell'}}],
			\qquad&
			\E[ \inprod{\hse_{\ell}}{\hse_{\ell'}}] &\overset{(b)}{=} \E[\inprod{\vse_{\ell}}{\vse_{\ell'}}] 
			\\
			\E[ \inprod{\gse_{\ell}}{\vse_{\ell'}}] &\overset{(c)}{=} \E[\inprod{\use_{\ell}}{\uhatse_{\ell'}}], 
			\qquad&
			\E[ \sqrt{\Lambda} \inprod{\hse_{\ell}}{\use_{\ell'}}] &\overset{(d)}{=} \E[\inprod{\vse_{\ell}}{\vhatse_{\ell'}}].
		\end{aligned}
	\end{equation}

	\item The inclusions
	\begin{align} \label{eq:span-inclusions-induction-step-HO}
		\vhatse_{k+1} \in \spn( \{ \vse_{\ell} \}_{\ell \leq k+1}), \qquad &\uhatse_{k+1} \in \spn( \{ \bu^{\SE}_{\ell} \}_{s \leq t}).
	\end{align}
	
	\item The relations
	\begin{align}\label{eq:KKT-cond-induction-step}
				\sqrt{\Lambda} \hse_{k+1} - \uhatse_{k+1}
				=
				\nabla \phi_{k+1}^u\bigl(\use_{k+1}; \Use_k\bigr)
				,
				\qquad 
				\gse_{k+1} - \vhatse_{k+1} = \nabla \phi_{k+1}^{v}\bigl(\vse_{k+1}; \Vse_k\bigr).
	\end{align}
\end{itemize}
Furthermore, we would like the joint (conditional) laws of 
\[
(\gse_{k+1}, \vse_{k+1}, \vhatse_{k+1}) \mid \Gse_k, \Vse_k, \Vhatse_k \quad \text{ and } \quad (\hse_{k+1}, \use_{k+1}, \uhatse_{k+1}) \mid \Hse_k, \Use_k, \Uhatse_k
\]
to be uniquely defined.

 We perform the induction step by defining two Hilbert spaces as well as operators between them. We then define a useful optimization problem on these spaces that will help us argue uniqueness.

Consider auxiliary randomness $\bxi_g \sim \normal(0,\id_d / d)$ and $\bxi_h \sim \normal(0,\id_n / n)$ chosen independently of everything else.  
Define the Hilbert space $\mathcal{H}_d^{k+1}$ as all $d$-dimensional, square integrable functions of the tuple $(\Gse_k, \Vse_k, \Vhatse_k, \bxi_g)$ and the Hilbert space\footnote{Note that these are subspaces of the $L^2$ spaces considered in Section~\ref{sec:main-results}.} $\mathcal{H}_n^{k+1}$ of all $n$-dimensional, square integrable functions of the tuple $(\Hse_k, \Use_k, \Uhatse_k, \bxi_h)$, i.e.\footnote{Recall that the standard inner product between two deterministic vectors $\bx, \by$ of comparable dimension is denoted by $\inprod{\bx}{\by}$, and we have $\| \bx \|_2^2 = \inprod{\bx}{\bx}$. }
\begin{subequations}  \label{eq:Hilbert-defs}
	\begin{align}
		\mathcal{H}_d^{k+1} &= \{ \bmf:  \mathbb{R}^{d \times k} \times \mathbb{R}^{d \times k} \times \mathbb{R}^{d \times k} \times \mathbb{R}^d\to \mathbb{R}^d : \mathbb{E} [ \| \bmf(\Gse_k, \Vse_k, \Vhatse_k, \bxi_g,) \|_2^2] < \infty\}, \\
		\mathcal{H}_n^{k+1} &= \{ \bl: \mathbb{R}^{n \times k} \times \mathbb{R}^{n \times k} \times \mathbb{R}^{n \times k} \times \mathbb{R}^n \to \mathbb{R}^n : \mathbb{E} [ \| \bl(\Hse_k, \Use_k, \Uhatse_k, \bxi_h) \|_2^2 ]  < \infty\}.
	\end{align}
\end{subequations}
Equip each space with the inner product $\inprod{\cdot}{\cdot}_{L^2}$: for $\bmf, \bmf' \in \mathcal{H}_d^{k+1}$, we have 
\[
\inprod{\bmf}{\bmf'}_{L^2} = \mathbb{E} \bigl[ \inprod{\bmf(\Gse_k, \Vse_k, \Vhatse_k, \bxi_g)}{\bmf'(\Gse_k, \Vse_k, \Vhatse_k, \bxi_g)}  \bigr],
\]
and with a slight abuse of notation, we also use similar conventions for $\bl, \bl' \in \mathcal{H}_n^{k+1}$. For $\bmf \in \cH_d^{k+1}$, we write $\| \bmf \|_{L^2} = \sqrt{\inprod{\bmf}{\bmf}_{L^2}}$ and similarly define $\| \bl \|_{L^2}$ for $\bl \in \cH_n^{k+1}$.
For a matrix $\bZ = [\bx_1 \mid \cdots \mid \bx_N ]$ having columns $\bx_1, \ldots, \bx_N \in \cH_d^{k+1}$ and another $\bv \in \cH_d^{k+1}$, we write $\proj_{\bZ}^\perp \bv$ 
for the projection of $\bv$ onto the subspace orthogonal to the span of $\bx_1, \ldots, \bx_N$. Orthogonal projections in the Hilbert space $\mathcal{H}_n^{k+1}$ are defined similarly. Also recall our convention that $\llangle \bZ, \bv \rrangle_{L^2}$ is a deterministic $N$-dimensional vector with $j$-th entry equal to $\inprod{\bx_j}{\bv}_{L^2}$.

For $\ell = 1, 2, \ldots, k$, let 
\[
\useperp_{\ell} := \frac{\proj_{\Use_{\ell - 1}}^{\perp} \use_{\ell}}{\|\proj_{\Use_{\ell - 1}}^{\perp} \use_{\ell} \|_{L^2}}, \quad \vseperp_{\ell} := \frac{\proj_{\Vse_{\ell - 1}}^{\perp} \vse_{\ell}}{\|\proj_{\Vse_{\ell - 1}}^{\perp} \vse_{\ell} \|_{L^2}}, \quad \gseperp_{\ell} := \frac{\proj_{\Gse_{\ell - 1}}^{\perp} \gse_{\ell}}{\|\proj_{\Gse_{\ell - 1}}^{\perp} \gse_{\ell} \|_{L^2}}, \;\; \text{ and } \;\; \hseperp_{\ell} := \frac{\proj_{\Hse_{\ell - 1}}^{\perp} \use_{\ell}}{\|\proj_{\Hse_{\ell - 1}}^{\perp} \hse_{\ell} \|_{L^2}}.
\]
In all of these definitions, we adopt the convention that $\bm{0}/\| \bm{0} \|_{L^2} := \bm{0}$. Ignoring the zero elements, note that each of the collections $\{ \useperp_{\ell} \}_{\ell = 1}^{k}, \{ \vseperp_{\ell} \}_{\ell = 1}^{k}, \{ \gseperp_{\ell} \}_{\ell = 1}^{k}, \{ \hseperp_{\ell} \}_{\ell = 1}^{k}$ are orthonormal in their respective Hilbert spaces.

Next, define the linear operators $\mathsf{g}: \mathcal{H}_n^{k+1} \to \mathcal{H}_d^{k+1}$ and $\mathsf{h}: \mathcal{H}_d^{k+1} \to \mathcal{H}_n^{k+1}$ via 
\begin{equation} \label{eq:SE-g-hilbert-h-hilbert}
	\sg[\bu]
	=
	\sum_{\ell=1}^{k} \< \useperp_{\ell} , \bu \>_{L^2} \gseperp_{\ell}
	+
	\|\proj_{\Use_{k}}^\perp \bu \|_{L^2} \bxi_g,
	\qquad
	\sh[\bv]
	=
	\sum_{\ell=1}^{k} \< \vseperp_{\ell} , \bv \>_{L^2} \hseperp_{\ell}
	+
	\|\proj_{\Vse_{k}}^\perp \bv \|_{L^2} \bxi_h,
\end{equation}
and the function $\mathfrak{L}_{k+1}: \mathcal{H}_n^{k+1} \times \mathcal{H}_d^{k+1} \to \mathbb{R}$ via
\begin{align} \label{def:saddle-obj-Hilbert}
	\mathfrak{L}_{k+1}(\bu, \bv) := -\<\sg[\bu] , \bv \>_{L^2}
	+\sqrt{\Lambda} \<\sh[\bv] , \bu \>_{L^2}
	-\E[\phi_{k+1}^u(\bu;\Use_{k})]
	+\E[\phi_{k+1}^v(\bv;\Vse_k)].
\end{align}
Also define the linear functions $\mathfrak{G}: \mathcal{H}_d^{k+1} \to \mathbb{R}$ and $\mathfrak{H}: \mathcal{H}_n^{k+1} \to \mathbb{R}$ via
\begin{align} \label{def:saddle-const-Hilbert}
	\mathfrak{G}(\bv) := \inprod{\bv}{\bxi_g}_{L^2} \quad \text{ and } \quad \mathfrak{H}(\bu) := \inprod{\bu}{\bxi_h}_{L^2},
\end{align}
respectively.
Finally, define the constrained saddle point problem
\begin{equation}
	\label{eq:L2-inductive-saddle-new}
	\min_{ \substack{\bv \in \mathcal{H}_d^{k+1} \\ \mathfrak{G}(\bv) \geq 0}} \;\;
	\max_{ \substack{\bu \in \mathcal{H}_n^{k+1} \\ \mathfrak{H}(\bu) \geq 0}} \;\;
	\mathfrak{L}_{k+1}(\bu, \bv).
\end{equation}

Equipped with this setup, we next present several useful properties. 
\begin{lemma} \label{lem:Hilbert-saddle-lemma}
	The saddle problem~\eqref{eq:L2-inductive-saddle-new} has a unique solution $(\butilde, \bvtilde) \in \cH_n^{k+1} \times \cH_p^{k+1}$ that also satisfies the equalities
	\begin{subequations} \label{eq:unique-KKT-saddle}
		\begin{align}
			\sg[\butilde] 
			-
			\sqrt{\Lambda} \sum_{\ell=1}^{k} \< \hseperp_{\ell} , \butilde \>_{L^2} \vseperp_{\ell}
			-
			\sqrt{\Lambda} \< \bxi_h , \butilde \>_{L^2} \frac{\proj_{\Vse_k}^\perp\bvtilde}{\|\proj_{\Vse_{k}}^\perp\bvtilde\|_{L^2}}
			&=
			(\nabla \phi_{k+1}^v)(\bvtilde; \Vse_k), \label{eq:unique-KKT-saddle-g}
			\\
			\sqrt{\Lambda} \sh[\bvtilde] 
			-
			\sum_{\ell=1}^{k} \< \gseperp_{\ell}, \bvtilde \>_{L^2} \useperp_{\ell}
			-
			\< \bxi_g , \bvtilde \>_{L^2} \frac{\proj_{\Use_k}^\perp\butilde}{\|\proj_{\Use_k}^\perp\butilde\|_{L^2}}
			&=
			(\nabla \phi_t^u)(\butilde;\bU_{k}), \label{eq:unique-KKT-saddle-h}
		\end{align}
	\end{subequations}
	where we use the convention that $\bm{0} / \| \bm{0} \|_{L^2} = \bm{0}$.
\end{lemma}
Lemma~\ref{lem:Hilbert-saddle-lemma} is proved in Section~\ref{sec:proof-saddle-fixed-pt-lemma}; operationally, it shows that provided we adopt the convention $\bm{0} / \| \bm{0} \|_{L^2} = \bm{0}$, the conditions~\eqref{eq:unique-KKT-saddle}---which are effectively the KKT conditions of the \emph{unconstrained} optimization problem $\min_{\bv \in \mathcal{H}_d^{k+1}} \;\;
\max_{\bu \in \mathcal{H}_n^{k+1}} \;\;
\mathfrak{L}_{k+1}(\bu, \bv)$---are also the same conditions characterizing optimality in the \emph{constrained} problem~\eqref{eq:L2-inductive-saddle-new}. This technical lemma is the key to our argument, and is crucial to establishing the following propositions.

\begin{proposition} [Existence of fixed point solutions] \label{prop:existence-SE-fixpt}
	Let $(\butilde, \bvtilde) \in \cH_n^{k+1} \times \cH_d^{k+1}$ denote the solution to the saddle problem~\eqref{eq:L2-inductive-saddle-new} guaranteed by Lemma~\ref{lem:Hilbert-saddle-lemma}. Then the tuple $(\gse_{k+1}, \vse_{k+1}, \vhatse_{k+1}, \hse_{k+1}, \use_{k+1}, \uhatse_{k+1})$ given by
	\begin{subequations} \label{eq:construction-fixedpoint}
		\begin{align}
			\use_{k+1} = \butilde, \qquad  \vse_{k+1} = \bvtilde,  \qquad \gse_{k+1} = \sg[\butilde], \qquad \hse_{k+1} = \sh[\bvtilde],
		\end{align}
		\begin{align}
			\uhatse_{k+1} &= \sum_{\ell=1}^{k} \< \gseperp_{\ell}, \bvtilde \>_{L^2} \useperp_{\ell} + \< \bxi_g , \bvtilde \>_{L^2} \frac{\proj_{\Use_k}^\perp\butilde}{\|\proj_{\Use_k}^\perp\butilde\|_{L^2}}, \label{eq:uhatse-defn}\\
			\vhatse_{k+1} &= \sqrt{\Lambda}\sum_{\ell=1}^{k} \< \hseperp_{\ell}, \butilde \>_{L^2} \vseperp_{\ell} + \sqrt{\Lambda} \< \bxi_h , \butilde \>_{L^2} \frac{\proj_{\Vse_k}^\perp\bvtilde}{\|\proj_{\Vse_k}^\perp\bvtilde\|_{L^2}}. \label{eq:vhatse-defn}
		\end{align}
	\end{subequations}
	satisfies Eqs.~\eqref{eq:fix-pt-HO-at-t},~\eqref{eq:span-inclusions-induction-step-HO} and~\eqref{eq:KKT-cond-induction-step}.
\end{proposition}
Proposition~\ref{prop:existence-SE-fixpt} is proved in Section~\ref{sec:prop1-proof}.

\begin{proposition} [Uniqueness of fixed point solutions] \label{prop:uniqueness-SE-fixpt}
	If the tuple $\bigl(\bg_{k+1}^{\SE, *}, \bv_{k+1}^{\SE, *}, \hbv_{k+1}^{\SE, *}, \bh_{k+1}^{\SE, *}, \bu_{k+1}^{\SE, *}, \hbu_{k+1}^{\SE, *}\bigr)$ is a valid fixed point solution satisfying Eqs.~\eqref{eq:fix-pt-HO-at-t},~\eqref{eq:span-inclusions-induction-step-HO} and~\eqref{eq:KKT-cond-induction-step}, then $(\bv_{k+1}^{\SE, *}, \bu_{k+1}^{\SE, *})$ is a solution to the saddle problem~\eqref{eq:L2-inductive-saddle-new}. Furthermore, we must have
	\begin{subequations} \label{eq:construction-candidate-fixedpoint}
		\begin{align}
			\bg_{k+1}^{\SE, *} = \sg\bigl[\bu_{k+1}^{\SE, *}\bigr], \qquad \bh_{k+1}^{\SE, *} = \sh\bigl[\bv_{k+1}^{\SE, *}\bigr],
		\end{align}
		\begin{align}
			\hbu_{k+1}^{\SE, *} &= \sum_{\ell=1}^{k} \< \gseperp_{\ell}, \bv_{k+1}^{\SE, *} \>_{L^2} \bu_{\ell}^{\SE,\perp} + \< \bxi_g , \bv_{k+1}^{\SE, *} \>_{L^2} \frac{\proj_{\bU_{k}^{\SE}}^\perp \bu_{k+1}^{\SE, *}}{\|\proj_{\bU_{k}^{\SE}}^\perp \bu_{k+1}^{\SE, *} \|_{L^2}}, \\
			\hbv_{k+1}^{\SE, *} &= \sqrt{\Lambda} \sum_{\ell=1}^{k} \< \bh_{\ell}^{\SE,\perp}, \bv_{k+1}^{\SE, *} \>_{L^2} \bv_{\ell}^{\SE,\perp} + \sqrt{\Lambda} \< \bxi_h , \bv_{k+1}^{\SE, *} \>_{L^2} \frac{\proj_{\bV_{k}^{\SE}}^\perp \bv_{k+1}^{\SE, *}}{\|\proj_{\bV_{k}^{\SE}}^\perp \bv_{k+1}^{\SE, *}\|_{L^2}}.
		\end{align}
	\end{subequations}
\end{proposition}
Proposition~\ref{prop:uniqueness-SE-fixpt} is proved in Section~\ref{sec:prop2-proof}.

Armed with these two propositions, we now complete the inductive step. On the one hand, Proposition~\ref{prop:existence-SE-fixpt} and Lemma~\ref{lem:Hilbert-saddle-lemma} together guarantee a solution $(\bv^{\SE}_t, \bu^{\SE}_t)$ to the SE fixed point equations~\eqref{eq:fix-pt-HO-at-t},~\eqref{eq:span-inclusions-induction-step-HO} and~\eqref{eq:KKT-cond-induction-step}. 
On the other hand, by Proposition~\ref{prop:existence-SE-fixpt}, any solution to the SE fixed point equations can be used to construct a solution $(\bv_{k+1}^{\SE, *}, \bu_{k+1}^{\SE, *})$ to the saddle problem~\eqref{eq:L2-inductive-saddle-new}. But since Lemma~\ref{lem:Hilbert-saddle-lemma} also guarantees that the solution to the saddle problem~\eqref{eq:L2-inductive-saddle-new} is unique, we must have $(\bv_{k+1}^{\SE, *}, \bu_{k+1}^{\SE, *}) = (\bv_{k+1}^{\SE}, \bu_{k+1}^{\SE})$. Finally, comparing Eqs.~\eqref{eq:construction-candidate-fixedpoint} and Eq.~\eqref{eq:construction-fixedpoint} and noting that $(\bv_{k+1}^{\SE, *}, \bu_{k+1}^{\SE, *}) = (\bv_{k+1}^{\SE}, \bu_{k+1}^{\SE})$, we may conclude equality of the entire tuple $(\bg_{k+1}^{\SE, *}, \bv_{k+1}^{\SE, *}, \hbv_{k+1}^{\SE, *}, \bh_{k+1}^{\SE, *}, \bu_{k+1}^{\SE, *}, \hbu_{k+1}^{\SE, *}) = (\bg_{k+1}^{\SE}, \bv_{k+1}^{\SE}, \hbv_{k+1}^{\SE}, \bh_{k+1}^{\SE}, \bu_{k+1}^{\SE}, \hbu_{k+1}^{\SE})$. We have thus established existence and uniqueness of the SE fixed points at time $t$, and this completes the induction step for saddle update.
\qed

\subsubsection{Proof of Lemma~\ref{lem:Hilbert-saddle-lemma}} \label{sec:proof-saddle-fixed-pt-lemma}

Since the functions $\phi_{k+1}^u$, $\phi_{k+1}^v$ are strongly convex in their first arguments, the maps $\E[\phi^u_{k+1}]: \mathcal{H}_n^{k+1} \to \mathbb{R}$ and $\E[\phi^v_{k+1}]: \mathcal{H}_d^{k+1} \to \mathbb{R}$ are also strongly convex. Furthermore, both sets $\mathcal{V} = \{ \bv: \mathfrak{G}(\bv) \geq 0 \}$ and $\mathcal{U} = \{ \bu: \mathfrak{H}(\bu) \geq 0 \}$ are convex, and on $\mathcal{U} \times \mathcal{V}$ the function $\mathfrak{L}_{k+1}$~\eqref{def:saddle-obj-Hilbert} is strongly concave convex. 
Define the function $\overline{\mathfrak{L}}_{k+1}$ mapping $\mathcal{H}_n^{k+1} \times \mathcal{H}_d^{k+1}$ to values on the extended reals, via
\[
\overline{\mathfrak{L}}_{k+1}(\bu, \bv) :=  \mathfrak{L}_{k+1}(\bu, \bv) + \mathbb{I}_{\mathcal{V}}(\bv) - \mathbb{I}_{\mathcal{U}}(\bu),
\]
where $\mathbb{I}_{\mathcal{X}}(x)$ is an indicator that takes value $0$ if $x \in \mathcal{X}$ and the value $+\infty$ otherwise. The function $\overline{\mathfrak{L}}_{k+1}$ is a proper saddle function in the terminology of~\citet[Definition 2.106]{barbu2012convexity}.  Moreover, since $\mathfrak{L}_{k+1}$ is strongly convex in $\bv$ and strongly concave in $\bu$ (and hence coercive in each variable respectively) on $\mathcal{U}\times \mathcal{V}$, we see that $\lim_{\| \bu \|_{L^2} + \| \bv \|_{L^2} \rightarrow \infty, (\bu, \bv) \in \mathcal{U}\times \mathcal{V}} = -\infty$, so that by~\citep[Corollary 2.118]{barbu2012convexity}, there exists a solution to the saddle point problem 
\begin{align}\label{eq:extended-hilbert-saddle}
	\min_{\bv \in \mathcal{H}_d^{k+1}} \max_{\bu \in \mathcal{H}_n^{k+1}} \overline{\mathfrak{L}}_{k+1}(\bu, \bv).
\end{align}
Additionally, strong convexity of the map $\bv \mapsto \max_{\bu \in \mathcal{U}} \overline{\mathfrak{L}}_{k+1}(\bu, \bv)$ ensures that any solution to the above saddle point problem is unique.  Further,~\citet[Theorem 2.114]{barbu2012convexity} implies that the subdifferential $\partial \overline{\mathfrak{L}}_{k+1} = (\partial_{\bv} \overline{\mathfrak{L}}_{k+1}(\bu, \bv), -\partial_{\bu} \overline{\mathfrak{L}}_{k+1}(\bu, \bv))$ is maximally monotone and by~\citet[Eq.~(2.145)]{barbu2012convexity}, the solution $(\butilde, \bvtilde)$ to the saddle point problem~\eqref{eq:extended-hilbert-saddle} satisfies $(\bm{0}, \bm{0}) \in \partial \overline{\mathfrak{L}}_{k+1}(\butilde, \bvtilde)$. This, in turn, by dominated convergence, simplifies to the system
\begin{subequations} \label{eq:KKT-orig}
	\begin{align}
		\sg[\butilde] 
		-
		\sqrt{\Lambda} \sum_{\ell=1}^{k} \inprod{\bh_{\ell}^{\SE,\perp}}{\butilde}_{L^2} \bv_{\ell}^{\SE,\perp}
		-
		\sqrt{\Lambda} \inprod{\bxi_h}{\butilde}_{L^2} \cdot \partial_{\bv} \left(\|\proj_{\bV_{k}^{\SE}}^\perp\bv\|_{L^2} \right)_{\bv = \bvtilde}
		+
		\lambda_v \bxi_g
		&=
		(\nabla \phi_{k+1}^v)(\bvtilde;\Vse_k) \\
		\sqrt{\Lambda} \sh[\bvtilde] 
		-
		\sum_{\ell=1}^{k} \inprod{\bg_{\ell}^{\SE,\perp}}{\bvtilde}_{L^2} \bu_{\ell}^{\SE,\perp}
		-
		\inprod{\bxi_g}{\bvtilde}_{L^2} \cdot \partial_{\bu} \left(\|\proj_{\bU_{k}^{\SE}}^\perp\bu\|_{L^2} \right)_{\bu = \butilde}
		+
		\lambda_u \bxi_h
		&=
		(\nabla \phi_{k+1}^u)(\butilde;\Use_k),
	\end{align}
\end{subequations}
where $\lambda_v \geq 0$ should be interpreted as a slack variable for the constraint $\inprod{\bxi_g}{\bv}_{L^2} \geq 0$, and so $\lambda_v \cdot \inprod{\bxi_g}{\bv}_{L^2} = 0$. Similarly, $\lambda_u \geq 0$ should be interpreted as a slack variable for the constraint $\inprod{\bxi_h}{\bu}_{L^2} \geq 0$, and so $\lambda_u \cdot \inprod{\bxi_h}{\bu}_{L^2} = 0$.

The following claim is a key cog of the proof.
\begin{claim} \label{clm:chains}
	The following chains of implication hold:
	\begin{enumerate}
		\item[(i)] We have $\|\proj_{\bU_{k}^{\SE}}^\perp\butilde\|_{L^2} = 0$ if and only if $\|\proj_{\bV_{k}^{\SE}}^\perp\bvtilde\|_{L^2} = 0$, and if so then $\inprod{\bxi_h}{\butilde}_{L^2} = 0$ and $\inprod{\bxi_g}{\bvtilde}_{L^2} = 0$.
		\item[(ii)] If $\|\proj_{\bU_{k}^{\SE}}^\perp\butilde\|_{L^2} > 0$ and $\|\proj_{\bV_{k}^{\SE}}^\perp\bvtilde\|_{L^2} > 0$, then $\inprod{\bxi_g}{\bvtilde}_{L^2} > 0$ and $\inprod{\bxi_h}{\butilde}_{L^2} > 0$.
	\end{enumerate}
\end{claim}
We return to prove this claim shortly. 
In view of the claim, there are two cases to consider.

\paragraph{Case 1: $\|\proj_{\bU_{k}^{\SE}}^\perp\butilde\|_{L^2} = \|\proj_{\bV_{k}^{\SE}}^\perp\bvtilde\|_{L^2} = \inprod{\bxi_h}{\butilde}_{L^2} = \inprod{\bxi_g}{\bvtilde}_{L^2} = 0$:} 
Since we have adopted the convention that $\bm{0} / \| \bm{0} \|_{L^2} = 0$, we may write 
\begin{align*}
0 &= \inprod{\bxi_h}{\butilde}_{L^2} \cdot \partial_{\bv} \left(\|\proj_{\bV_{k}^{\SE}}^\perp\bv\|_{L^2} \right)_{\bv = \bvtilde} = \inprod{\bxi_h}{\butilde}_{L^2} \cdot \frac{\proj_{\bV_{k}^{\SE}}^\perp\bvtilde}{\|\proj_{\bV_{k}^{\SE}}^\perp\bvtilde\|_{L^2}} \\
&= \inprod{\bxi_g}{\bvtilde}_{L^2} \cdot \partial_{\bu} \left(\|\proj_{\bU_{k}^{\SE}}^\perp\bu\|_{L^2} \right)_{\bu = \butilde} = \inprod{\bxi_g}{\bvtilde}_{L^2} \cdot \frac{\proj_{\bU_{k}^{\SE}}^\perp\butilde}{\|\proj_{\bU_{k}^{\SE}}^\perp\butilde\|_{L^2}}.
\end{align*}
The KKT conditions~\eqref{eq:KKT-orig} thus simplify to
\begin{subequations} \label{eq:simplified-KKT}
	\begin{align}
		\sum_{\ell=1}^{k} \< \bu_{\ell}^{\SE,\perp} , \butilde \>_{L^2} \bg_{\ell}^{\SE,\perp}
		-
		\sqrt{\Lambda} \sum_{\ell=1}^{k} \inprod{\bh_{\ell}^{\SE,\perp}}{\butilde}_{L^2} \bv_{\ell}^{\SE,\perp}
		+
		\lambda_v \bxi_g
		&=
		(\nabla \phi_{k+1}^v)(\bvtilde;\Vse_{k}), \\
		\sqrt{\Lambda} \sum_{\ell = 1}^{k} \< \bv_{\ell}^{\SE,\perp} , \bvtilde \>_{L^2} \bh_{\ell}^{\SE,\perp}
		-
		\sum_{\ell=1}^{k} \inprod{\bg_{\ell}^{\SE,\perp}}{\bvtilde}_{L^2} \bu_{\ell}^{\SE,\perp}
		+
		\lambda_u \bxi_h
		&=
		(\nabla \phi_{k+1}^u)(\butilde;\Use_{k}),
	\end{align}
\end{subequations}
where we have substituted for the operators $\sg$ and $\sh$ from Eq.~\eqref{eq:SE-g-hilbert-h-hilbert} and used $\|\proj_{\bU_{k}^{\SE}}^\perp\butilde\|_{L^2} = \|\proj_{\bV_{k}^{\SE}}^\perp\bvtilde\|_{L^2} = 0$. 

Now note that the random vectors $\bxi_g$ and $\bxi_h$ are chosen independently of the collections $\{ \bv^{\SE}_{\ell} \}_{\ell = 1}^{k}$ and $\{ \bu^{\SE}_{\ell} \}_{\ell = 1}^{k}$, respectively.
In Eqs.~\eqref{eq:simplified-KKT} defining the pair $(\butilde, \bvtilde)$, suppose for the sake of contradiction that $\lambda_v > 0$. Then $\bvtilde$ depends nontrivially on $\bxi_g$. On the other hand, $\proj_{\bV_{k}^{\SE}}^\perp\bvtilde = 0$, so the random vector $\hbv$ is a measurable function of the collection $\{ \bv^{\SE}_{\ell} \}_{\ell = 1}^{k}$, which is a contradiction. Thus, we must have $\lambda_v = 0$. Reasoning via a parallel argument, we also have $\lambda_u = 0$.

To conclude, note that the KKT conditions~\eqref{eq:simplified-KKT} simplify to Eq.~\eqref{eq:unique-KKT-saddle} when $\lambda_v = \lambda_u = 0$ and under the convention that $\bm{0} / \| \bm{0} \|_{L^2} = 0$.

\paragraph{Case 2: $\min(\|\proj_{\bU_{k}^{\SE}}^\perp\butilde\|_{L^2}, \|\proj_{\bV_{k}^{\SE}}^\perp\bvtilde\|_{L^2}, \inprod{\bxi_g}{\bvtilde}_{L^2}, \inprod{\bxi_h}{\butilde}_{L^2}) > 0$:} In this case, the subdifferentials in Eq.~\eqref{eq:KKT-orig} are derivatives, and we have
\begin{align}
	\partial_{\bv} \left(\|\proj_{\bV_{k}^{\SE}}^\perp\bv\|_{L^2} \right)_{\bv = \bvtilde} = \frac{\proj_{\bV_{k}^{\SE}}^\perp\bvtilde}{\|\proj_{\bV_{k}^{\SE}}^\perp\bvtilde\|_{L^2}} \quad \text{ and } \quad \partial_{\bu} \left(\|\proj_{\bU_{k}^{\SE}}^\perp\bu\|_{L^2} \right)_{\bu = \butilde} = \frac{\proj_{\bU_{k}^{\SE}}^\perp\butilde}{\|\proj_{\bU_{k}^{\SE}}^\perp\butilde\|_{L^2}}
\end{align}
Furthermore, since $\min(\inprod{\bxi_g}{\bvtilde}_{L^2}, \inprod{\bxi_h}{\butilde}_{L^2}) > 0$, we have $\lambda_v = \lambda_u = 0$, since neither of the constraints is met with equality. Consequently, Eq.~\eqref{eq:KKT-orig} simplifies to Eq.~\eqref{eq:unique-KKT-saddle}, as claimed.

\medskip

\paragraph{Proof of Claim~\ref{clm:chains}:} 
We prove each subclaim in turn.

\medskip

\noindent \underline{Establishing (i):}
We will first show that if $\| \proj_{\bU_{k}^{\SE}}^\perp \butilde \|_{L^2} = 0$, then $\|\proj_{\bV_{k}^{\SE}}^\perp\bvtilde\|_{L^2} = \inprod{\bxi_h}{\butilde}_{L^2} = \inprod{\bxi_g}{\bvtilde}_{L^2} = 0$. Since $\proj_{\bU_{k}^{\SE}}^\perp \butilde = 0$ and because $\bxi_h$ is chosen independently of $\bU_{k}^{\SE}$,
we must have that $\inprod{\bxi_h}{\butilde}_{L^2} = 0$.
The first KKT condition~\eqref{eq:KKT-orig} can then be written as
\begin{equation} \label{eq:helper-KKT}
	\sum_{\ell = 1}^{k} \< \bu_{\ell}^{\SE,\perp} , \butilde \>_{L^2} \bg_{\ell}^{\SE,\perp}
	-
	\sqrt{\Lambda} \sum_{\ell = 1}^{k} \< \bh_{\ell}^{\SE,\perp} , \butilde \>_{L^2} \bv_{\ell}^{\SE,\perp}
	+
	\lambda_v \bxi_g
	=
	(\nabla \phi_{k+1}^v)(\bvtilde;\Vse_{k}),
\end{equation}
and the second KKT condition~\eqref{eq:KKT-orig} as
\begin{equation} \label{eq:helper-KKT-2}
	\sqrt{\Lambda} \sum_{\ell = 1}^{k} \< \bv_{\ell}^{\SE,\perp} , \bvtilde \>_{L^2} \bh_{\ell}^{\SE,\perp}
	-
	\sum_{\ell = 1}^{k} \inprod{\bg_{\ell}^{\SE,\perp}}{\bvtilde}_{L^2} \bu_{\ell}^{\SE,\perp}
	+
	(\sqrt{\Lambda}\|\proj_{\bV_{k}^{\SE}}^\perp\bvtilde\|_{L^2} + \lambda_u) \bxi_h
	=
	(\nabla \phi_{k+1}^u)(\butilde;\Use_{k})
\end{equation}
Suppose that $\sqrt{\Lambda}\|\proj_{\bV_{k}^{\SE}}^\perp\bvtilde\|_{L^2} + \lambda_u > 0$. Then by condition~\eqref{eq:helper-KKT-2}, the vector $\butilde$ depends nontrivially on $\bxi_h$, which is in turn independent of everything else. But $\proj_{\bU_{k}^{\SE}}^\perp \butilde = 0$, and so $\butilde$ is a measurable function of $\{ \use_{\ell} \}_{\ell = 1}^{k}$ and cannot depend on $\bxi_h$. This leads to a contradiction, and so $\sqrt{\Lambda} \|\proj_{\bV_{k}^{\SE}}^\perp\bvtilde\|_{L^2} + \lambda_u = 0$. Since both terms are nonnegative, we must have $\|\proj_{\bV_{k}^{\SE}}^\perp\bvtilde\|_{L^2} = \lambda_u = 0$.

A parallel argument under the supposition that $\| \proj_{\bV_{k}^{\SE}}^\perp \bvtilde \|_{L^2} = 0$ yields $\|\proj_{\bU_{k}^{\SE}}^\perp\butilde\|_{L^2} = \lambda_v = 0$. Putting these two pieces together, we see that $\|\proj_{\bU_{k}^{\SE}}^\perp\butilde\|_{L^2} = 0$ if and only if $\|\proj_{\bV_{k}^{\SE}}^\perp\bvtilde\|_{L^2} = 0$, and if so then $\inprod{\bxi_h}{\butilde}_{L^2} = 0$ and $\inprod{\bxi_g}{\bvtilde}_{L^2} = 0$.

\medskip

\noindent \underline{Establishing (ii):}
Expanding the definition of the operator $\sg$ from Eq.~\eqref{eq:SE-g-hilbert-h-hilbert} and using the fact that \mbox{$\|\proj_{\bV_{k}^{\SE}}^\perp \bvtilde\|_{L^2} > 0$},
the first KKT condition~\eqref{eq:KKT-orig} can be written as
\begin{equation}
	(\|\proj_{\bU_{k}^{\SE}}^\perp \butilde\|_{L^2}+\lambda_v) \bxi_g
	+
	\sum_{\ell = 1}^{k} \inprod{\hbu^{\SE, \perp}_{\ell}}{\butilde}_{L^2} \bg_{\ell}^{\SE,\perp}
	-
	\sqrt{\Lambda}\sum_{\ell = 1}^{k} \inprod{\bh^{\SE, \perp}_{\ell}}{\butilde}_{L^2} \bv_{\ell}^{\SE,\perp}
	=
	\sqrt{\Lambda} \frac{\< \bxi_h , \butilde \>_{L^2}}{\|\proj_{\bV_{k}^{\SE}}^\perp\bvtilde\|_{L^2}} \cdot \proj_{\bV_{k}^{\SE}}^\perp\bvtilde
	+
	(\nabla \phi_{k+1}^v)(\bvtilde;\Vse_{k}).
\end{equation}
In particular, viewing $\butilde$ as fixed, the vector $\bvtilde$ can be viewed as the minimizer---under the constraint that $\inprod{\bxi_g}{\bv}_{L^2} \geq 0$---of the function $\alpha(\bv) + \beta(v) - \gamma(\bv)$, where
\begin{align*}
	\alpha(\bv) &= \sqrt{\Lambda} \< \bxi_h , \butilde \>_{L^2} \cdot \|\proj_{\bV_{k}^{\SE}}^\perp\bv\|_{L^2}, \\
	\beta(\bv) &= \phi_t^v(\bv;\bV_{k}) + \sqrt{\Lambda} \sum_{\ell = 1}^{k} \inprod{\bh^{\SE, \perp}_{\ell}}{\butilde}_{L^2} \inprod{\bv_{\ell}^{\SE,\perp}}{\bv}_{L^2} - \sum_{\ell = 1}^{k} \inprod{\hbu^{\SE, \perp}_{\ell}}{\butilde}_{L^2} \inprod{\bg_{\ell}^{\SE,\perp}}{\bv}_{L^2}, \text{ and } \\
	\gamma(\bv) &= \|\proj_{\bU_{k}^{\SE}}^\perp \butilde\|_{L^2} \cdot \inprod{\bxi_g}{\bv}_{L^2}.
\end{align*}

We now show that if in addition $\|\proj_{\bU_{k}^{\SE}}^\perp \butilde\|_{L^2} > 0$, then $\inprod{\bxi_g}{\bvtilde}_{L^2} > 0$.
Suppose that $\|\proj_{\bU_{k}^{\SE}}^\perp \butilde\|_{L^2} > 0$. 
In preparing to apply Lemma~\ref{lem:prox-like} from the appendix, set 
\begin{align}
	A &= \proj_{\bV_{k}^{\SE}}^\perp, \quad x_0 = \|\proj_{\bU_{k}^{\SE}}^\perp \butilde\|_{L^2} \cdot \bxi_g, \quad c_1 = \sqrt{\Lambda} \< \bxi_h , \butilde \>_{L^2}, \quad f = \beta,
\end{align}
and note that $f$ is $L$-smooth. We also have that $x_0 \neq 0$ since $\|\proj_{\bU_{k}^{\SE}}^\perp \butilde\|_{L^2} > 0$, and that $c_1 \geq 0$. Note that $x_1$ in the notation of Lemma~\ref{lem:prox-like} is the solution to the unconstrained problem
\begin{align}
	\widetilde{\bv} = \argmin_{\bv} \; \alpha(\bv) + \beta(\bv),
\end{align}
and is unique since the function $\alpha + \beta$ is strongly convex. Finally, we have $Ax_2 = \proj_{\bV_{k}^{\SE}}^\perp \bvtilde \neq 0$. Applying 
Lemma~\ref{lem:prox-like} then directly implies that $\|\proj_{\bU_{k}^{\SE}}^\perp \butilde\|_{L^2} \cdot \inprod{\bxi_g}{\bvtilde}_{L^2} > 0$, which in turn implies that $\inprod{\bxi_g}{\bvtilde}_{L^2} > 0$. 

Executing the same argument with the second KKT condition~\eqref{eq:KKT-orig}, we obtain that $\inprod{\bxi_h}{\butilde}_{L^2} > 0$. Putting together the two pieces completes the proof.
\qed

\subsubsection{Proof of Proposition~\ref{prop:existence-SE-fixpt}} \label{sec:prop1-proof}

Recall the objects constructed in Eq.~\eqref{eq:construction-fixedpoint}.
By definition~\eqref{eq:SE-g-hilbert-h-hilbert} of the operator $\sg[\,\cdot\,]$ and since $\bxi_g$ is an independent standard Gaussian scaled by $d^{-1/2}$, we see that $\bg_{k+1}^{\SE} \in \mathbb{R}^{d}$ is jointly Gaussian with $\bG^{\SE}_{k}$ and has independent entries. 
Similarly using the definition of the operator $\sh[\,\cdot\,]$ and since $\bxi_h$ is an independent standard Gaussian scaled by $n^{-1/2}$, the vector $\bh_{k+1}^{\SE} \in \mathbb{R}^{n}$ is jointly Gaussian with $\bH^{\SE}_{k}$ and has independent entries.
Next, we have that $\hbu_{k+1}^{\SE} \in \spn\{\bu_{\ell}^{\SE},\,\ell\leq k+1\}$ and $\hbv_{k+1}^{\SE} \in \spn\{\bv_{\ell}^{\SE},\,\ell\leq k+1\}$ by construction. Moreover, since $\bg^{\SE}_{k+1} = \sg[\butilde]$, combining Eqs.~\eqref{eq:vhatse-defn} and~\eqref{eq:unique-KKT-saddle-g} yields
\begin{subequations} \label{eq:verify-KKT}
\begin{align}
\bg_{k+1}^{\SE} - \hbv^{\SE}_{k+1} = \sg[\butilde] 
		-
		\sqrt{\Lambda} \sum_{\ell = 1}^{k} \< \bh_{\ell}^{\SE,\perp} , \butilde \>_{L^2} \bv_{\ell}^{\SE,\perp}
		-
		\sqrt{\Lambda} \< \bxi_h , \butilde \>_{L^2} \frac{\proj_{\bV_{k}^{\SE}}^\perp\bvtilde}{\|\proj_{\bV_{k}^{\SE}}^\perp\bvtilde\|_{L^2}}
		=
		(\nabla \phi_{k+1}^v)(\bv^{\SE}_{k+1};\Vse_{k}).
\end{align}
Similarly, combining Eqs.~\eqref{eq:uhatse-defn} and~\eqref{eq:unique-KKT-saddle-h} with the definition $\bh^{\SE}_{k+1} = \sh[\bvtilde]$ yields
\begin{align}
\sqrt{\Lambda} \bh_{k+1}^{\SE} - \hbu^{\SE}_{k+1} = \sqrt{\Lambda} \sh[\bvtilde] 
		-
		\sum_{\ell = 1}^{k} \< \bg_{\ell}^{\SE,\perp} , \bvtilde \>_{L^2} \bu_{\ell}^{\SE,\perp}
		-
		\< \bxi_g , \bvtilde \>_{L^2} \frac{\proj_{\bU_{k}^{\SE}}^\perp\butilde}{\|\proj_{\bU_{k}^{\SE}}^\perp\butilde\|_{L^2}}
		=
		(\nabla \phi_{k+1}^u)(\bu^{\SE}_{k+1};\Use_{k}).
\end{align}
\end{subequations}
Taken together, the calculations in Eq.~\eqref{eq:verify-KKT} verify Eq.~\eqref{eq:KKT-cond-induction-step}. Furthermore, Eq.~\eqref{eq:span-inclusions-induction-step-HO} is immediately true. To complete the induction step,
 it remains to verify Eq.~\eqref{eq:fix-pt-HO-at-t}.

By the definition~\eqref{eq:SE-g-hilbert-h-hilbert} of the operator $\sg[\,\cdot\,]$, we have that for $\ell \leq k$, 
\begin{align} 
\inprod{\bg^{\SE}_{\ell}}{\bg^{\SE}_{k+1}}_{L^2} 
= \inprod{\bg^{\SE}_{\ell}}{\sg[\butilde]}_{L^2} 
&\overset{\1}{=} \Bigl \langle\sum_{\ell' = 1}^{\ell} \inprod{\bg^{\SE}_{\ell}}{\bg^{\SE, \perp}_{\ell'}}_{L^2} \cdot \bg^{\SE, \perp}_{\ell'},  \sum_{\ell''=1}^{k} \< \bu_{\ell''}^{\SE,\perp} , \butilde \>_{L^2} \cdot \bg_{\ell''}^{\SE,\perp}
	+
	\|\proj_{\bU_{k}^{\SE}}^\perp \butilde \|_{L^2} \bxi_g \Bigr \rangle_{L^2} \notag \\
&\overset{\2}{=} \sum_{\ell' = 1}^{\ell} \inprod{\bg^{\SE}_{\ell}}{\bg^{\SE, \perp}_{\ell'}}_{L^2} \cdot \inprod{\bu_{\ell'}^{\SE,\perp}}{\butilde }_{L^2} = \sum_{\ell' = 1}^{\ell} \inprod{\bg^{\SE}_{\ell}}{\bg^{\SE, \perp}_{\ell'}}_{L^2} \cdot \inprod{\bu_{\ell'}^{\SE,\perp}}{\bu^{\SE}_{k+1} }_{L^2}. \label{eq:SE-verify-first}
\end{align}
Here step $\1$ follows from expanding $\bg^{\SE}_{\ell}$ as a linear combination of the orthonormal vectors $\{ \bg^{\SE, \perp}_{\ell'} \}_{\ell' = 1}^\ell$, and step $\2$ uses the orthonormality of these vectors along with the fact that $\inprod{\bxi_g}{\bg^{\SE}_{\ell}}_{L^2} = 0$ for all $\ell \leq k$.
Now by the induction hypothesis, we have $\inprod{\bg^{\SE}_{\ell}}{\bg^{\SE, \perp}_{\ell'}}_{L^2} = \inprod{\bu^{\SE}_{\ell}}{\bu_{\ell'}^{\SE,\perp}}_{L^2}$ for all $\ell, \ell' \leq k$. Putting together the pieces, we have
\begin{align*}
\inprod{\bg^{\SE}_{\ell}}{\bg^{\SE}_{k+1}}_{L^2}  = \sum_{\ell' = 1}^\ell \inprod{\bu^{\SE}_{\ell}}{\bu_{\ell'}^{\SE,\perp}}_{L^2} \cdot \inprod{\bu_{\ell'}^{\SE,\perp}}{\bu^{\SE}_{k+1} }_{L^2} \overset{\1}{=}  \inprod{\bu^{\SE}_{\ell}}{\bu^{\SE}_{k+1} }_{L^2} \quad \text{ for all } \ell \leq k,
\end{align*}
where step $\1$ follows from the orthonormality of the vectors $\{ \bu_{\ell'}^{\SE,\perp} \}_{\ell' = 1}^\ell$ and the fact that $\bu^{\SE}_{\ell}$ is in the span of these vectors. At the same time, we have
\begin{align}
\inprod{\sg[\butilde]}{\sg[\butilde]}_{L^2} &=  \Bigl\| \sum_{\ell'=1}^{k} \< \bu_{\ell'}^{\SE,\perp} , \butilde \>_{L^2} \cdot \bg_{\ell'}^{\SE,\perp}
	+
	\|\proj_{\bU_{k}^{\SE}}^\perp \butilde \|_{L^2} \bxi_g \Bigr\|_{L^2}^2 \notag \\
	&\overset{\1}{=} \sum_{\ell'=1}^{k} \< \bu_{\ell'}^{\SE,\perp} , \butilde \>^2_{L^2} + \| \proj_{\bU_{k}^{\SE}}^\perp \butilde \|_{L^2}^2 = \| \butilde \|_{L^2}^2 = \| \bu^{\SE}_{k+1} \|_{L^2}^2, \label{eq:SE-verify-second}
\end{align}
where in step $\1$ we have used the orthonormality of the collection $\{ \bg_{\ell'}^{\SE,\perp} \}_{\ell' = 1}^{k}, \bxi_g$.
Combining Eqs.~\eqref{eq:SE-verify-first} and~\eqref{eq:SE-verify-second} verifies the induction hypothesis Eq.~\eqref{eq:fix-pt-HO-at-t}(a). Proceeding in an identical manner with $(\bH^{\SE}_{k+1}, \bV^{\SE}_{k+1})$ instead of $(\bG^{\SE}_{k+1}, \bU^{\SE}_{k+1})$ verifies the induction hypothesis Eq.~\eqref{eq:fix-pt-HO-at-t}(b), since the collection $\{ \bh_{\ell'}^{\SE,\perp} \}_{\ell' = 1}^{k}, \bxi_h$ is also orthonormal.

To verify~\eqref{eq:fix-pt-HO-at-t}(c), we substitute the definition of $\hbu^{\SE}_{k+1}$ from Eq.~\eqref{eq:uhatse-defn} to see that for any $\ell \leq k$, we have
\begin{align} \label{eq:SE-verify-c-first}
\< \bu_{\ell}^{\SE} , \hbu_t^{\SE} \>_{L^2} = \sum_{\ell'=1}^{k} \< \bg_{\ell'}^{\SE,\perp} , \bv_t^{\SE} \>_{L^2} \< \bu_{\ell'}^{\SE,\perp}, \bu_{\ell}^{\SE}\>_{L^2} \overset{\1}{=} \sum_{\ell'=1}^{k} \< \bg_{\ell'}^{\SE,\perp} , \bv_t^{\SE} \>_{L^2} \< \bg_{\ell'}^{\SE,\perp}, \bg_{\ell}^{\SE}\>_{L^2} \overset{\2}{=} \< \bg_{\ell}^{\SE}, \bv_t^{\SE}  \>_{L^2}.
\end{align}
In step $\1$, we have used the induction hypothesis to write $\< \bu_{\ell'}^{\SE,\perp}, \bu_{\ell}^{\SE}\>_{L^2} = \< \bg_{\ell'}^{\SE,\perp}, \bg_{\ell}^{\SE}\>_{L^2}$ for all $\ell, \ell' \leq k$. On the other hand, step $\2$ follows from the orthonormality of the vectors $\{ \bg_{\ell'}^{\SE,\perp} \}_{\ell' = 1}^\ell$ and the fact that $\bg^{\SE}_{\ell}$ is in the span of these vectors.

At the same time, using the definition of $\hbu^{\SE}_{k+1}$ from Eq.~\eqref{eq:uhatse-defn} and noting that $\butilde = \bu^{\SE}_{k+1}$ and $\bvtilde = \bv^{\SE}_{k+1}$ once again yields
\begin{align} \label{eq:SE-verify-c-second}
\< \bu_{k+1}^{\SE} , \hbu_{k+1}^{\SE} \>_{L^2} &= \sum_{\ell'=1}^{k} \< \bg_{\ell'}^{\SE,\perp} , \bv_{k+1}^{\SE} \>_{L^2} \< \bu_{\ell'}^{\SE,\perp}, \bu_{k+1}^{\SE}\>_{L^2} + \< \bxi_g , \bv^{\SE}_{k+1} \>_{L^2} \frac{\inprod{\bu_{k+1}^{\SE}}{\proj_{\bU_{k}^{\SE}}^\perp \bu_{k+1}^{\SE}}_{L^2}}{\|\proj_{\bU_{k}^{\SE}}^\perp \bu^{\SE}_{k+1} \|_{L^2}} \notag \\
&= \sum_{\ell'=1}^{k} \< \bg_{\ell'}^{\SE,\perp} , \bv_{k+1}^{\SE} \>_{L^2} \< \bu_{\ell'}^{\SE,\perp}, \bu_{k+1}^{\SE}\>_{L^2} + \< \bxi_g , \bv^{\SE}_{k+1} \>_{L^2} \cdot \|\proj_{\bU_{k}^{\SE}}^\perp\bu^{\SE}_{k+1}\|_{L^2} \notag \\
&= \inprod{\sg[\butilde]}{\bv^{\SE}_{k+1}}_{L^2} = \inprod{\bg^{\SE}_{k+1}}{\bv^{\SE}_{k+1}}_{L^2}. 
\end{align}
Combining Eqs.~\eqref{eq:SE-verify-c-first} and~\eqref{eq:SE-verify-c-second} verifies Eq.~\eqref{eq:fix-pt-HO-at-t}(c). Eq.~\eqref{eq:fix-pt-HO-at-t}(d) can be similarly verified via a parallel argument on the term $\< \bv_{\ell}^{\SE} , \hbv_{k+1}^{\SE} \>_{L^2}$.
\qed

\subsubsection{Proof of Proposition~\ref{prop:uniqueness-SE-fixpt}} \label{sec:prop2-proof}

As hypothesized, any solution to the fixed-point equations must have $(\bg^{\SE, *}_{k+1}, \bG^{\SE}_{k})$ jointly Gaussian as well as $(\bh^{\SE, *}_{k+1}, \bH^{\SE}_{k})$ jointly Gaussian, with the entries of these drawn i.i.d. across the row dimension. Define the collections $\{ \bg^{\SE, \perp}_{\ell} \}_{\ell = 1}^{k}$ and $\{ \bh^{\SE, \perp}_{\ell} \}_{\ell = 1}^{k}$ as the orthonormal vectors from before, noting that each of these also has i.i.d. entries.
Owing to joint Gaussianity and the fact that $\bg_{k+1}^{\SE, *}$ must have i.i.d. entries, there must exist scalars $\{L_{k+1, \ell}^g,L_{k+1, \ell}^h \}_{\ell = 1}^{k+1}$ such that $\bg_{k+1}^{\SE, *} = \sum_{\ell = 1}^{k} L^g_{k+1, \ell} \bg_{\ell}^{\SE,\perp} + L^g_{k+1, k+1} \bxi_g$ for an independent Gaussian $\bxi_g \sim \mathsf{N}(0, \bI_d/d)$, and
similarly, $\bh_{k+1}^{\SE, *} = \sum_{\ell = 1}^{k} L^h_{k+1, \ell} \bh_{\ell}^{\SE,\perp} + L^h_{k+1, k+1} \bxi_h$ for $\bxi_h \sim \mathsf{N}(0, \bI_n/n)$.

Any solution to the fixed point equations must also satisfy Eq.~\eqref{eq:span-inclusions-induction-step-HO}, in that we have $\hbv^{\SE, *}_{k+1} \in \spn( \{ \bv^{\SE}_{\ell} \}_{\ell = 1}^{k}, \bv^{\SE, *}_{k+1})$ and $\hbu^{\SE, *}_{k+1} \in \spn(\{ \bu^{\SE}_{\ell} \}_{\ell = 1}^{k}, \bu^{\SE, *}_{k+1} )$. Define the collection $\{ \bu^{\SE, \perp}_{\ell} \}_{\ell = 1}^{k}$ and $\{ \bv^{\SE, \perp}_{\ell} \}_{\ell = 1}^{k}$ as the orthonormal vectors from before---there must exist scalars $\{ \zeta_{k+1, \ell}^u,\zeta_{k+1, \ell}^v \}_{\ell = 1}^{k+1}$ such that $\hbu_{k+1}^{\SE,*} = \sum_{\ell = 1}^{k} \zeta_{k+1, \ell}^u \bu_{\ell}^{\SE, \perp} + \zeta_{k+1, k+1}^u \proj^{\perp}_{\bU^{\SE}_{k}} \bu_{k+1}^{\SE, *}$ and $\hbv_{k+1}^{\SE, *} = \sum_{\ell = 1}^{k} \zeta_{k+1, \ell}^v \bv_{\ell}^{\SE, \perp} + \zeta_{k+1, k+1}^v \proj^{\perp}_{\bV^{\SE}_{k}} \bv_{k+1}^{\SE, *}$. 
The SE equations must hold: Explicitly writing out Eqs.~\eqref{eq:fix-pt-HO-at-t}(a) and (b), we must have that 
\begin{subequations} \label{eq:L-zeta}
\begin{align}
L^g_{k+1, \ell} &= \inprod{\bg^{\SE, *}_{k+1}}{\bg^{\SE, \perp}_{\ell}}_{L^2} = \inprod{\bu^{\SE, *}_{k+1}}{\bu^{\SE, \perp}_{\ell}}_{L^2} \quad \text{ and }\nonumber\\
L^h_{k+1, \ell} &= \inprod{\bh^{\SE, *}_{k+1}}{\bh^{\SE, \perp}_{\ell}}_{L^2} = \inprod{\bv^{\SE, *}_{k+1}}{\bv^{\SE, \perp}_{\ell}}_{L^2} \quad \text{ for all } \ell \leq k; \\
L^g_{k+1, k+1} &= \inprod{\bg^{\SE, *}_{k+1}}{\bg^{\SE, *}_{k+1}}^{1/2}_{L^2} = \inprod{\bu^{\SE, *}_{k+1}}{\bu^{\SE, *}_{k+1}}^{1/2}_{L^2} = \|\proj_{\bU_{k}^{\SE}}^\perp\bu^{\SE, *}_{k+1} \|_{L^2} \quad \text{ and }\nonumber\\
L^h_{k+1, k+1} &= \inprod{\bh^{\SE, *}_{k+1}}{\bh^{\SE, *}_{k+1}}^{1/2}_{L^2} = \inprod{\bv^{\SE, *}_{k+1}}{\bv^{\SE, *}_{k+1}}^{1/2}_{L^2} = \|\proj_{\bV_{k}^{\SE}}^\perp\bv^{\SE, *}_{k+1}\|_{L^2}.
\end{align}
Similarly, writing out Eqs.~\eqref{eq:fix-pt-HO-at-t}(c) and (d), we must have
\begin{align}
\zeta_{k+1, \ell}^u = \inprod{\hbu^{\SE, *}_{k+1}}{\bu^{\SE, \perp}_{\ell}}_{L^2} = \inprod{\bv^{\SE, *}_{k+1}}{\bg^{\SE, \perp}_{\ell}}_{L^2} \;\; \text{ and } \;
&\zeta_{k+1, \ell}^v = \inprod{\hbv^{\SE, *}_{k+1}}{\bv^{\SE, \perp}_{\ell}}_{L^2} = \sqrt{\Lambda} \inprod{\bu^{\SE, *}_{k+1}}{\bh^{\SE, \perp}_{\ell}}_{L^2} \text{ for all } \ell \leq k; 
\end{align}
as well as
\begin{align}
\zeta_{k+1, k+1}^u &= \frac{\inprod{\hbu^{\SE, *}_{k+1}}{\proj^{\perp}_{\bU^{\SE}_{k}}\bu^{\SE, *}_{k+1}}_{L^2}}{\|\proj_{\bU_{k}^{\SE}}^\perp\bu^{\SE, *}_{k+1}\|_{L^2}^2} = \frac{\inprod{\bv^{\SE, *}_{k+1}}{\proj^{\perp}_{\bG^{\SE}_{k}} \bg^{\SE, *}_{k+1}}_{L^2}}{\|\proj_{\bU_{k}^{\SE}}^\perp\bu^{\SE, *}_{k+1}\|_{L^2}^2} \;\; \text{ and } \nonumber\\
\zeta_{k+1, k+1}^v &= \frac{\inprod{\hbv^{\SE, *}_{k+1}}{\proj^{\perp}_{\bV^{\SE}_{k}} \bv^{\SE, *}_{k+1}}_{L^2}}{\|\proj_{\bV_{k}^{\SE}}^\perp\bv^{\SE, *}_{k+1}\|_{L^2}^2} = \sqrt{\Lambda} \frac{\inprod{\bu^{\SE, *}_{k+1}}{\proj^{\perp}_{\bH^{\SE}_{k}} \bh^{\SE, *}_{k+1}}_{L^2}}{\|\proj_{\bV_{k}^{\SE}}^\perp\bv^{\SE, *}_{k+1}\|_{L^2}^2}.
\end{align}
\end{subequations}

Finally, a fixed-point solution $(\bu^{\SE, *}_{k+1}, \bv^{\SE, *}_{k+1})$ must also satisfy Eq.~\eqref{eq:KKT-cond-induction-step}, in that
\begin{align*}
\sqrt{\Lambda} \left( \sum_{\ell = 1}^{k} L^h_{k+1,\ell} \bh_{\ell}^{\SE,\perp} + L^h_{k+1, k+1} \bxi_h \right) - \sum_{\ell = 1}^{k} \zeta_{k+1, \ell}^u \bu_{\ell}^{\SE, \perp} - \zeta_{k+1, k+1}^u \proj^{\perp}_{\bU^{\SE}_{k}} \bu_{k+1}^{\SE, *} &= \nabla (\phi^u_{k+1})(\bu^{\SE, *}_{k+1}; \Use_{k}) \text{ and } \\
\sum_{\ell = 1}^{k} L^g_{k+1, \ell} \bg_{\ell}^{\SE,\perp} + L^g_{k+1, k+1} \bxi_g - \sum_{\ell = 1}^{k} \zeta_{k+1, \ell}^v \bv_{\ell}^{\SE, \perp} - \zeta_{k+1, k+1}^v \proj^{\perp}_{\bV^{\SE}_{k}} \bv_{k+1}^{\SE, *} &= \nabla (\phi^v_{k+1})(\bv^{\SE, *}_{k+1}; \Vse_{k}). 
\end{align*}
Substituting Eq.~\eqref{eq:L-zeta} above and noting that $\proj^{\perp}_{\bG^{\SE}_{k}} \bg^{\SE, *}_{k+1} = L^g_{k+1, k+1} \bxi_g = \| \proj_{\bU_{k}^{\SE}}^\perp\bu^{\SE, *}_{k+1} \|_{L^2} \cdot \bxi_g$, the first display reads as 
\begin{subequations} \label{eq:final-solution-constructed}
\begin{align}
&\sqrt{\Lambda} \Bigg( \underbrace{\sum_{\ell = 1}^{k} \inprod{\bv^{\SE, \perp}_{\ell}}{\bv^{\SE, *}_{k+1}}_{L^2} \bh_{\ell}^{\SE,\perp} + \| \proj_{\bV^{\SE}_{k}} \bv^{\SE, *}_{k+1}\|_{L^2} \cdot \bxi_h}_{\sh[\bv^{\SE, *}_{k+1}]} \Bigg) \nonumber\\
& \qquad \qquad - \sum_{\ell = 1}^{k} \inprod{\bg^{\SE, \perp}_{\ell}}{\bv^{\SE, *}_{k+1}}_{L^2}  \bu_{\ell}^{\SE, \perp} - 
		\< \bxi_g , \bv^{\SE, *}_{k+1} \>_{L^2} \frac{\proj_{\bU_{k}^{\SE}}^\perp\bu^{\SE, *}_t}{\|\proj_{\bU_{k}^{\SE}}^\perp\bu^{\SE, *}_{k+1}\|_{L^2}} = \nabla (\phi^u_{k+1})(\bu^{\SE, *}_{k+1}; \Use_{k}).
\end{align}
Proceeding in an identical manner for the second display after noting that 
\[
\proj^{\perp}_{\bH^{\SE}_{k}} \bh^{\SE, *}_{k+1} = L^h_{k+1, k+1} \bxi_h = \| \proj_{\bV_{k}^{\SE}}^\perp\bv^{\SE, *}_{k+1} \|_{L^2} \cdot \bxi_h,
\]
we have
\begin{align}
\underbrace{\sum_{\ell = 1}^{k} \inprod{\bu^{\SE, \perp}_{\ell}}{\bu^{\SE, *}_{k+1}}_{L^2} \bg_{\ell}^{\SE,\perp} + \| \proj_{\bU^{\SE}_{k}} \bu^{\SE, *}_{k+1} \|_{L^2} \cdot \bxi_g}_{\sg[\bu^{\SE}_{k+1}]} - \sqrt{\Lambda} \sum_{\ell = 1}^{k} &\inprod{\bh^{\SE, \perp}_{\ell}}{\bu^{\SE, *}_{k+1}}_{L^2}  \bv_{\ell}^{\SE, \perp} - \sqrt{\Lambda}
		\< \bxi_h , \bu^{\SE, *}_{k+1}\>_{L^2} \frac{\proj_{\bV_{k}^{\SE}}^\perp\bv^{\SE, *}_{k+1}}{\|\proj_{\bV_{k}^{\SE}}^\perp\bv^{\SE, *}_{k+1}\|_{L^2}} \nonumber\\
		&= \nabla (\phi^v_{k+1})(\bv^{\SE, *}_{k+1}; \Vse_{k}).
\end{align}
\end{subequations}
Comparing Eq.~\eqref{eq:final-solution-constructed} with Eq.~\eqref{eq:unique-KKT-saddle} and applying Lemma~\ref{lem:Hilbert-saddle-lemma}, we see that $(\bv^{\SE, *}_{k+1}, \bu^{\SE, *}_{k+1})$ is a solution to the Hilbert saddle problem~\eqref{eq:L2-inductive-saddle-new}. The construction of relations above confirm how the tuple $(\bg^{\SE, *}_{k+1}, \bh^{\SE, *}_{k+1}, \hbv^{\SE, *}_{k+1}, \hbu^{\SE, *}_{k+1})$ can be obtained from $(\bv^{\SE, *}_{k+1}, \bu^{\SE, *}_{k+1})$ via Eq.~\eqref{eq:construction-candidate-fixedpoint}.
\qed


\section{Proof of Theorem \ref{thm:exact-asymptotics}: Finite sample deviation around SE}\label{sec:proof-exact-asymptotics}

Let us begin by recalling some definitions for convenience. Recall from Eq.~\eqref{eq:PO-g-h-def} that
\[
\bg_{k} = \sum_{\ell=1}^{k} \bigl(L^{v}_{k}\bigr)_{k, \ell}\bv_{\ell} - \bX^{\top} \bu_{k} \quad \text{and} \quad \sqrt{\Lambda} \cdot \bh_{k} :=  \sum_{\ell = 1}^{k} \bigl(L^{u}_{k}\bigr)_{k, \ell}\bu_{\ell} + \bX \bv_k,
\]
so that $\bX^{\top} \bu_{k} =  \sum_{\ell=1}^{k} \bigl(L^{v}_{k}\bigr)_{k, \ell}\bv_{\ell}  - \bg_{k}$ and similarly $\bX \bv_k = \sqrt{\Lambda} \cdot \bh_k - \sum_{\ell=1}^{k} \bigl(L^{u}_{k}\bigr)_{k, \ell}\bu_{\ell}$. Also, since $\Gse_{k}$ consists of jointly Gaussian vectors (and analogously $\Hse_{k}$), we have the decomposition $\gse_{k} = \sum_{\ell = 1}^{k - 1} \alpha_{k,\ell} \gse_{\ell} + \alpha_{k,k} \bxi_g$ and $\hse_{k} = \sum_{\ell = 1}^{k - 1} \beta_{k,\ell} \hse_{\ell} + \beta_{k, k} \bxi_h$ for deterministic scalars $\{\alpha_{k ,\ell} \}_{\ell=1}^{k}$ and $\{ \beta_{k ,\ell} \}_{\ell=1}^{k}$ constructible from $\bK^g_k$ and $\bK^h_k$, respectively, as well as standard Gaussians $\bxi_g \sim \normal(0, \bI_d/d)$ and $\bxi_h \sim \normal(0, \bI_n/n)$ chosen independently of the past. 
Also recall the pseudo-Lipschitz functions $\psi_d$ and $\psi_n$ from the statement of the theorem.

Let 
\[
\mathcal{V}^{\mathrm{SE}}_k := \sigma\bigl(\{\Vse_k, \Gse_k\}\bigr) \quad \text{ and } \quad \mathcal{U}^{\mathrm{SE}}_k := \sigma\bigl(\{\Use_k, \Hse_k\}\bigr)
\]
denote the $\sigma$-algebras generated by the state evolution at time $k$. The following structural lemma regarding the state evolution is also a key cog that will be used in the proof.
\begin{lemma}
	\label{lem:conditional-expectation-PL}
	Let $\tau^{v}: \mathbb{R}^{d \times T} \times \mathbb{R}^{d \times T}$ and $\tau^{u}: \mathbb{R}^{n \times T} \times \mathbb{R}^{n \times T}$ be order-$2$ pseudo-Lipschitz with constant $L$.  For each $k \in \mathbb{N}$ such that $k < T$, there exists a $\mathcal{V}_{T-k}^{\mathrm{SE}}$-measurable function $\zeta^{v}_{k}: \mathbb{R}^{d \times (T - k)} \times \mathbb{R}^{d \times (T - k)}$ and  a $\mathcal{U}_{T-k}^{\mathrm{SE}}$-measurable function $\zeta^{u}_{k}: \mathbb{R}^{n \times (T - k)} \times \mathbb{R}^{n \times (T - k)}$ such that
	\begin{align*}
		\zeta_{k}^{v}(\bA, \bB) &= \mathbb{E}\bigl[\tau^v(\Vse_T, \Gse_T) \mid \Vse_{T-k} = \bA, \Gse_{T-k} = \bB\bigr], \;\; \text{ a.s., }\\
		\zeta_{k}^{u}(\bC, \bD) &= \mathbb{E}\bigl[\tau^u(\Use_T, \Hse_T) \mid \Use_{T-k} = \bC, \Hse_{T-k} = \bD\bigr], \;\; \text{ a.s., }
	\end{align*}
	and $\zeta_{k}^v, \zeta_{k}^{u}$ are order-$2$ pseudo-Lipschitz with constant $C_{\mathrm{SE}, T} > 0$ which depends only on the state evolution and the total number of iterations $T$. 
\end{lemma}
Lemma~\ref{lem:conditional-expectation-PL} is proved in Section~\ref{sec:proof-PL-SE}.

With this technical lemma in hand, we begin by postulating a stronger induction hypothesis in Proposition~\ref{prop:events-induction}, from which we prove Theorem~\ref{thm:exact-asymptotics}. 
We first define a collection of relevant events. We begin by defining the scalar $\Delta_0(\delta) := 1 + \sqrt{\Lambda} + \sqrt{\frac{2\Lambda\bigl[T\log(3 + 6T) + \log(1/\delta)\bigr]}{n }}$ and the sequence of scalars
\begin{align}
	 \Delta_{1}(k,\delta) &:= C_{\mathrm{SE}} \bigl(k!\bigr)^2 \cdot \Bigl(\frac{T\log(3 + 6T) + \log(\frac{1}{\delta})}{n}\Bigr)^{2^{-k}} \quad \text{ for } \quad k \in [T],
\end{align}
where $C_{\mathrm{SE}}$ is constant depending purely on the state evolution quantities.

Our events are then defined as:
\begin{align*}
 	\Omega_{op} &:= \Bigl\{ \| \bX \|_{\mathrm{op}} \leq \Delta_0(\delta) \Bigr\},\\
	\Omega_k^{vv} &:= \Bigl\{ \bigl\| \mathbb{E}\bigl[\llangle \Vse_T, \Vse_{T} \rrangle \mid \Vse_{k} = \bV_{k}, \Gse_{k} = \bG_k\bigr] - \llangle \Vse_T, \Vse_T \rrangle_{L^2} \bigr\|_{\infty} \leq \Delta_{1}(k, \delta)\Bigr\},\\
	\Omega_k^{gv} &:= \Bigl\{ \bigl\| \mathbb{E}\bigl[\llangle \Gse_T, \Vse_{T} \rrangle \mid \Vse_{k} = \bV_{k}, \Gse_{k} = \bG_k\bigr] - \llangle \Gse_T, \Vse_T \rrangle_{L^2} \bigr\|_{\infty} \leq \Delta_{1}(k, \delta)\Bigr\},\\
	\Omega_k^{gg} &:= \Bigl\{ \bigl\| \mathbb{E}\bigl[\llangle \Gse_T, \Gse_{T} \rrangle \mid \Vse_{k} = \bV_{k}, \Gse_{k} = \bG_k\bigr] - \llangle \Gse_T, \Gse_T \rrangle_{L^2} \bigr\|_{\infty} \leq \Delta_{1}(k, \delta)\Bigr\},\\
	\Omega_k^{uu} &:= \Bigl\{ \bigl\| \mathbb{E}\bigl[\llangle \Use_T, \Use_{T} \rrangle \mid \Use_{k} = \bU_{k}, \Hse_{k} = \bH_k\bigr] - \llangle \Use_T, \Use_T \rrangle_{L^2} \bigr\|_{\infty} \leq \Delta_{1}(k, \delta)\Bigr\},\\
	\Omega_k^{hu} &:= \Bigl\{ \bigl\| \mathbb{E}\bigl[\llangle \Hse_T, \Use_{T} \rrangle \mid \Use_{k} = \bU_{k}, \Hse_{k} = \bH_k\bigr] - \llangle \Hse_T, \Use_T \rrangle_{L^2} \bigr\|_{\infty} \leq \Delta_{1}(k, \delta)\Bigr\}, \\
	\Omega_k^{hh} &:= \Bigl\{ \bigl\| \mathbb{E}\bigl[\llangle \Hse_T, \Hse_{T} \rrangle \mid \Use_{k} = \bU_{k}, \Hse_{k} = \bH_k\bigr] - \llangle \Hse_T, \Hse_T \rrangle_{L^2} \bigr\|_{\infty} \leq \Delta_{1}(k, \delta)\Bigr\}, \\
	\Omega_{k}^{\psi_d} &:= \Bigl\{ \bigl \lvert \mathbb{E}\bigl[ \psi_d(\Vse_T, \Gse_T) \mid \Vse_k = \bV_k, \Gse_k = \bG_k\bigr] - \mathbb{E}\bigl[\psi_d(\Vse_T, \Gse_T)\bigr] \bigr \rvert \leq \Delta_1(k, \delta)\Bigr\}, \\
	\Omega_{k}^{\psi_n} &:= \Bigl\{ \bigl \lvert \mathbb{E}\bigl[ \psi_n(\Use_T, \Hse_T) \mid \Use_k = \bU_k, \Hse_k = \bH_k\bigr] - \mathbb{E}\bigl[\psi_n(\Use_T, \Hse_T)\bigr] \bigr \rvert \leq \Delta_1(k, \delta)\Bigr\},\\
	\Omega_k^{\phi^v} &:= \Bigl \{\max_{\ell \in T} \Bigl \lvert \mathbb{E}\bigl[\phi_{\ell}^{v}(\vse_{\ell}) \mid \Vse_{k} = \bV_k, \Gse_k = \bG_k\bigr] - \EE\bigl[\phi_{\ell}^{v}(\vse_{\ell})\bigr] \Bigr \rvert \leq \Delta_1(k, \delta)\Bigr\},\\\
		\Omega_k^{\phi^u} &:= \Bigl \{\max_{\ell \in T} \Bigl \lvert \mathbb{E}\bigl[\phi_{\ell}^{u}(\use_{\ell}) \mid \Use_{k} = \bU_k, \Hse_k = \bH_k\bigr] - \EE\bigl[\phi_{\ell}^{u}(\use_{\ell})\bigr] \Bigr \rvert \leq \Delta_1(k, \delta)\Bigr\}
\end{align*}
Note that Lemma~\ref{lem:conditional-expectation-PL} guarantees that the conditional expectations above are all well-defined measurable functions.

Our aforementioned stronger claim is then given by the following proposition.
\begin{proposition} \label{prop:events-induction}
With the events defined above, let $\Omega_k = \Omega^{op} \cap \Omega_k^{vv} \cap \Omega_k^{gv} \cap  \Omega_k^{gg} \cap \Omega_k^{uu} \cap \Omega_k^{hu} \cap \Omega_k^{hh} \cap \Omega_k^{\psi_d} \cap \Omega_k^{\psi_n} \cap \Omega_k^{\phi^u} \cap \Omega_k^{\phi^v}$. Then for each $k \in [T]$, we have
\[
\mathbb{P}(\Omega_k^c) \leq \frac{2\delta \cdot 3^{k-1}}{3^{T} - 1}.
\]
\end{proposition}
Let us confirm that Proposition~\ref{prop:events-induction} implies the claimed theorem. 
By definition, on event $\Omega_T$, we have
\begin{align*}
	&\bigl \lvert \psi_d(\bV_T, \bG_T)  - \mathbb{E}\bigl[\psi_d(\Vse_T, \Gse_T)\bigr] \bigr \rvert \leq \Delta_1(T, \delta) \quad \text{ and } \quad \bigl \lvert \psi_n(\bU_T, \bH_T)  - \mathbb{E}\bigl[\psi_n(\Use_T, \Hse_T)\bigr] \bigr \rvert \leq \Delta_1(T,\delta).
\end{align*}
Additionally, the event $\Omega_k$ occurs with probability at least $1 - \delta$ for each $k \in [T]$, by Proposition~\ref{prop:events-induction}. Applying the proposition for $k = T$ thus proves Theorem~\ref{thm:exact-asymptotics}. 

We are now ready to proceed to the proof of Proposition~\ref{prop:events-induction}, by induction.

\subsection{Base case: $k = 1$}
First note that by, e.g.,~\citet[][Theorem 6.1]{wainwright2019high}, the event $\Omega_{op}$ holds with probability at least $1-\delta/\{T \log(3 + 6T)\}$.  
By Eq.~\eqref{eq:SE-base-case}, we have $\use_{1} = \bu_1$ and $\vse_1 = \bv_1$. Furthermore, since $\vhatse_1 = 0$ and $\uhatse_1 = 0$, we have $L^{v}_1 = L^u_1 = 0$. Consequently, we have
$\bg_1 = - \bX^{\top} \bu_1$ and $\sqrt{\Lambda} \cdot \bh_1 = \bX \bv_1$.
 Since $\bu_1$ and $\bv_1$ are chosen independently of $\bX$, we have $\bg_1 \sim \normal(0, \frac{\| \bu_1 \|_2^2}{p} \bI_p)$ and $\bh_1 \sim \normal(0, \frac{\| \bv_1 \|_2^2}{n} \bI_n)$. Thus, we have $(\bg_1, \bv_1, \bh_1, \bu_1) \overset{(d)}{=} (\gse_1, \vse_1, \hse_1, \use_1)$.  
We establish that $\mathbb{P}(\Omega_1^{vv}) \geq 1 - \frac{\delta}{10(3^{T} - 1)}$ and , noting that similar arguments can be used to control the probability of the other events by the same value. Taking a union bound then proves the base case. 

Since $\vse_1 = \bv_1$ by construction, conditioning on the event $\{ \vse_1 = \bv_1 \}$ is trivial and we have
\begin{align} \label{eq:base1}
\mathbb{E}\bigl[\llangle \Vse_T, \Vse_{T} \rrangle \mid \vse_{1} = \bv_{1}, \gse_{1} = \bg_1 \bigr] = \mathbb{E}\bigl[\llangle \Vse_T, \Vse_{T} \rrangle \mid \gse_{1} = \bg_1 \bigr].
\end{align}
But each entry of the matrix on the RHS, by Lemma~\ref{lem:conditional-expectation-PL}, is an order-2 pseudo-Lipschitz function of $\bg_1$. Thus, for each each $1 \leq \ell \leq \ell' \leq T$, 
Lemma~\ref{lem:pseudo-Lipschitz-concentration} applied in conjunction with Bernstein's inequality~\citep[see, e.g.,][Theorem 2.8.1]{vershynin2018high} yields
\begin{align} \label{eq:base2}
\left| \EE[ \inprod{\vse_\ell}{\vse_{\ell'}} \mid \gse_{1} = \bg_{1} ] - \EE[\inprod{\vse_\ell}{\vse_{\ell'}}] \right| \leq \Delta_1(1, \delta) \quad \text{ with probability } \quad \geq 1 - \frac{\delta}{10T^2\log(3 + 6T)}.
\end{align}
Putting together Eqs.~\eqref{eq:base1} and~\eqref{eq:base2} with a union bound over the $\binom{T}{2}$ entries in the matrix, we obtain that with probability at least $1 - \frac{\delta}{10T^2\log(3 + 6T)}$, the following deviation bound holds:
\[
\bigl\| \mathbb{E}\bigl[\llangle \Vse_T, \Vse_{T} \rrangle \mid \Vse_{k} = \bV_{k}, \Gse_{k} = \bG_k\bigr] - \llangle \Vse_T, \Vse_T \rrangle_{L^2} \bigr\|_{\infty} \leq \Delta_{1}(1, \delta).
\]
This completes the proof of the base case.

\subsection{Induction hypothesis}

We now explicitly state the induction hypothesis and collect some facts that are guaranteed by it.
The induction hypothesis is that $\Omega_{k'}$ occurs with high probability for all $1 \leq k' \leq k$.
In particular, on $\cap_{k' = 1}^k \Omega_{k'}$, we have (among other bounds) that for all $k' \leq k$:
\begin{subequations} \label{eq:ind-events-four}
\begin{align}
\max_{\ell, \ell' \in [T]} \left| \EE[ \inprod{\vse_\ell}{\vse_{\ell'}} \mid \Vse_{k'} = \bV_{k'}, \Gse_{k'} = \bG_{k'} ] - \EE[\inprod{\vse_\ell}{\vse_{\ell'}}] \right| &\leq \Delta_1(k', \delta), \label{eq:ind-events-four-1} \\
\max_{\ell, \ell' \in [T]} \left| \EE[ \inprod{\gse_\ell}{\vse_{\ell'}} \mid \Vse_{k'} = \bV_{k'}, \Gse_{k'} = \bG_{k'} ] - \EE[\inprod{\gse_\ell}{\vse_{\ell'}}] \right| &\leq \Delta_1(k', \delta),  \label{eq:ind-events-four-2} \\
\max_{\ell, \ell' \in [T]} \left| \EE[ \inprod{\gse_\ell}{\gse_{\ell'}} \mid \Vse_{k'} = \bV_{k'}, \Gse_{k'} = \bG_{k'} ] - \EE[\inprod{\gse_\ell}{\gse_{\ell'}}] \right| &\leq \Delta_1(k', \delta),  \label{eq:ind-events-four-gg} \\
\max_{\ell, \ell' \in [T]} \left| \EE[ \inprod{\use_\ell}{\use_{\ell'}} \mid \Use_{k'} = \bU_{k'}, \Hse_{k'} = \bH_{k'} ] - \EE[\inprod{\use_\ell}{\use_{\ell'}}] \right| &\leq \Delta_1(k', \delta), \label{eq:ind-events-four-3} \\
\max_{\ell, \ell' \in [T]} \left| \EE[ \inprod{\hse_\ell}{\use_{\ell'}} \mid \Use_{k'} = \bU_{k'}, \Hse_{k'} = \bH_{k'} ] - \EE[\inprod{\hse_\ell}{\use_{\ell'}}] \right| &\leq \Delta_1(k', \delta), \label{eq:ind-events-four-4} \\
\max_{\ell, \ell' \in [T]} \left| \EE[ \inprod{\hse_\ell}{\hse_{\ell'}} \mid \Use_{k'} = \bU_{k'}, \Hse_{k'} = \bH_{k'} ] - \EE[\inprod{\hse_\ell}{\hse_{\ell'}}] \right| &\leq \Delta_1(k', \delta), \label{eq:ind-events-four-hh} \\
\bigl \lvert \mathbb{E}\bigl[ \psi_d(\Vse_T, \Gse_T) \mid \Vse_{k'} = \bV_{k'}, \Gse_{k'} = \bG_{k'}\bigr] - \mathbb{E}\bigl[\psi_d(\Vse_T, \Gse_T)\bigr] \bigr \rvert &\leq \Delta_1(k', \delta), \label{eq:ind-events-four-psid}\\
\bigl \lvert \mathbb{E}\bigl[ \psi_n(\Use_T, \Hse_T) \mid \Use_{k'} = \bU_{k'}, \Hse_{k'} = \bH_{k'}\bigr] - \mathbb{E}\bigl[\psi_n(\Use_T, \Hse_T)\bigr] \bigr \rvert &\leq \Delta_1(k', \delta), \label{eq:ind-events-four-psin}\\
\max_{\ell \in T} \Bigl \lvert \mathbb{E}\bigl[\phi_{\ell}^{v}(\vse_{\ell}) \mid \Vse_{k'} = \bV_{k'}, \Gse_{k'} = \bG_{k'}\bigr] - \EE\bigl[\phi_{\ell}^{v}(\vse_{\ell})\bigr] \Bigr \rvert &\leq \Delta_1(k', \delta), \label{eq:ind-events-four-phiv}\\
\max_{\ell \in T} \Bigl \lvert \mathbb{E}\bigl[\phi_{\ell}^{u}(\use_{\ell}) \mid \Use_{k'} = \bU_{k'}, \Hse_{k'} = \bH_{k'}\bigr] - \EE\bigl[\phi_{\ell}^{u}(\use_{\ell})\bigr] \Bigr \rvert &\leq \Delta_1(k', \delta) \label{eq:ind-events-four-phiu}.
\end{align}
\end{subequations}
Setting $\ell = \ell' = k'$ above, we obtain that on $\Omega_k$
\[
|\| \bu_{k'} \|_2^2 - \EE[ \| \use_{k'} \|_2^2]| \;\vee\; |\| \bv_{k'} \|_2^2 - \EE[ \| \vse_{k'} \|_2^2]| \;\vee\; |\| \bg_{k'} \|_2^2 - \EE[ \| \gse_{k'} \|_2^2]| \;\vee\; |\| \bh_{k'} \|_2^2 - \EE[ \| \hse_{k'} \|_2^2]| \leq \Delta_1(k', \delta).
\]
Since the state evolution quantities themselves are bounded and we have $n$ sufficiently large, this implies that on $\Omega_k$,
\begin{align} \label{eq:bounded-norms}
\max( \| \bu_{k'} \|^2_2, \| \bv_{k'} \|^2_2, \| \bg_{k'} \|^2_2, \| \bh_{k'} \|^2_2, \| \use_{k'} \|^2_2, \| \vse_{k'} \|^2_2, \| \gse_{k'} \|^2_2, \| \hse_{k'} \|^2_2) \lesssim 1 \quad \text{for all } k' \in [k]. 
\end{align}

Let us now turn to listing a key result that is guaranteed on event $\Omega_k$.
The following useful lemma that decomposes $\bX$ at iteration $k$, allowing us to approximate the projection of the data matrix $\bX$ onto the subspace spanned by the previous iterates. The lemma is proved in Section~\ref{sec:helper-thm2}. Define the pair of matrices 
\begin{align} \label{eq:def-Tg-Th}
	\bT_g := \bG_k \bigl(\bK^g_k\bigr)^{-1} (\bU_k)^{\top} \qquad \text{ and } \qquad \bT_h := \sqrt{\Lambda} \cdot \bH_k \bigl(\bK^h_k\bigr)^{-1} (\bV_k )^{\top}.
\end{align}
\begin{lemma}
	\label{lem:X-decomp}
	Consider the random matrix $\bX^\parallel = \bX - \proj^{\perp}_{\bU_k} \bX \proj_{\bV_k}^{\perp}$.  There exists a constant $C_{\mathrm{SE}} > 0$, depending only on the state evolution quantities, such that on the event $\Omega_k$, we have
	\[
	\Big\| \bX^\| - (-\bT_g^\top + \bT_h) \Big\|_{\op} \leq C_{\mathrm{SE}} \cdot k \cdot \Delta_1(k, \delta),
	\]
	where $\bT_g$ and $\bT_h$ are defined in Eq.~\eqref{eq:def-Tg-Th}.
	Moreover, if $\widetilde{\bX}$ denotes an independent copy of the data matrix $\bX$, then $\bX \; \vert \; \bU_{k}, \bV_{k} \overset{d}{=} \bX^{\parallel} + \proj^{\perp}_{\bU_k} \widetilde{\bX} \proj_{\bV_k}^{\perp}$.
\end{lemma}

\subsection{Proof of induction step for first order update} \label{sec:proof-exact-asymptotics-iterative}
We first require a preliminary lemma.  Recall the coefficient matrices $(\bL^u_\ell)_{\ell = 1}^T, (\bL^v_\ell)_{\ell = 1}^T$ that define expansions of $(\uhatse_\ell, \vhatse_\ell)$, and through them $(\bh_\ell, \bg_\ell)$. Also recall the scalars $(\alpha_{\ell, \ell'}, \beta_{\ell, \ell'})_{1 \leq \ell' \leq \ell \leq T}$ defining expansions of $(\gse_\ell, \hse_{\ell})_{\ell = 1}^T$. The following fact guarantees some useful relations between these coefficients, the matrices $\bT_g, \bT_h$, and the iterates.

\begin{lemma} \label{lem:facts}
	On the event $\Omega_k$, we have
	\begin{align}
		&\Bigl \| \sum_{\ell=1}^{k}\bigl(L^{u}_{k+1}\bigr)_{k+1, \ell} \bu_{\ell} -\bT_g^{\top} \bv_{k+1} \Bigr \|_2 \overset{(a)}{\leq} C_{\mathrm{SE}} k \Delta_1(k, \delta), \;\; &\Bigl \| \sum_{\ell=1}^{k}\bigl(L^{v}_{k+1}\bigr)_{k+1, \ell} \bv_{\ell} -\bT_h^{\top} \bu_{k+1} \Bigr \|_2 \overset{(b)}{\leq} C_{\mathrm{SE}} k \Delta_1(k, \delta), \label{eq:facts1} \\
		&\Bigl \| \sqrt{\Lambda} \cdot \sum_{\ell=1}^{k} \beta_{k+1, \ell} \bh_{\ell} - \bT_h \bv_{k+1} \Bigr \|_2 \overset{(a)}{\leq} C_{\mathrm{SE}} k \Delta_1(k, \delta), \quad &\Bigl \| \sum_{\ell=1}^{k} \alpha_{k+1, \ell} \bg_{\ell} - \bT_g \bu_{k+1} \Bigr \|_2 \overset{(b)}{\leq} C_{\mathrm{SE}} k \Delta_1(k, \delta), \label{eq:facts2} \\
		&\Bigl \lvert \| \proj_{\bV_k}^{\perp} \bv_{k+1}  \|_2 - \beta_{k+1, k + 1} \Bigr \rvert \overset{(a)}{\leq} C_{\mathrm{SE}} k^2\Delta_1(k, \delta), \quad &\Bigl \lvert \| \proj_{\bU_k}^{\perp} \bu_{k+1}  \|_2 - \alpha_{k+1, k + 1} \Bigr \rvert \overset{(b)}{\leq} C_{\mathrm{SE}} k^2\Delta_1(k, \delta). \label{eq:facts3}
	\end{align}
\end{lemma}
We must show that $\Omega_{k+1}$ occurs with high probability when iteration $k$ is a first-order method. In particular, when $(\bu_{k + 1}, \bv_{k + 1})$ are generated by a GFOM, we must show that
\begin{subequations} \label{eq:ind-events-new-four}
\begin{align}
\max_{\ell, \ell' \in [T]} \left| \EE[ \inprod{\vse_\ell}{\vse_{\ell'}} \mid \Vse_{k+1} = \bV_{k+1}, \Gse_{k+1} = \bG_{k+1} ] - \EE[\inprod{\vse_\ell}{\vse_{\ell'}}] \right| &\leq \Delta_1(k+1, \delta), \label{eq:ind-events-new-four-1} \\
\max_{\ell, \ell' \in [T]} \left| \EE[ \inprod{\gse_\ell}{\vse_{\ell'}} \mid \Vse_{k+1} = \bV_{k+1}, \Gse_{k+1} = \bG_{k+1} ] - \EE[\inprod{\gse_\ell}{\vse_{\ell'}}] \right| &\leq \Delta_1(k+1, \delta), \label{eq:ind-events-new-four-2} \\
\max_{\ell, \ell' \in [T]} \left| \EE[ \inprod{\gse_\ell}{\gse_{\ell'}} \mid \Vse_{k+1} = \bV_{k+1}, \Gse_{k+1} = \bG_{k+1} ] - \EE[\inprod{\gse_\ell}{\gse_{\ell'}}] \right| &\leq \Delta_1(k+1, \delta), \label{eq:ind-events-new-four-gg} \\
\max_{\ell, \ell' \in [T]} \left| \EE[ \inprod{\use_\ell}{\use_{\ell'}} \mid \Use_{k+1} = \bU_{k+1}, \Hse_{k+1} = \bH_{k+1} ] - \EE[\inprod{\use_\ell}{\use_{\ell'}}] \right| &\leq \Delta_1(k+1, \delta), \label{eq:ind-events-new-four-3} \\
\max_{\ell, \ell' \in [T]} \left| \EE[ \inprod{\hse_\ell}{\use_{\ell'}} \mid \Use_{k+1} = \bU_{k+1}, \Hse_{k+1} = \bH_{k+1} ] - \EE[\inprod{\hse_\ell}{\use_{\ell'}}] \right| &\leq \Delta_1(k+1, \delta), \label{eq:ind-events-new-four-4} \\
\max_{\ell, \ell' \in [T]} \left| \EE[ \inprod{\hse_\ell}{\hse_{\ell'}} \mid \Use_{k+1} = \bU_{k+1}, \Hse_{k+1} = \bH_{k+1} ] - \EE[\inprod{\hse_\ell}{\hse_{\ell'}}] \right| &\leq \Delta_1(k+1, \delta), \label{eq:ind-events-new-four-hh}\\
\bigl \lvert \mathbb{E}\bigl[ \psi_d(\Vse_T, \Gse_T) \mid \Vse_{k+1} = \bV_{k+1}, \Gse_{k+1} = \bG_{k+1}\bigr] - \mathbb{E}\bigl[\psi_d(\Vse_T, \Gse_T)\bigr] \bigr \rvert &\leq \Delta_1(k+1, \delta), \label{eq:ind-events-new-four-psid}\\
\bigl \lvert \mathbb{E}\bigl[ \psi_n(\Use_T, \Hse_T) \mid \Use_{k+1} = \bU_{k+1}, \Hse_{k+1} = \bH_{k+1}\bigr] - \mathbb{E}\bigl[\psi_n(\Use_T, \Hse_T)\bigr] \bigr \rvert &\leq \Delta_1(k+1, \delta). \label{eq:ind-events-new-four-psin},\\
\max_{\ell \in T} \Bigl \lvert \mathbb{E}\bigl[\phi_{\ell}^{v}(\vse_{\ell}) \mid \Vse_{k+1} = \bV_{k+1}, \Gse_{k+1} = \bG_{k+1}\bigr] - \EE\bigl[\phi_{\ell}^{v}(\vse_{\ell})\bigr] \Bigr \rvert &\leq \Delta_1(k+1, \delta), \label{eq:ind-events-new-four-phiv}\\
\max_{\ell \in T} \Bigl \lvert \mathbb{E}\bigl[\phi_{\ell}^{u}(\use_{\ell}) \mid \Use_{k+1} = \bU_{k+1}, \Hse_{k+1} = \bH_{k+1}\bigr] - \EE\bigl[\phi_{\ell}^{u}(\use_{\ell})\bigr] \Bigr \rvert &\leq \Delta_1(k+1, \delta) \label{eq:ind-events-new-four-phiu}.
\end{align}
\end{subequations}

Fix $\ell, \ell' \in [T]$. By Lemma~\ref{lem:conditional-expectation-PL}, the $\mathbb{R}^{n \times (k + 1)} \times \mathbb{R}^{n \times (k + 1)} \to \mathbb{R}$ maps 
\begin{subequations}\label{eq:tau-definitions}
\begin{align} 
\tau^{uu, k+1}_{\ell, \ell'}: (\bC, \bD) &\mapsto \EE[ \langle \use_{\ell}, \use_{\ell'} \rangle \mid \Use_{k + 1} = \bC, \Hse_{k + 1} = \bD] \\
 \tau^{hu, k+1}_{\ell, \ell'}: (\bC, \bD) &\mapsto \EE[ \langle \hse_{\ell}, \use_{\ell'} \rangle \mid \Use_{k + 1} = \bC, \Hse_{k + 1} = \bD] \\
 \tau^{hh, k+1}_{\ell, \ell'}: (\bC, \bD) &\mapsto \EE[ \langle \hse_{\ell}, \hse_{\ell'} \rangle \mid \Use_{k + 1} = \bC, \Hse_{k + 1} = \bD]\\
 \tau^{\psi_n, k+1}: (\bC, \bD) &\mapsto \EE[\psi_n(\Use_T, \Hse_T) \mid \Use_{k+1} = \bC, \Hse_{k+1} = \bD],\\
 \tau^{\phi^{u}_{\ell}, k+1}: (\bC, \bD) &\mapsto \EE[\phi^{u}_{\ell}(\use_{\ell}) \mid \Use_{k+1} = \bC, \Hse_{k+1} = \bD] \text{ for all } \ell \in [T], 
\end{align}
\end{subequations}
all exist, and are pseudo-Lipschitz. Analogously define the $\mathbb{R}^{d \times (k + 1)} \times \mathbb{R}^{d \times (k + 1)} \to \mathbb{R}$ pseudo-Lipschitz maps $\tau^{vv, k+1}_{\ell, \ell'}, \tau^{gv, k+1}_{\ell, \ell'}, \tau^{gg, k+1}_{\ell, \ell'}, \tau^{\psi_d, k+1}, \tau^{\phi^v_{\ell}, k+1}$.
We establish Eq.~\eqref{eq:ind-events-new-four} by making use of the pseudo-Lipschitz property of these maps. In particular, we carry out the argument for establishing Eq.~\eqref{eq:ind-events-new-four-4}, noting that all the other arguments can be carried out in the same fashion.

Write $\tau^{hu, k+1}_{\ell, \ell'}(\bU_{k + 1}, \bH_{k + 1}) = \tau^{hu, k+1}_{\ell, \ell'}(\bU_{k}, \bu_{k + 1}, \bH_{k}, \bh_{k + 1})$ for convenience. To establish Eq.~\eqref{eq:ind-events-new-four-4}, we must bound the quantity
$\left \lvert \tau^{hu, k+1}_{\ell, \ell'}(\bU_{k}, \bu_{k+1}, \bH_{k}, \bh_{k+1} ) - \mathbb{E} \tau^{hu, k+1}_{\ell, \ell'}(\Use_{k}, \use_{k+1}, \Hse_{k}, \hse_{k+1}) \right \rvert$ by $\Delta_1(k+1, \delta)$ with high probability. Toward that end, we use the triangle inequality to decompose the desired quantity as
\begin{align}
&\left| \EE[ \inprod{\hse_\ell}{\use_{\ell'}} \mid \Use_{k+1} = \bU_{k+1}, \Hse_{k+1} = \bH_{k+1} ] - \EE[\inprod{\hse_\ell}{\use_{\ell'}}] \right| \notag \\
&\quad \leq \left| \EE[ \inprod{\hse_\ell}{\use_{\ell'}} \mid \Use_{k+1} = \bU_{k+1}, \Hse_{k+1} = \bH_{k+1} ] - \EE[ \inprod{\hse_\ell}{\use_{\ell'}} \mid \Use_{k} = \bU_{k}, \Hse_{k} = \bH_{k} ] \right| \notag \\
&\qquad \qquad + \left| \EE[ \inprod{\hse_\ell}{\use_{\ell'}} \mid \Use_{k} = \bU_{k}, \Hse_{k} = \bH_{k} ] - \EE[\inprod{\hse_\ell}{\use_{\ell'}}] \right| \notag \\
&\quad \overset{\1}{\leq}  \left| \EE[ \inprod{\hse_\ell}{\use_{\ell'}} \mid \Use_{k+1} = \bU_{k+1}, \Hse_{k+1} = \bH_{k+1} ] - \EE[ \inprod{\hse_\ell}{\use_{\ell'}} \mid \Use_{k} = \bU_{k}, \Hse_{k} = \bH_{k} ] \right| + \Delta_1(k, \delta), \label{eq:first-bd-tau}
\end{align}
where in step $\1$ we have used the induction hypothesis~\eqref{eq:ind-events-four-4}. 
We now make the following (deterministic) claim regarding state evolution --- we prove the claim at the end of the section.
\begin{claim}
	\label{claim:first-order-hat-value}
	If iteration $k$ is a first-order update, then $(L^{u}_{k+1})_{k+1,k+1} = (L^v_{k+1})_{k+1,k+1} = 0$.
\end{claim}
Note that if the $k$-th step is first-order, then we have 
\[
\bu_{k + 1} = f^u_{k + 1}(\bX \bv_k; \bU_k) \overset{\1}{=} f^u_{k + 1}\Bigl(\sqrt{\Lambda} \cdot \bh_k - \sum_{\ell=1}^{k}\bigl(L^{u}_{k+1}\bigr)_{k+1, \ell} \bu_{\ell}; \bU_k\Bigr)
\]
where step $\1$ follows by definition of $\bh_k$ and as a result of Claim~\ref{claim:first-order-hat-value}.
At the same time, by definition of the state evolution for a GFOM we have
\[
\use_{k + 1} = f^u_{k + 1}(\sqrt{\Lambda} \cdot \hse_k - \uhatse_k; \Use_k) = f^u_{k + 1}\Bigl(\sqrt{\Lambda} \cdot \hse_k - \sum_{\ell=1}^{k}\bigl(L^{u}_{k+1}\bigr)_{k+1, \ell} \use_{\ell}; \Use_k\Bigr).
\]
\begin{align} \label{eq:norms-k+1}
\text{(From these expressions, the following bound can also be verified):} \qquad \qquad  \| \bu_{k + 1} \|_2 \vee \| \use_{k + 1} \|_2 \leq C_{\mathrm{SE}}.
\end{align}

Consequently, we have the equivalence of events $\{\Use_k = \bU_k, \Hse_k = \bH_k \} = \{ \Use_{k + 1} = \bU_{k + 1}, \Hse_k = \bH_k \}$, and thus, 
\begin{align} \label{eq:tau-proof-exp}
\EE[ \inprod{\hse_\ell}{\use_{\ell'}} \mid \Use_{k} = \bU_{k}, \Hse_{k} = \bH_{k} ] = \EE[ \inprod{\hse_\ell}{\use_{\ell'}} \mid \Use_{k + 1} = \bU_{k + 1}, \Hse_{k} = \bH_{k} ].
\end{align}
Note that by definition of $\tau^{hu, k+1}_{\ell, \ell'}$, we also have
\begin{align} \label{eq:tau-proof-term}
\EE[ \inprod{\hse_\ell}{\use_{\ell'}} \mid \Use_{k+1} = \bU_{k+1}, \Hse_{k+1} = \bH_{k+1} ] = \tau^{hu, k+1}_{\ell, \ell'}(\bU_{k}, \bu_{k+1}, \bH_{k}, \bh_{k+1} ).
\end{align}
Putting together Eqs.~\eqref{eq:first-bd-tau},~\eqref{eq:tau-proof-exp} and~\eqref{eq:tau-proof-term}, we have
\begin{align*}
&\mathbb{P} \Bigl( \left| \EE[ \inprod{\hse_\ell}{\use_{\ell'}} \mid \Use_{k+1} = \bU_{k+1}, \Hse_{k+1} = \bH_{k+1} ] - \EE[\inprod{\hse_\ell}{\use_{\ell'}}] \right|  \geq \Delta_1(k+1, \delta) \Bigr) \\
	&\leq \mathbb{P} \Bigl( \Bigl\{ \left \lvert \tau^{hu, k+1}_{\ell, \ell'}(\bU_{k}, \bu_{k+1}, \bH_{k}, \bh_{k+1} ) - \EE[ \inprod{\hse_\ell}{\use_{\ell'}} \mid \Use_{k + 1} = \bU_{k + 1}, \Hse_{k} = \bH_{k} ] \right \rvert \geq \Delta_1(k+1, \delta) - \Delta_1(k, \delta) \Bigr\} \cap \Omega_k\Bigr) \\
	& \qquad \qquad + \frac{2\delta \cdot 3^{k-1}}{3^T - 1}.
\end{align*}

It remains to bound the probability on the RHS.
Recall the quantity $\bX^{\|}$ from Lemma~\ref{lem:X-decomp}. Applying this together with Claim~\ref{claim:first-order-hat-value} yields
\begin{align*}
\sqrt{\Lambda} \cdot \bh_{k+1} = \sum_{\ell=1}^{k+1}\bigl(L^{u}_{k+1}\bigr)_{k+1, \ell} \bu_{\ell} + \bX\bv_{k+1} &= \sum_{\ell=1}^{k}\bigl(L^{u}_{k+1}\bigr)_{k+1, \ell} \bu_{\ell} + \bX\bv_{k+1}\\
&= \sum_{\ell=1}^{k}\bigl(L^{u}_{k+1}\bigr)_{k+1, \ell} \bu_{\ell} + \bX^{\parallel} \bv_{k+1} + \proj_{\bU_k}^{\perp} \bX\proj_{\bV_k}^{\perp} \bv_{k+1}.
\end{align*}
Recalling the approximation matrices $\bT_{g}, \bT_{h}$ from Eq.~\eqref{eq:def-Tg-Th}, define the quantity $\widebar{\bh}_{k+1}$ and the random variable $\widetilde{\bh}_{k+1}$ via the relations
\begin{align*}
\sqrt{\Lambda} \cdot \widebar{\bh}_{k+1} &:= \sum_{\ell=1}^{k}\bigl(L^{u}_{k+1}\bigr)_{k+1, \ell} \bu_{\ell} -\bT_g^{\top} \bv_{k+1} + \bT_h \bv_{k+1} + \proj_{\bU_k}^{\perp} \bX\proj_{\bV_k}^{\perp} \bv_{k+1},\\
\sqrt{\Lambda} \cdot \widetilde{\bh}_{k+1} &:= \sum_{\ell=1}^{k}\bigl(L^{u}_{k+1}\bigr)_{k+1, \ell} \bu_{\ell} -\bT_g^{\top} \bv_{k+1} + \bT_h \bv_{k+1} + \sqrt{\Lambda} \cdot \| \proj_{\bV_k}^{\perp} \bv_{k+1}  \|_2 \bgamma_n,
\end{align*}
where $\bgamma_n \sim \mathsf{N}(0, I_n/n)$ is independent of everything else.  Note, by Lemma~\ref{lem:X-decomp}, on $\Omega_k$,
\[
\| \bh_{k+1} - \widebar{\bh}_{k+1} \|_2 \leq C_{\mathrm{SE}}k \Delta_1(k, \delta) \quad \text{ and } \quad \sqrt{\Lambda} \cdot \widebar{\bh}_{k+1} \overset{(d)}{=}\sum_{\ell=1}^{k}\bigl(L^{u}_{k+1}\bigr)_{k+1, \ell} \bu_{\ell} -\bT_g^{\top} \bv_{k+1} + \bT_h \bv_{k+1} + \proj_{\bU_k}^{\perp} \widetilde{\bX}\proj_{\bV_k}^{\perp} \bv_{k+1},
\]
where $\widetilde{\bX}$ is an independent copy of $\bX$.  Moreover,  $\proj_{\bU_k}^{\perp} \widetilde{\bX} \proj_{\bV_k}^{\perp} \bv_{k+1} \overset{(d)}{=} \sqrt{\Lambda} \| \proj_{\bV_k}^{\perp} \bv_{k+1}  \|_2 \proj_{\bU_k}^{\perp}\bgamma_n$.  Thus, we set
\begin{align}
	\label{eq:def-delta'}
	\delta' = \frac{20T^2 \delta 3^{k-1}}{3^{T} - 1}
\end{align}
and note that, by rotational invariance of the Gaussian distribution, Bernstein's inequality, and Lemma~\ref{lem:facts}, on $\Omega_k$, the event
\[
\Omega' := \Bigl\{ \sqrt{\Lambda} \bigl \| \| \proj_{\bV_k}^{\perp} \bv_{k+1} \|_2 \proj_{\bU_k} \bgamma_n \bigr\|_2 \leq C_{\mathrm{SE}}\sqrt{\frac{k\log(10T^2/\delta')}{n}}\Bigr\} \quad \text{ satisfies } \quad \mathbb{P}(\Omega') \geq 1 - \frac{\delta'}{10T^2}.
\]
Thus, there exists a coupling of $\widetilde{\bX}$ and $\bgamma_n$ such that on $\Omega'$, we have
\begin{align} \label{eq:htilde-close}
\| \bh_{k+1} - \widetilde{\bh}_{k+1} \|_2 \leq k C_{\mathrm{SE}} \Delta_1(k, \delta) + C_{\mathrm{SE}}\sqrt{\frac{k\log(10T^2/\delta')}{n}} \overset{\1}{\leq} k C_{\mathrm{SE}} \Delta_1(k, \delta),
\end{align}
where step $\1$ follows from our choice of $\delta'$ since $n$ is large enough. 
Note that the constants $C_{\mathrm{SE}}$ may be different in the two instances but depend only on the state evolution quantities.

Moreover, note from Eq.~\eqref{eq:bounded-norms} and and Eq.~\eqref{eq:norms-k+1} that $\max_{k' \in [k]} \| \bh_k \|_2 \lesssim 1$ and $\max_{k' \in [k+1]} \| \bu_{k+1} \|_2 \lesssim 1$. Since $\tau^{hu, k+ 1}_{\ell, \ell'}$ is pseudo-Lipschitz in its final argument, we have
\begin{align}
\left \lvert \tau^{hu, k+1}_{\ell, \ell'}(\bU_{k}, \bu_{k+1}, \bH_{k}, \bh_{k+1} ) -  \tau^{hu, k+1}_{\ell, \ell'}(\bU_{k}, \bu_{k+1}, \bH_{k}, \widetilde{\bh}_{k+1}) \right \rvert &\lesssim k \cdot \| \bh_{k + 1} - \widetilde{\bh}_{k + 1} \|_2 \notag \\
&\overset{\1}{\leq} k^2 C_{\mathrm{SE}} \cdot \Delta_1(k, \delta), \label{eq:close-to-final}
\end{align}
where step $\1$ follows from Eq.~\eqref{eq:htilde-close}.
Putting the pieces together yields
\small
\begin{align*}
&\mathbb{P} \Bigl( \Bigl\{ \left \lvert \tau^{hu, k+1}_{\ell, \ell'}(\bU_{k}, \bu_{k+1}, \bH_{k}, \bh_{k+1} ) -  \EE[ \inprod{\hse_\ell}{\use_{\ell'}} \mid \Use_{k + 1} = \bU_{k + 1}, \Hse_{k} = \bH_{k} ] \right \rvert \geq \Delta_1(k+1, \delta) - \Delta_1(k, \delta) \Bigr\} \cap \Omega_k\Bigr) \\
&\leq \mathbb{P} \Bigl( \Bigl\{ \left \lvert \tau^{hu, k+1}_{\ell, \ell'}(\bU_{k}, \bu_{k+1}, \bH_{k}, \bh_{k+1} ) -  \EE[ \inprod{\hse_\ell}{\use_{\ell'}} \mid \Use_{k + 1} = \bU_{k + 1}, \Hse_{k} = \bH_{k} ] \right \rvert \geq \Delta_1(k+1, \delta) - \Delta_1(k, \delta) \Bigr\} \cap \Omega_k \cap \Omega' \Bigr) + \mathbb{P}\{(\Omega')^c \} \\
&\overset{\1}{\leq} \mathbb{P} \Bigl( \Bigl\{ \left \lvert \tau^{hu, k+1}_{\ell, \ell'}(\bU_{k}, \bu_{k+1}, \bH_{k}, \widetilde{\bh}_{k+1} ) -  \EE[ \inprod{\hse_\ell}{\use_{\ell'}} \mid \Use_{k + 1} = \bU_{k + 1}, \Hse_{k} = \bH_{k} ] \right \rvert \geq \Delta_1(k+1, \delta) - k^2 C_{\mathrm{SE}} \cdot \Delta_1(k, \delta) \Bigr\} \cap \Omega_k \cap \Omega' \Bigr) \\
&\qquad \qquad \qquad \qquad + \delta'/(10T^2) \\
&\overset{\2}{\leq} \mathbb{P} \Bigl( \Bigl\{ \left \lvert \tau^{hu, k+1}_{\ell, \ell'}(\bU_{k}, \bu_{k+1}, \bH_{k}, \widetilde{\bh}_{k+1} ) -  \EE[ \inprod{\hse_\ell}{\use_{\ell'}} \mid \Use_{k + 1} = \bU_{k + 1}, \Hse_{k} = \bH_{k} ] \right \rvert \geq \frac{2\Delta_1(k+1, \delta)}{3} \Bigr\} \cap \Omega_k \cap \Omega' \Bigr) + \delta'/(10T^2),
\end{align*}
\normalsize
where step $\1$ follows from Eq.~\eqref{eq:close-to-final} and step $\2$ from the definition of $\Delta_1(k +1, \delta)$, taking $C_{\mathrm{SE}}$ to be large enough. 

The final step is to relate $\widetilde{\bh}_{k + 1}$ to another RV that resembles $\hse_{k + 1}$ conditional on the event $\{ \Use_{k + 1} = \bU_{k + 1}, \Hse = \bH_k \}$. Using the shorthand \mbox{$\rho_k = \sum_{\ell=1}^{k}\bigl(L^{u}_{k+1}\bigr)_{k+1, \ell} \bu_{\ell} -\bT_g^{\top} \bv_{k+1}$}, $\sigma_k = \bT_h \bv_{k+1} - \sqrt{\Lambda}\sum_{\ell=1}^{k} \beta_{k+1, \ell} \bh_{\ell}$ and $\epsilon_k = \sqrt{\Lambda}\Bigl( \| \proj_{\bV_k}^{\perp} \bv_{k+1}  \|_2  - \beta_{k+1, k+1} \Bigr) \bgamma_n$, we have
\begin{align*}
\sqrt{\Lambda} \cdot \widetilde{\bh}_{k + 1} = \rho_k + \sigma_k + \epsilon_k + \sqrt{\Lambda} \sum_{\ell=1}^{k} \beta_{k+1, \ell} \bh_{\ell} + \sqrt{\Lambda} \cdot \beta_{k+1, k+1} \bgamma_n, 
\end{align*}
But by Lemma~\ref{lem:facts}, we have $\| \rho_k \|_2 \vee \| \sigma_k \|_2 \vee \| \epsilon_k \|_2 \leq C_{\mathrm{SE}} \cdot \Delta_1(k, \delta)$. Setting $\bh^{\mathsf{surr}}_{k + 1} = \sum_{\ell=1}^{k} \beta_{k+1, \ell} \bh_{\ell} + \beta_{k+1, k+1} \bgamma_n$, we obtain
\[
\| \widetilde{\bh}_{k + 1} - \bh^{\mathsf{surr}}_{k + 1} \|_2 \leq kC_{\mathrm{SE}} \cdot \Delta_1(k, \delta).
\]
Invoking pseudo-Lipschitzness of $\tau^{hu, k+1}_{\ell, \ell'}$ as before, we obtain
\small
\begin{align*}
&\mathbb{P} \Bigl( \Bigl\{ \left \lvert \tau^{hu, k+1}_{\ell, \ell'}(\bU_{k}, \bu_{k+1}, \bH_{k}, \widetilde{\bh}_{k+1} ) -  \EE[ \inprod{\hse_\ell}{\use_{\ell'}} \mid \Use_{k + 1} = \bU_{k + 1}, \Hse_{k} = \bH_{k} ] \right \rvert \geq \frac{2\Delta_1(k+1, \delta)}{3} \Bigr\} \cap \Omega_k \cap \Omega' \Bigr) \\
&\leq \mathbb{P} \Bigl( \Bigl\{ \left \lvert \tau^{hu, k+1}_{\ell, \ell'}(\bU_{k}, \bu_{k+1}, \bH_{k}, \bh^{\mathsf{surr}}_{k+1} ) -  \EE[ \inprod{\hse_\ell}{\use_{\ell'}} \mid \Use_{k + 1} = \bU_{k + 1}, \Hse_{k} = \bH_{k} ] \right \rvert \geq \frac{2\Delta_1(k+1, \delta)}{3} - k^2 C_{\mathrm{SE}} \cdot \Delta_1(k, \delta) \Bigr\} \cap \Omega_k \cap \Omega' \Bigr) \\
&\overset{\1}{\leq} \mathbb{P} \Bigl( \left \lvert \tau^{hu, k+1}_{\ell, \ell'}(\bU_{k}, \bu_{k+1}, \bH_{k}, \bh^{\mathsf{surr}}_{k+1} ) -  \EE[ \inprod{\hse_\ell}{\use_{\ell'}} \mid \Use_{k + 1} = \bU_{k + 1}, \Hse_{k} = \bH_{k} ] \right \rvert \geq \frac{\Delta_1(k+1, \delta)}{3} \Bigr)
\end{align*}
\normalsize
where in step $\1$ we once again use the definition of $\Delta_1(k + 1, \delta)$, taking the defining constant $C_{\mathrm{SE}}$ to be large enough.

Moreover, by definition of the state evolution, we have
$\hse_{k + 1} = \sum_{\ell=1}^{k} \beta_{k+1, \ell} \hse_{\ell} + \beta_{k+1, k+1} \bgamma_n$,
and so by definition of $\tau^{hu, k+1}_{\ell, \ell'}$ we have
\[
\EE[ \inprod{\hse_\ell}{\use_{\ell'}} \mid \Use_{k + 1} = \bU_{k + 1}, \Hse_{k} = \bH_{k} ] = \EE_{\bgamma_n} \Bigl[ \tau^{hu, k+1}_{\ell, \ell'}\Bigl(\bU_k, \bu_{k + 1}, \bH_k, \sum_{\ell=1}^{k} \beta_{k+1, \ell} \bh_{\ell} + \beta_{k+1, k+1} \bgamma_n \Bigr) \Bigr].
\]
Consequently, by definition of $\bh^{\mathsf{surr}}_{k + 1}$ we must bound the tails of the random variable
\begin{align*}
R := \tau^{hu, k+1}_{\ell, \ell'}\Bigl(\bU_k, \bu_{k + 1}, \bH_k, \sum_{\ell=1}^{k} \beta_{k+1, \ell} \bh_{\ell} + \beta_{k+1, k+1} \bgamma_n \Bigr) -  \EE_{\bgamma_n} \Bigl[ \tau^{hu, k+1}_{\ell, \ell'}\Bigl(\bU_k, \bu_{k + 1}, \bH_k, \sum_{\ell=1}^{k} \beta_{k+1, \ell} \bh_{\ell} + \beta_{k+1, k+1} \bgamma_n \Bigr) \Bigr].
\end{align*}
But $\tau^{hu, k+1}_{\ell, \ell'}$ is pseudo-Lipschitz in $\bgamma_n$, so applying Lemma~\ref{lem:pseudo-Lipschitz-concentration} in conjunction with Bernstein's inequality yields
\begin{align*}
	\mathbb{P} \biggl( |R| > C_{\mathrm{SE}}\sqrt{\frac{k \log(10T^2/\delta')}{n}} \biggr) \leq \frac{
	\delta'}{10T^2}.
\end{align*}
Hence, since $\Delta_1(k + 1, \delta) \geq C_{\mathrm{SE}}\sqrt{\frac{k \log(10T^2/\delta')}{n}}$, combining the elements gives
\begin{align*}
	\mathbb{P} \Bigl( \left| \EE[ \inprod{\hse_\ell}{\use_{\ell'}} \mid \Use_{k+1} = \bU_{k+1}, \Hse_{k+1} = \bH_{k+1} ] - \EE[\inprod{\hse_\ell}{\use_{\ell'}}] \right|  \geq \Delta_1(k+1, \delta) \Bigr) \leq \frac{2\delta'}{10T^2} + \frac{2\delta 3^{k-1}}{3^T - 1}.
\end{align*}
Taking a union bound, and recalling the definition of $\delta'$ in Eq.~\eqref{eq:def-delta'} yields the desired result.

\paragraph{Proof of Claim~\ref{claim:first-order-hat-value}.}
	The proof proceeds by induction.  For the base case, note that by definition, $\uhatse_1 = \vhatse_1 = 0$.  Suppose now that the claim holds for all $\ell \in [k]$.  Note that the state evolution implies that we have 
	\begin{align} \label{eq:SE-FO-system}
		\langle \uhatse_{k+1}, \use_{\ell} \rangle_{L^2}  = \langle \vse_{k+1}, \gse_{\ell} \rangle_{L^2}, \quad \text{ for all } \quad \ell \in [k +1].
	\end{align}
	Moreover, by the state evolution, there is a unique $\uhatse_{k+1}$ which satisfies the above relation.  It thus suffices to show that there exist constants $\{c_{\ell}\}_{\ell=1}^{k}$ such that $\uhatse_{k+1} = \sum_{\ell=1}^{k} c_{\ell} \use_{\ell}$.  Expanding the first $k$ equations in~\eqref{eq:SE-FO-system} and letting $\boldsymbol{c} = [c_1 \;\vert \; \cdots \; \vert \; c_k]$ yields $
	\bK^g_k \boldsymbol{c} = \llangle \vse_{k+1}, \Gse_k \rrangle_{L^2}$ so that by Assumption~\ref{asm:non-degeneracy}, $\boldsymbol{c} = (\bK^g_k)^{-1} \llangle \vse_{k+1}, \Gse_k \rrangle_{L^2}$.  We now check that this choice of $\boldsymbol{c}$ satisfies the final relation $\langle \uhatse_{k+1}, \use_{k+1} \rangle_{L^2} = \langle \vse_{k+1}, \gse_{k+1} \rangle_{L^2}$.  To this end, we recall that, by joint Gaussianity, we can find scalars so that $\gse_{k+1} = \alpha_{k+1}\bxi_{g} + \widetilde{\bg}^{\mathrm{SE}}_{k}$, for some $\widetilde{\bg}^{\mathrm{SE}}_{k} \in \text{span}\{\gse_1, \gse_2, \ldots, \gse_{k}\}$ and where $\bxi_g$ is independent of $\{\gse_1, \ldots, \gse_k\}$.  Hence, since $\vse_{k+1}$ is independent of $\bxi_{g}$, we have $\langle \vse_{k+1}, \gse_{k+1} \rangle_{L^2} = \langle \vse_{k+1}, \widetilde{\bg}^{\mathrm{SE}}_{k} \rangle_{L^2}$.  Letting $\bw = [(\bK^g_{k+1})_{k+1, 1} \; \vert \; \cdots \; \vert \; (\bK^g_{k+1})_{k+1, k}] \in \mathbb{R}^k$, we have
	\[
	\langle \uhatse_{k+1}, \use_{k+1} \rangle_{L^2} = \langle \Use_{k+1} \boldsymbol{c}, \use_{k+1} \rangle_{L^2} = \bw^{\top} \boldsymbol{c} = \bw^{\top} (\bK^g_k)^{-1} \llangle \vse_{k+1}, \Gse_k \rrangle_{L^2}. 
	\]
	The proof is complete upon noting that $\widetilde{\bg}_k^{\mathrm{SE}} = \Gse_k (\bK^g_k)^{-1} \bw$ and executing an identical argument for $\vhatse_{k+1}$.  
\qed

\subsection{Proof of the induction step for saddle point updates} \label{sec:proof-saddle}
We must verify the deviation bounds in Eq.~\eqref{eq:ind-events-new-four}.  To be concrete, throughout this section, we will consider $\tau^{\psi_d, k+1}$~\eqref{eq:tau-definitions}, noting that an identical argument holds for each of the other functions.  Since 
\[
\EE[\tau^{\psi_d, k+1}(\vse_{k+1}, \Vse_k, \gse_{k+1}, \Gse_k)] = \EE[\psi_d(\Vse_{T}, \Gse_T)],
\]
it suffices to show that
\begin{align} \label{eq:f-desired}
\left \lvert \tau^{\psi_d, k+1}(\bv_{k+1}, \bV_k, \bg_{k+1}, \bG_k) - \mathbb{E} \tau^{\psi_d, k+1}_{d}\left( \vse_{k+1}, \Vse_{k}, \gse_{k+1}, \Gse_k \right) \right \rvert \leq \Delta_1(k+1, \delta).
\end{align}
The proof for saddle point updates is not as explicit as that for first order updates.  To set up the proof, recall that the iterates $(\bu_{k+1}, \bv_{k+1})$ are identified as the unique saddle point of the update function
\begin{align} \label{def:update-function}
	F_{k+1}(\bu, \bv) := \bu^{\top} \bX \bv 	-
	\phi_{k+1}^u(\bu)
	+
	\phi_{k+1}^v(\bv).
\end{align}

\paragraph{Proof roadmap:} The proof proceeds in several steps --- we first sketch these steps before presenting them in detail below.  As in the proof for first order updates, a key barrier to showing deviation results of the form~\eqref{eq:f-desired} is 
the dependence between the previous iterates $(\bu_{\ell}, \bv_{\ell})_{\ell=1}^{k}$ and the data matrix $\bX$.  Towards unraveling this dependence, we first apply Lemma~\ref{lem:X-decomp} to decompose---in step 0---the data $\bX$ into a component determined by the past iterates and an orthogonal component, allowing us to write an alternative strongly convex-strongly concave function $\Lpo_{k + 1}$ that---uniformly over a constant sized ball---approximates $F_{k + 1}$. 
Recalling the definitions of $\bg_{k+1}$ and $\bh_{k+1}$ for saddle updates, and noting that $\nabla_{v} F_{k+1}(\bu_{k+1}, \bv_{k+1}) = 0$ and $\nabla_{u}F_{k+1}(\bu_{k+1}, \bv_{k+1}) = 0$, we see that
\begin{align*}
	\bg_{k+1} &:= \bigl(L^{v}_{k+1}\bigr)_{k+1, k+1} \bv_{k+1} + \sum_{\ell=1}^{k} \bigl(L^{v}_{k+1}\bigr)_{k+1, \ell} \bv_k - \bX^{\top} \bu_{k+1} \\
	&= \bigl(L^{v}_{k+1}\bigr)_{k+1, k+1} \bv_{k+1} + \sum_{\ell=1}^{k} \bigl(L^{v}_{k+1}\bigr)_{k+1, \ell} \bv_k + \nabla \phi_{k+1}^{v}(\bv_{k+1}) \quad \text{ and }\\
	\sqrt{\Lambda} \cdot \bh_{k+1} &:= \bigl(L^{u}_{k+1}\bigr)_{k+1, k+1} \bu_{k+1} + \sum_{\ell=1}^{k} \bigl(L^{u}_{k+1}\bigr)_{k+1, \ell} \bu_k + \bX \vpo_{k+1}\\
	&= \bigl(L^{u}_{k+1}\bigr)_{k+1, k+1} \bu_{k+1} + \sum_{\ell=1}^{k} \bigl(L^{u}_{k+1}\bigr)_{k+1, \ell} \bu_k + \nabla \phi_{k+1}^{u}(\bu_{k+1}).
\end{align*}
This motivates the definition of 
\begin{subequations} \label{eq:g-h-po}
	\begin{align}
		\bg_{k+1}(\bv) &:= \bigl(L^{v}_{k+1}\bigr)_{k+1, k+1} \bv + \sum_{\ell=1}^{k} \bigl(L^{v}_{k+1}\bigr)_{k+1, \ell} \bv_k + \nabla \phi_{k+1}^{v}(\bv)\quad \text{ and }\\
		\sqrt{\Lambda} \cdot \bh_{k+1}(\bu) &:= \bigl(L^{u}_{k+1}\bigr)_{k+1, k+1} \bu + \sum_{\ell=1}^{k} \bigl(L^{u}_{k+1}\bigr)_{k+1, \ell} \bu_k + \nabla \phi^{u}_{k+1}(\bu_{k+1}).
	\end{align}
\end{subequations}
Writing $(\upo_{k + 1}, \vpo_{k + 1})$ as the unique saddle of $\Lpo_{k + 1}$ and taking $\gpo_{k+1} = \bg_{k+1}(\vpo_{k+1}), \hpo_{k+1} = \bh_{k+1}(\upo_{k+1})$, we note that in order to establish Eq.~\eqref{eq:f-desired}, it thus suffices to show that
\[
\left \lvert \tau^{\psi_d, k+1}(\vpo_{k+1}, \bV_k, \gpo_{k+1}, \bG_k) - \mathbb{E} \tau^{\psi_d, k+1}_{d}\left( \vse_{k+1}, \Vse_{k}, \gse_{k+1}, \Gse_k \right) \right \rvert \lesssim \Delta_1(k+1, \delta).
\]
Instead of directly bounding this quantity, we note that if these deviations are large, then two saddle values obtained by restrictions on the function $\Lpo_{k + 1}$ must obey certain properties. The problem is thus reduced to controlling the probability of these events by analyzing the restricted saddle objectives of $\Lpo_{k + 1}$.  This proof strategy is occasionally referred to in the literature as a local stability argument~\citep{miolane2021distribution}.

In step 1, we reduce the problem of analyzing the restricted saddle objectives of $\Lpo_{k + 1}$ to analyzing the restricted saddle objectives of a different loss function $\mathfrak{L}^{\mathrm{AO}}_{k+1}$, which in the typical terminology of the CGMT is known as the auxiliary loss. Unlike $\Lpo_{k + 1}$, the auxiliary loss may not be convex-concave, and therefore does not have a well-defined saddle point. Thus, the previous proof technique of reduction to studying deviations around a surrogate saddle point faces a roadblock.
To circumvent this, we invoke an idea of~\citet{celentano2023challenges} and introduce approximate stationary points $\vapx_{k+1}$ and $\uapx_{k+1}$ in step 2, proving certain properties of these points along the way.  
In step 3, we show that the minmax \emph{value} of $\mathfrak{L}^{\mathrm{AO}}_{k+1}$ is very close to pure maximization (respectively minimization) of $\mathfrak{L}^{\mathrm{AO}}_{k+1}(\bu, \vapx_{k + 1})$ and $\mathfrak{L}^{\mathrm{AO}}_{k+1}(\uapx_{k + 1}, \bv)$ over $\bu$ (respectively $\bv$). 
%
We conclude the proof in step 4 by considering each of the functions~\eqref{eq:tau-definitions} to control every required deviation in the inductive hypothesis.  Applying a union bound over each of these and taking the probability of error suitably small to ensure the geometric series sums to $\delta$ after $T$ iterations then yields the inductive step for a saddle update.

\paragraph{Step 0:} Recall the approximation matrices $\bT_g, \bT_h$~\eqref{eq:def-Tg-Th}.  In order to interpret these matrices, we define
\begin{align*}
	\bG_k^{\perp} &:= \bG_k\bigl(\bK_k^g\bigr)^{-1/2}, \qquad \bH_k^{\perp} := \bH_k \bigl(\bK_k^h\bigr)^{-1/2},\\
	\bU_k^{\perp} &:= \bU_k  \bigl(\bK_k^g\bigr)^{-1/2}, \qquad \bV_k^{\perp} := \bV_k  \bigl(\bK_k^h\bigr)^{-1/2},
\end{align*}
noting that on the event $\Omega_k$, the columns of each of the four matrices defined above are approximately orthonormal by an argument similar to Lemma~\ref{lem:facts}.  With these definitions, we have $\bT_g = \bG_k^{\perp} \bigl(\bU_k^{\perp}\bigr)^{\top}$ and $\bT_h =  \sqrt{\Lambda} \cdot \bH_k^{\perp} \bigl(\bV_k^{\perp}\bigr)^{\top}$.  

Lemma~\ref{lem:X-decomp} then inspires us to define the primary objective\footnote{We emphasize to the reader that here we use the terminology primary objective in anticipation of our application of the CGMT~\citep{thrampoulidis2018precise} in the sequel.} $\Lpo_{k+1}: \mathbb{R}^n \times \mathbb{R}^d \rightarrow \mathbb{R}$ as
\begin{align}
	\Lpo_{k+1}(\bu, \bv) := \bu^{\top}\Bigl(\bT_h^{\top} - \bT_g^{\top} + \proj^{\perp}_{\bU_k} \widetilde{\bX} \proj_{\bV_k}^{\perp}\Bigl)\bv
	-
	\phi_{k+1}^u(\bu)
	+
	\phi_{k+1}^v(\bv)
\end{align}
and its (unique) saddle point as $(\upo_{k+1}, \vpo_{k+1}) = \mathsf{saddle} \{L^{\mathrm{PO}}_{k+1}(\bu, \bv)\}$.

Note that the primary objective approximates the update function $F_{k+1}$.  Indeed, by Lemma~\ref{lem:X-decomp}, on $\Omega_k$, we have the uniform bound
\begin{align*} 
	\sup_{\bv \in \mathbb{B}_2(M), \bu \in \mathbb{B}_2(M')}\; \Bigl \lvert \Lpo_{k+1}(\bu, \bv) - F_{k+1}(\bu, \bv) \Bigr \rvert \leq C_{\mathrm{SE}, M, M'} \cdot k \cdot \Delta_1(k, \delta),
\end{align*}
where $M, M' > 0$ are small enough constants.  The following lemma, whose proof we provide in Section~\ref{sec:proof-lem-approximation-by-po}, is a straightforward corollary of this uniform approximation.
\begin{lemma}
	\label{lem:approximation-by-po}
	Let $F_{k+1}$ be as in~\eqref{def:update-function} and let $(\bu_{k+1}, \bv_{k+1})$ denote its unique saddle point.  There exist universal constants $c_{\mathrm{SE}}, C_{\mathrm{SE}} > 0$ such that if
	\begin{align*} 
		\sup_{\bv \in \mathbb{B}_2(M), \bu \in \mathbb{B}_2(M')}\; \Bigl \lvert \Lpo_{k+1}(\bu, \bv) - F_{k+1}(\bu, \bv) \Bigr \rvert \leq C_{M, M'} \cdot k \cdot \Delta_1(k, \delta),
	\end{align*}
	then 
		\begin{align}
		\label{eq:approximation-by-po}
		\| \bv_{k+1} \|_2 \vee \| \bu_{k+1} \|_2 \leq c_{\mathrm{SE}}, \quad \text{ and } \quad \| \vpo_{k+1} - \bv_{k+1} \|_2 \vee \| \upo_{k+1} - \bu_{k+1} \|_2  \leq C_{\mathrm{SE}} \cdot\sqrt{ k \cdot \Delta_1(k, \delta)}.
	\end{align}
\end{lemma}
Invoking the order-$2$ pseudo-Lipschitz nature of $f_d$ and $f_n$ in conjunction with the bound $\Delta_1(k+1, \delta) \geq \sqrt{k \cdot \Delta_1(k, \delta)}$ yields
\[
\left \lvert \tau^{\psi_d, k+1}(\vpo_{k+1}, \bV_k, \gpo_{k+1}, \bG_k) -  \tau^{\psi_d, k+1}(\bv_{k+1}, \bV_k, \bg_{k+1}, \bG_k) \right \rvert \leq \frac{\Delta_1(k+1, \delta)}{2}.
\]
Hence, to establish Eq.~\eqref{eq:f-desired}, it suffices to show that
\[
\left \lvert \tau^{\psi_d, k+1}(\vpo_{k+1}, \bV_k, \gpo_{k+1}, \bG_k)  - \mathbb{E} \tau^{\psi_d, k+1}(\vse_{k+1}, \Vse_k, \gse_{k+1}, \Gse_k)  \right \rvert \leq \frac{\Delta_1(k+1, \delta)}{2}.
\]
To this end, we define deviation sets
\begin{align*}
	\mathbb{D}_{d}(\tau^{\psi_d, k+1}) &:= \Bigl\{\bv \in \mathbb{B}_2(c_{\mathrm{SE}}): \,\left \lvert \tau^{\psi_d, k+1}(\bv, \bV_k, \bg_{k+1}(\bv), \bG_k)  - \mathbb{E} \tau^{\psi_d, k+1}(\vse_{k+1}, \Vse_k, \gse_{k+1}, \Gse_k)  \right \rvert \leq \Delta_1(k+1, \delta)\Bigr\},
\end{align*}
and study both the unrestricted behavior of the loss $L_{k+1}^{\mathrm{PO}}$ as well as the behavior of the loss when its arguments are restricted to the deviation sets.  To do so, we define the quantity $\mathrm{OPT}_{k+1}$ as the value of the auxiliary state evolution saddle point problem in Eq.~\eqref{def:saddle-obj-Hilbert}.  That is, 
\begin{align*}
	\mathrm{OPT}_{k+1} := \min_{\bv \in \mathcal{H}^{k + 1}_d} \max_{\bu \in \mathcal{H}^{k+1}_n}\; 	\big\{
	\mathfrak{L}_{k+1}(\bu, \bv)
	\big\}.
\end{align*}
Equipped with this notation, we define pair of events $\Omega_1^{\mathrm{PO}}$ and $\Omega_2^{\mathrm{PO}}$ as
\begin{align*}
	\Omega^{\mathsf{PO}}_1 &:= \Bigl\{ \min_{\bv \in \mathbb{R}^d} \max_{\bu \in \mathbb{R}^n} \; \bigl\{\Lpo_{k+1}(\bu, \bv)\bigr\} < \mathrm{OPT}_{k+1} + \Delta_1(k, \delta) \Bigr\} \quad \text{ and }\\
	\Omega^{\mathsf{PO}}_2 &:= \Bigl\{ \min_{\bv \in \mathbb{D}_d^{c}(f_d)} \max_{\bu \in \mathbb{R}^n} \; \bigl\{\Lpo_{k+1}(\bu, \bv)\bigr\} > \mathrm{OPT}_{k+1} +  C_{\mathrm{SE}} \frac{\Delta_1(k+1, \delta)^2}{k^2} \Bigr\}.
\end{align*}
In words, $\Omega^{\mathsf{PO}}_1$ encodes the event that the unrestricted saddle point has value bounded above by $\mathrm{OPT}_{k+1} + \Delta_1(k, \delta)$ and $\Omega^{\mathsf{PO}}_2$ encodes the event that the value of the saddle problem when the minimization variable is restricted to the set $\mathbb{D}_d^{c}$ is bounded below by $\mathrm{OPT}_{k+1} +  C_{\mathrm{SE}} \frac{\Delta_1(k+1, \delta)^2}{k^2}$.  Since $\Delta_1(k+1, \delta)/k \asymp \sqrt{\Delta_1(k, \delta)}$, this implies that on the intersection $\Omega^{\mathsf{PO}}_1 \cap \Omega^{\mathsf{PO}}_2$, we have $\vpo_{k+1}\in \mathbb{D}_d(\tau^{\psi_d, k+1})$.  Formally, we have
\begin{align*}
	\mathbb{P}\Bigl\{\left \lvert \tau^{\psi_d, k+1}(\vpo_{k+1}, \bV_k, \gpo_{k+1}, \bG_k)  - \mathbb{E} \tau^{\psi_d, k+1}(\vse_{k+1}, \Vse_k, \gse_{k+1}, \Gse_k)  \right \rvert > \Delta_1(k+1, \delta)\Bigr\} &\leq \mathbb{P}\bigl\{ (\Omega^{\mathsf{PO}}_1)^{c} \cup (\Omega^{\mathsf{PO}}_2)^{c} \bigr\} \\
	&\leq  \mathbb{P}\bigl\{ (\Omega^{\mathsf{PO}}_1)^{c} \bigr\} +  \mathbb{P}\bigl\{(\Omega^{\mathsf{PO}}_2)^{c} \bigr\}.
\end{align*}
In the rest of the proof, we show that the probability of each of the events $(\Omega^{\mathsf{PO}}_1)^{c}$ and $(\Omega^{\mathsf{PO}}_2)^{c}$ is small.

\paragraph{Step 1: Reduction to the auxiliary objective via the CGMT.} To bound the probability of our events, the direct objects to analyze
are the (restricted) saddles of $\Lpo_{k+1}$. However, this object is complex because the loss is defined by a random matrix. To simplify the loss under consideration, we define a closely related auxiliary objective.  Towards defining the auxiliary objective, let $\bgamma_d \sim \mathsf{N}(0, I_d)$ and $\bgamma_n \sim \mathsf{N}(0, I_n)$ be independent of all other randomness, and define the auxiliary objective $\mathfrak{L}^{\mathrm{AO}}_{k+1}: \mathbb{R}^n \times \mathbb{R}^{d} \rightarrow \mathbb{R}$ as
\begin{align}\label{eq:def-aux-objective}
	\mathfrak{L}^{\mathrm{AO}}_{k+1}(\bu, \bv) := &-\Bigl( \bT_{g} \bu + \bigl \| \proj^{\perp}_{\bU_k} \bu \|_2 \frac{\bgamma_d}{\sqrt{d}}\Bigr)^{\top} \bv \nonumber\\
	& + \Bigl( \bT_{h} \bv + \sqrt{\Lambda}\bigl \| \proj^{\perp}_{\bV_k} \bv \|_2 \frac{\bgamma_n}{\sqrt{n}}\Bigr)^{\top} \bu - \phi^{u}_{k+1}(\bu) + \phi^{v}_{k+1}(\bv). 
\end{align}
We bound the probability of our events of interest by relating them to the auxiliary loss $\mathfrak{L}^{\mathrm{AO}}_{k+1}$, using the convex Gaussian minmax theorem.  In particular, we have
\begin{subequations}
\begin{align*}
	\mathbb{P}\bigl\{(\Omega^{\mathsf{PO}}_1)^{c}\bigr\} &\leq \mathbb{P}\Bigl(\Bigl\{\min_{\bv \in \mathbb{R}^d} \max_{\bu \in \mathbb{R}^n} \{\Lpo_{k+1}(\bu, \bv)\} \geq \mathrm{OPT}_{k+1} + \Delta_1(k, \delta)\Bigr\}\bigcap \Omega_k \Bigr) + \mathbb{P}(\Omega_k^c)\\
	& \overset{\1}{\leq} \mathbb{P}\Bigl(\min_{\bv \in \mathbb{B}_2(c_{\mathrm{SE}})} \max_{\bu \in \mathbb{B}_2(c_{\mathrm{SE}})} \{\Lpo_{k+1}(\bu, \bv)\} \geq \mathrm{OPT}_{k+1} + \Delta_1(k, \delta)\Bigr) + \mathbb{P}(\Omega_k^c)\\
	&\overset{\2}{\leq} 2\mathbb{P}\Bigl(\min_{\bv \in \mathbb{B}_2(c_{\mathrm{SE}})} \max_{\bu \in \mathbb{B}_2(c_{\mathrm{SE}})} \{\mathfrak{L}^{\mathrm{AO}}_{k+1}(\bu, \bv)\} \geq \mathrm{OPT}_{k+1} + \Delta_1(k, \delta)\Bigr) +  \mathbb{P}(\Omega_k^c) \numberthis \label{ineq:AO-Omega1'} \\
	&=: 2\mathbb{P}\Bigl((\Omega^{\mathsf{AO}}_1)^{c}\Bigr) +  \mathbb{P}(\Omega_k^c).
\end{align*}
where step $\1$ follows from Lemma~\ref{lem:approximation-by-po} from step 1 and step $\2$ follows from the first part of the CGMT (see, e.g.,~\citet[][Theorem 5.1]{miolane2021distribution} or ~\citet[][Proposition 1(b)]{chandrasekher2023sharp}). Similarly, we have
\begin{align*}
	\mathbb{P}\bigl\{(\Omega^{\mathsf{PO}}_2)^{c}\bigr\} &\leq \mathbb{P}\Bigl(\Bigl\{\min_{\bv \in \mathbb{D}_d^c(\tau^{\psi_d, k+1})} \max_{\bu \in \mathbb{R}^n} \{\Lpo_{k+1}(\bu, \bv)\} \leq \mathrm{OPT}_{k+1} + C_{\mathrm{SE}} \frac{\Delta_1(k+1, \delta)^2}{k^2}\Bigr\}\bigcap \Omega_k\biggr) +\mathbb{P}(\Omega_k^c)\\
	& \leq \mathbb{P}\Bigl(\min_{\bv \in \mathbb{D}_d^{c}(\tau^{\psi_d, k+1})} \max_{\bu \in \mathbb{B}_2(c_{\mathrm{SE}})} \{\Lpo_{k+1}(\bu, \bv)\} \leq \mathrm{OPT}_{k+1} + C_{\mathrm{SE}} \frac{\Delta_1(k+1, \delta)^2}{k^2}\Bigr) + \mathbb{P}(\Omega_k^c)\\
	&\overset{\1}{\leq} 2\mathbb{P}\Bigl(\min_{\bv \in \mathbb{D}_d^{c}(\tau^{\psi_d, k+1})} \max_{\bu \in \mathbb{B}_2(c_{\mathrm{SE}})} \{\mathfrak{L}^{\mathrm{AO}}_{k + 1}(\bu, \bv)\} \leq \mathrm{OPT}_{k+1} +C_{\mathrm{SE}} \frac{\Delta_1(k+1, \delta)^2}{k^2}\Bigr) + \mathbb{P}(\Omega_k^c) \numberthis \label{ineq:AO-Omega2'} \\
	&=: 2\mathbb{P}\Bigl((\Omega^{\mathsf{AO}}_2)^{c}\Bigr) +  \mathbb{P}(\Omega_k^c),
\end{align*}
\end{subequations}
where step $\1$ follows from the second part of the CGMT (see, e.g.,~\citet[][Theorem 5.1]{miolane2021distribution} or ~\citet[][Proposition 1(a)]{chandrasekher2023sharp}).  It thus suffices to the probabilities of the events $(\Omega^{\mathsf{AO}}_1)^{c}$ and $(\Omega^{\mathsf{AO}}_2)^{c}$, which involve studying (restricted) saddles of the auxiliary loss.
While the auxiliary loss has the advantage that there is no longer a random matrix involves, it is no longer convex-concave. Accordingly, to understand its restricted saddle values, we will construct \emph{approximate} stationary points of the auxiliary loss.

\paragraph{Step 2: Approximate stationary points and their properties.} Before defining the approximate stationary points, we require some preliminary notation.  Recalling our invertibility assumption on $\bK^g_k$ and $\bK^h_k$, define functions $\bg_{k+1}^{\mathrm{apx}}: \mathbb{R}^{d \times k} \times \mathbb{R}^d \to \mathbb{R}^d, \widehat{\bv}_{k+1}^{\mathrm{apx}}: \mathbb{R}^{d \times k} \times \mathbb{R}^d \to \mathbb{R}^d$,  $\bh_{k+1}^{\mathrm{apx}}: \mathbb{R}^{n \times k} \times \mathbb{R}^n \to \mathbb{R}^n$, and $\widehat{\bu}_{k+1}^{\mathrm{apx}}: \mathbb{R}^{n \times k} \times \mathbb{R}^n \to \mathbb{R}^n$ as
\begin{align*}
	\gapx_{k+1}(\bB, \bw) &:= \bB \cdot \bigl(\bK^g_k\bigr)^{-1/2}\bigl\langle \! \bigl \langle\Use_k \bigl(\bK^g_k\bigr)^{-1/2}, \use_{k+1} \bigr\rangle \!\bigr\rangle_{L^2} +	\|\proj_{\Use_{k}}^\perp \use_{k+1} \|_{L^2} \bw,\\
	\sqrt{\Lambda} \cdot \hapx_{k+1}(\bE, \bz) &:= \bE  \cdot \bigl(\bK^h_k\bigr)^{-1/2}\bigl\langle \! \bigl \langle\Vse_k \bigl(\bK^h_k\bigr)^{-1/2}, \vse_{k+1} \bigr\rangle \!\bigr\rangle_{L^2} +	\|\proj_{\Vse_{k}}^\perp \vse_{k+1} \|_{L^2} \bz, \\
	\widehat{\bv}_{k+1}^{\mathrm{apx}}(\bA) &:= \bA \cdot \bigl(\bK^h_k\bigr)^{-1} \Bigl\{  \langle \! \langle \Hse_k, \use_{k+1} \rangle \! \rangle_{L^2}  - L^{v}_{k+1, k+1} \langle \! \langle \Vse_k, \vse_{k+1} \rangle \! \rangle_{L^2}\Bigr\} \quad \text{and}\\
	\widehat{\bu}_{k+1}^{\mathrm{apx}}(\bD) &:= \bD \cdot \bigl(\bK^g_k\bigr)^{-1} \Bigl\{  \langle \! \langle \Gse_k, \vse_{k+1} \rangle \! \rangle_{L^2}  - L^{u}_{k+1, k+1} \langle \! \langle \Use_k, \use_{k+1} \rangle \! \rangle_{L^2}\Bigr\}.
\end{align*}
Use these objects to define the functions $\vapx_{k+1}: \mathbb{R}^{d \times k} \times \mathbb{R}^{d \times k} \times \mathbb{R}^d \to \mathbb{R}^d$ and $\uapx_{k+1}: \mathbb{R}^{n \times k} \times \mathbb{R}^{n \times k} \times \mathbb{R}^n \to \mathbb{R}^n$ as
\begin{subequations} \label{eq:apx-general}
	\begin{align} \label{eq:def-vapx}
		\bv_{k+1}^{\mathrm{apx}}(\bA, \bB, \bw) = \argmin_{\bv \in \mathbb{R}^d} \biggl\{\frac{L^v_{k+1, k+1}}{2} \Bigl \| \bv - (L^v_{k+1, k+1})^{-1} \cdot \bigl[\bg_{k+1}^{\mathrm{apx}}(\bB, \bw) - \widehat{\bv}_{k+1}^{\mathrm{apx}}(\bA)\bigr] \Bigr \|_2^2 + \phi^v_{k+1}(\bv) \biggr\},
	\end{align}
	and 
	\begin{align}\label{eq:def-uapx}
		\bu_{k+1}^{\mathrm{apx}}(\bD, \bE, \bz) = \argmin_{\bu \in \mathbb{R}^n} \biggl\{\frac{L^u_{k+1, k+1}}{2} \Bigl \| \bu - (L^u_{k+1, k+1})^{-1} \cdot \bigl[\sqrt{\Lambda} \cdot \bh_{k+1}^{\mathrm{apx}}(\bE, \bz) - \widehat{\bu}_{k+1}^{\mathrm{apx}}(\bD)\bigr] \Bigr \|_2^2 + \phi^u_{k+1}(\bu) \biggr\}.
	\end{align}
\end{subequations}
Note that we have
\begin{align} \label{eq:apx-se-equivalence}
\vse_{k+1} = \vapx_{k+1}(\Vse_{k}, \Gse_k, \bgamma_d/\sqrt{d}) \quad \text{ and } \quad \use_{k+1} = \uapx_{k+1}(\Use_k, \Hse_k, \bgamma_n/\sqrt{n})
\end{align}
for $\bgamma_d \sim \normal(0, \bI_d)$ and $\bgamma_n \sim \normal(0, \bI_n)$.
To exploit these relations to the SE, we work with a different evaluation of these functions, which we will then show are approximate stationary points of the auxiliary loss. Abusing notation slightly, define 
\[
\vapx_{k + 1} := \vapx_{k+1}(\bV_k, \bG_k, \bgamma_d/\sqrt{d})
\]
according to Eq.~\eqref{eq:def-vapx} and 
\[
\uapx_{k + 1} := \uapx_{k+1}(\bU_k, \bH_k, \bgamma_n/\sqrt{n})
\]
according to Eq.~\eqref{eq:def-uapx}, where $\bgamma_d \sim \normal(0, \bI_d)$ and $\bgamma_n \sim \normal(0, \bI_n)$.
The following lemma establishes several useful concentration properties of the approximate stationary points, and parallels guarantees in Proposition~\ref{prop:events-induction}.
\begin{lemma} \label{lem:apx-concentration-useful}
	Let $\bV_{k+1}^{\mathrm{apx}} = [\bV_k \; \vert \; \vapx_{k+1}]$, $\bU_{k+1}^{\mathrm{apx}} = [\bU_k \; \vert \; \uapx_{k+1}]$, $\bG_{k+1}^{\mathrm{apx}} = [\bG_k \; \vert \; \gapx_{k+1}]$, and $\bH_{k+1}^{\mathrm{apx}} = [ \bH_k \; \vert \; \hapx_{k+1}]$.  If $\delta \in (e^{-c_{\mathrm{SE}}n}, 1)$, then with probability at least $1 - \delta' - \mathbb{P}(\Omega_k^{c})$, the following hold simultaneously.  
	\small
	\begin{subequations}\label{ineq:induction-concentration-apx}
		\begin{align}
			\max_{\ell, \ell' \in [T]} \left| \EE[ \inprod{\vse_\ell}{\vse_{\ell'}} \mid \Vse_{k+1} = \bV_{k+1}^{\mathrm{apx}}, \Gse_{k+1} = \bG_{k+1}^{\mathrm{apx}} ] - \EE[\inprod{\vse_\ell}{\vse_{\ell'}}] \right| &\leq \Delta_1(k, \delta) + C_{\mathrm{SE}}\sqrt{\frac{\log(\frac{10T^2}{\delta'})}{n}}, \label{eq:apx-ind-events-new-four-1} \\
			\max_{\ell, \ell' \in [T]} \left| \EE[ \inprod{\gse_\ell}{\vse_{\ell'}} \mid \Vse_{k+1} = \bV_{k+1}^{\mathrm{apx}}, \Gse_{k+1} = \bG_{k+1}^{\mathrm{apx}} ] - \EE[\inprod{\gse_\ell}{\vse_{\ell'}}] \right| &\leq \Delta_1(k, \delta) + C_{\mathrm{SE}}\sqrt{\frac{\log(\frac{10T^2}{\delta'})}{n}}, \label{eq:apx-ind-events-new-four-2} \\
			\max_{\ell, \ell' \in [T]} \left| \EE[ \inprod{\gse_\ell}{\gse_{\ell'}} \mid \Vse_{k+1} = \bV_{k+1}^{\mathrm{apx}}, \Gse_{k+1} = \bG_{k+1}^{\mathrm{apx}} ] - \EE[\inprod{\gse_\ell}{\gse_{\ell'}}] \right| &\leq\Delta_1(k, \delta) + C_{\mathrm{SE}}\sqrt{\frac{\log(\frac{10T^2}{\delta'})}{n}}, \label{eq:apx-ind-events-new-four-gg} \\
			\max_{\ell, \ell' \in [T]} \left| \EE[ \inprod{\use_\ell}{\use_{\ell'}} \mid \Use_{k+1} = \bU_{k+1}^{\mathrm{apx}}, \Hse_{k+1} = \bH_{k+1}^{\mathrm{apx}} ] - \EE[\inprod{\use_\ell}{\use_{\ell'}}] \right| &\leq \Delta_1(k, \delta) + C_{\mathrm{SE}}\sqrt{\frac{\log(\frac{10T^2}{\delta'})}{n}}, \label{eq:apx-ind-events-new-four-3} \\
			\max_{\ell, \ell' \in [T]} \left| \EE[ \inprod{\hse_\ell}{\use_{\ell'}} \mid \Use_{k+1} = \bU_{k+1}^{\mathrm{apx}}, \Hse_{k+1} = \bH_{k+1}^{\mathrm{apx}} ] - \EE[\inprod{\hse_\ell}{\use_{\ell'}}] \right| &\leq \Delta_1(k, \delta) + C_{\mathrm{SE}}\sqrt{\frac{\log(\frac{10T^2}{\delta'})}{n}}, \label{eq:apx-ind-events-new-four-4} \\
			\max_{\ell, \ell' \in [T]} \left| \EE[ \inprod{\hse_\ell}{\hse_{\ell'}} \mid \Use_{k+1} = \bU_{k+1}^{\mathrm{apx}}, \Hse_{k+1} = \bH_{k+1}^{\mathrm{apx}} ] - \EE[\inprod{\hse_\ell}{\hse_{\ell'}}] \right| &\leq \Delta_1(k, \delta) + C_{\mathrm{SE}}\sqrt{\frac{\log(\frac{10T^2}{\delta'})}{n}}, \label{eq:apx-ind-events-new-four-hh}\\
			\bigl \lvert \mathbb{E}\bigl[ \psi_d(\Vse_T, \Gse_T) \mid \Vse_{k+1} = \bV_{k+1}^{\mathrm{apx}}, \Gse_{k+1} = \bG_{k+1}^{\mathrm{apx}}\bigr] - \mathbb{E}\bigl[\psi_d(\Vse_T, \Gse_T)\bigr] \bigr \rvert &\leq \Delta_1(k, \delta) + C_{\mathrm{SE}}\sqrt{\frac{\log(\frac{10T^2}{\delta'})}{n}}, \label{eq:apx-ind-events-new-four-psid}\\
			\bigl \lvert \mathbb{E}\bigl[ \psi_n(\Use_T, \Hse_T) \mid \Use_{k+1} = \bU_{k+1}^{\mathrm{apx}}, \Hse_{k+1} = \bH_{k+1}^{\mathrm{apx}}\bigr] - \mathbb{E}\bigl[\psi_n(\Use_T, \Hse_T)\bigr] \bigr \rvert &\leq \Delta_1(k, \delta) + C_{\mathrm{SE}}\sqrt{\frac{\log(\frac{10T^2}{\delta'})}{n}} \label{eq:apx-ind-events-new-four-psin},\\
			\max_{\ell \in T} \Bigl \lvert \mathbb{E}\bigl[\phi_{\ell}^{v}(\vse_{\ell}) \mid \Vse_{k+1} = \bV_{k+1}^{\mathrm{apx}}, \Gse_{k+1} = \bG_{k+1}^{\mathrm{apx}}\bigr] - \EE\bigl[\phi_{\ell}^{v}(\vse_{\ell})\bigr] \Bigr \rvert &\leq \Delta_1(k, \delta) + C_{\mathrm{SE}}\sqrt{\frac{\log(\frac{10T^2}{\delta'})}{n}}, \label{eq:apx-ind-events-four-phiv}\\
			\max_{\ell \in T} \Bigl \lvert \mathbb{E}\bigl[\phi_{\ell}^{u}(\use_{\ell}) \mid \Use_{k+1} = \bU_{k+1}^{\mathrm{apx}}, \Hse_{k+1} = \bH_{k+1}^{\mathrm{apx}}\bigr] - \EE\bigl[\phi_{\ell}^{u}(\use_{\ell})\bigr] \Bigr \rvert &\leq \Delta_1(k, \delta) + C_{\mathrm{SE}}\sqrt{\frac{\log(\frac{10T^2}{\delta'})}{n}} \label{eq:apx-ind-events-four-phiu}.
		\end{align}
	\normalfont
	\end{subequations}
\end{lemma}

\paragraph{Step 3: Relating auxiliary saddle value to evaluations on approximate stationary points.}
We now consider evaluations of the auxiliary objective at the approximate stationary points $\vapx_{k + 1}$ and $\uapx_{k + 1}$.
  Toward this end, we define the pair of functions $G: \mathbb{R}^n \rightarrow \mathbb{R}$ and $H: \mathbb{R}^d \rightarrow \mathbb{R}$ as
\begin{subequations}
	\begin{align}
		G(\bv) &:= \bv^{\top}\Bigl(-\bT_g \uapx_{k+1}+ \bT_h^{\top} \uapx_{k+1} + \| \proj^{\perp}_{\bU_k} \uapx_{k+1} \|_2 \frac{\bgamma_d}{\sqrt{d}} \Bigr) \nonumber\\
		&\qquad \qquad \qquad \qquad \qquad+ \bigl \| \proj^{\perp}_{\bV_k} \bv \|_2 \cdot \langle \bxi_h, \bu^{\mathrm{SE}}_{k+1} \rangle_{L^2} - \phi_{k+1}^{u}(\uapx_{k+1}) + \phi_{k+1}^{v}(\bv), \\
		H(\bu) &:= \bu^{\top}\Bigl(\bT_h \vapx_{k+1} - \bT_g^{\top} \vapx_{k+1} + \sqrt{\Lambda}\| \proj^{\perp}_{\bV_k} \vapx_{k+1} \|_2 \frac{\bgamma_n}{\sqrt{n}} \Bigr) \nonumber\\
		&\qquad \qquad \qquad \qquad \qquad - \bigl \| \proj^{\perp}_{\bU_k} \bu \|_2 \cdot \langle \bxi_g, \bv^{\mathrm{SE}}_{k+1} \rangle_{L^2} - \phi_{k+1}^{u} (\bu) + \phi_{k+1}^{v}(\vapx_{k+1}).
	\end{align}
\end{subequations}
Importantly, by Claim~\ref{clm:chains},  $\langle \bxi_g, \bv^{\mathrm{SE}}_{k+1} \rangle_{L^2} \geq 0$ and $\langle \bxi_h, \bu^{\mathrm{SE}}_{k+1} \rangle_{L^2} \geq 0$ so that $G$ is strongly convex and $H$ is strongly concave. Also, by construction, we have
\begin{subequations}
\begin{align}
	\mathfrak{L}^{\mathrm{AO}}_{k+1}(\bu, \vapx_{k+1}) &= H(\bu) + \Bigl( \bigl \| \proj^{\perp}_{\bU_k} \bu \|_2 \cdot \langle \bxi_g, \bv^{\mathrm{SE}}_{k+1} \rangle_{L^2} -  \bigl \| \proj^{\perp}_{\bU_k} \bu \|_2 \cdot \frac{\bgamma_d^{\top}\vapx_{k+1}}{\sqrt{d}}\Bigr), \text{ and } \label{eq:aux-loss-H} \\
	\mathfrak{L}^{\mathrm{AO}}_{k+1}(\uapx_{k + 1}, \bv) &= G(\bv) - \Bigl( \bigl \| \proj^{\perp}_{\bV_k} \bv \|_2 \cdot \langle \bxi_h, \bu^{\mathrm{SE}}_{k+1} \rangle_{L^2} -  \sqrt{\Lambda} \bigl \| \proj^{\perp}_{\bV_k} \bv \|_2 \cdot \frac{\bgamma_n^{\top}\uapx_{k+1}}{\sqrt{n}}\Bigr). \label{eq:aux-loss-G}
\end{align} 
\end{subequations}
We will show shortly (invoking Lemma~\ref{lem:apx-concentration-useful}) that the terms in brackets above are appropriately small, so that $H(\uapx_{k + 1})$ and $G(\vapx_{k + 1})$ carry (approximate) information about the  desired restricted saddle values.
The following lemma (proved in Section~\ref{sec:proof-lem-approximate-stationarity-F-func}) shows properties of $H$ and $G$ at $\uapx_{k+1}$ and $\vapx_{k + 1}$ respectively. 
\begin{lemma} \label{lem:approximate-stationarity-F-func}
	If $\delta' \in (e^{-c_{\mathrm{SE}}n}, 1)$, then with probability at least $1 - \delta' - \mathbb{P}(\Omega_k^c)$, the following hold simultaneously
	\begin{subequations}
		\label{eq:gradient-bound-F}
		\begin{align}
			\lvert G(\vapx_{k+1}) - \mathrm{OPT}_{k+1} \rvert \leq  C_{\mathrm{SE}}k^2\Delta_1(k, \delta)  C_{\mathrm{SE}} k^2&\sqrt{\frac{\log(10T^2/\delta')}{n}} \quad \text{ and }\label{ineq:gradient-bound-F-c}\\
			\| \partial G(\vapx_{k+1}) \|_2 \leq C_{\mathrm{SE}}k^2\Delta_1(k, \delta) + C_{\mathrm{SE}}k^2& \sqrt{\frac{\log(10T^2/\delta')}{n}}. \label{ineq:gradient-bound-F-d}\\
			\lvert H(\uapx_{k+1}) - \mathrm{OPT}_{k+1} \rvert \leq  C_{\mathrm{SE}}k^2\Delta_1(k, \delta)  C_{\mathrm{SE}} k^2&\sqrt{\frac{\log(10T^2/\delta')}{n}} \quad \text{ and }\label{ineq:gradient-bound-F-a}\\
			\| \partial H(\uapx_{k+1}) \|_2 \leq C_{\mathrm{SE}}k^2\Delta_1(k, \delta) + C_{\mathrm{SE}}k^2& \sqrt{\frac{\log(10T^2/\delta')}{n}}. \label{ineq:gradient-bound-F-b}
		\end{align}
	\end{subequations}
\end{lemma}

We are now in a position to bound the probabilities of $(\Omega_1^{\mathrm{AO}})^c$ and $(\Omega_2^{\mathrm{AO}})^c$. The following facts about the state evolution quantities will be useful in the proof. Recall that $\bxi_g \sim \normal(0, \bI_d/d)$ and $\bxi_h \sim \normal(0, \bI_n/n)$; we have
\begin{subequations}
\begin{align} 
\langle \bxi_g, \vse_{k+1}\rangle_{L^2} &= \frac{ \langle \vse_{k+1}, \gse_{k+1} \rangle_{L^2} - \llangle \vse_{k+1}, \Gse_k \rrangle_{L^2} \bigl(\bK^g_k\bigr)^{-1}\llangle \Use_k, \use_{k+1}\rrangle_{L^2}}{\| \proj_{\Vse_k}^{\perp} \vse_{k+1} \|_{L^2}}, \label{eq:SE-facts-saddle1} \\
\langle \bxi_h, \use_{k+1}\rangle_{L^2} &= \frac{ \langle \use_{k+1}, \hse_{k+1} \rangle_{L^2} - \llangle \use_{k+1}, \Hse_k \rrangle_{L^2} \bigl(\bK^h_k\bigr)^{-1}\llangle \Vse_k, \vse_{k+1}\rrangle_{L^2}}{\| \proj_{\Use_k}^{\perp} \use_{k+1} \|_{L^2}}. \label{eq:SE-facts-saddle2}
\end{align}
\end{subequations}
These relations can be directly obtained from the definition of the state evolution and Eq.~\eqref{eq:SE-g-hilbert-h-hilbert}; also see Eq.~\eqref{eq:SE-algebra} to follow.

Let $\Omega'$ denote the event that all the inequalities in Lemmas~\ref{lem:apx-concentration-useful} and~\ref{lem:approximate-stationarity-F-func} hold with parameter $\delta'$ so that $\PP(\Omega') \geq 1 - \delta' - \PP(\Omega_k^c)$ and let $\Omega_0' = \bigl\{\| \gamma_n \|_2 \vee \| \gamma_d \|_2 \leq C_{\Lambda} \sqrt{d}\bigr\}$, which satisfies $\mathbb{P}(\Omega_0') \geq 1 - e^{-cn}$. 
\medskip

\noindent \underline{Bounding $\mathbb{P}(\Omega_1^{\mathrm{AO}})^c$.}
 Note that by definition of $\gapx_{k + 1}$, we have
\begin{align*}
	\frac{\bgamma_d^{\top} \vapx_{k+1}}{\sqrt{d}} &= \frac{\bigl(\vapx_{k+1}\bigr)^{\top} \gapx_{k+1} - \bigl(\vapx_{k+1}\bigr)^{\top} \bG_k \bigl(\bK^g_k\bigr)^{-1}\llangle \Use_k, \use_{k+1}\rrangle_{L^2}}{\| \proj_{\Vse_k}^{\perp} \vse_{k+1} \|_{L^2}}.
\end{align*}
We also have $\| \proj_{\bU_k}^{\perp} \uapx_{k+1} \|_2 = \sqrt{\| \uapx_{k+1}\|_2^2 - \| \proj_{\bU_k} \uapx_{k+1} \|_2^2}$. But
by Lemma~\ref{lem:apx-concentration-useful}, on $\Omega' \cap \Omega_0'$, we have
\begin{align*}
|\bigl(\vapx_{k+1}\bigr)^{\top} \gapx_{k+1} - \llangle \vse_{k +1}, \gse_{k + 1} \rrangle_{L^2} | &\leq \Delta_1(k, \delta) + C_{\mathrm{SE}}\sqrt{\frac{\log(\frac{10T^2}{\delta'})}{n}} \\
\| \bigl(\vapx_{k+1}\bigr)^{\top} \bG_k - \llangle \vse_{k + 1}, \Gse_{k} \rrangle_{L^2} \|_{2} &\leq \sqrt{k} \Delta_1(k, \delta) + \sqrt{k} C_{\mathrm{SE}}\sqrt{\frac{\log(\frac{10T^2}{\delta'})}{n}}.
\end{align*}
Putting together the pieces, we obtain
\begin{align} \label{eq:corollary-73a-bound}
	\Bigl \lvert \frac{\bgamma_d^{\top} \vapx_{k+1}}{\sqrt{d}} - \langle \bxi_g, \vse_{k+1}\rangle_{L^2} \Bigr \rvert \vee \bigl \lvert \| \proj_{\bU_k}^{\perp} \uapx_{k+1} \|_2 - \| \proj_{\Use_k}^{\perp} \use_{k+1} \|_{L^2} \|_{L^2} \bigr \rvert \leq \sqrt{k} C_{\mathrm{SE}}\Delta_1(k, \delta) + \sqrt{k} C_{\mathrm{SE}} \sqrt{\frac{\log(10T^2/\delta')}{n}}.
\end{align}
Now note that on the event $\Omega' \cap \Omega_0'$, 
\begin{align*}
	\min_{\bv \in \mathbb{B}_2(c_{\mathrm{SE}})} \max_{\bu \in \mathbb{B}_2(c_{\mathrm{SE}})} \bigl\{&\mathfrak{L}^{\mathrm{AO}}_{k+1}(\bu, \bv)\bigr\} \leq \max_{\bu \in \mathbb{B}_2(c_{\mathrm{SE}})} \bigl\{\mathfrak{L}^{\mathrm{AO}}_{k+1}(\bu, \vapx_{k+1})\bigr\}\\
	&\overset{\1}{=} \max_{\bu \in \mathbb{B}_2(c_{\mathrm{SE}})} \Bigl\{H(\bu) + \Bigl( \bigl \| \proj^{\perp}_{\bU_k} \bu \|_2 \cdot \langle \bxi_g, \bv^{\mathrm{SE}}_{k+1} \rangle_{L^2} -  \bigl \| \proj^{\perp}_{\bU_k} \bu \|_2 \cdot \frac{\bgamma_d^{\top}\vapx_{k+1}}{\sqrt{d}}\Bigr) \Bigr\} \\
	&\overset{\2}{\leq} \max_{\bu \in \mathbb{B}_2(c_{\mathrm{SE}})} \bigl\{H(\bu)\bigr\} + \sqrt{k} C_{\mathrm{SE}} \Delta_1(k, \delta) + \sqrt{k} C_{\mathrm{SE}}\sqrt{\frac{\log(10T^2/\delta')}{n}} \numberthis \label{eq:bound-minmax-local-stability},
\end{align*}
where step $\1$ follows from Eq.~\eqref{eq:aux-loss-H} and
step $\2$ from the inequality~\eqref{eq:corollary-73a-bound}.  
Hence, invoking the concavity of $H$ in conjunction with the Cauchy--Schwarz inequality, we deduce that, for any $\bu \in \mathbb{B}_2(K_2)$, 
\[
H(\bu) \leq H(\bu^{\mathrm{apx}}) + \partial H(\bu^{\mathrm{apx}})^{\top}(\bu - \bu^{\mathrm{apx}}) \overset{\text{Lemma}~\ref{lem:approximate-stationarity-F-func}}{\leq} \mathrm{OPT}_{k+1} + C_{\mathrm{SE}} k^2 \Delta_1(k, \delta) + C_{\mathrm{SE}}k^2\sqrt{\frac{\log(10T^2/\delta')}{n}}.
\]
Combining the elements yields that on $\Omega' \cap \Omega_0'$,
\begin{align*} 
	\min_{\bv \in \mathbb{B}_2(c_{\mathrm{SE}})} \max_{\bu \in \mathbb{B}_2(c_{\mathrm{SE}})} \bigl\{\mathfrak{L}^{\mathrm{AO}}_{k+1}(\bu, \bv)\bigr\}\leq \mathrm{OPT}_{k+1} + C_{\mathrm{SE}} k^2 \Delta_1(k, \delta) + C_{\mathrm{SE}}k^2\sqrt{\frac{\log(10T^2/\delta')}{n}}, 
\end{align*}
so that
\begin{align}\label{ineq:conclusion-step2a}
	\mathbb{P}\Bigl(\min_{\bv \in \mathbb{B}_2(c_{\mathrm{SE}})} \max_{\bu \in \mathbb{B}_2(c_{\mathrm{SE}})} \{\mathfrak{L}^{\mathrm{AO}}_{k+1}(\bu, \bv)\} \geq \mathrm{OPT}_{k+1} + \Delta_1(k, \delta)\Bigr) \leq \mathbb{P}\left\{(\Omega' \cap \Omega_0')^{c}\right\} \leq \delta' + \mathbb{P}(\Omega_k^{c}).
\end{align}

\medskip
\noindent \underline{Bounding $\mathbb{P}(\Omega_2^{\mathrm{AO}})$.}
In order to reduce the notational burden, we set 
\[
\Delta := 2C_{\mathrm{SE}}' k \cdot \biggl\{k^2\Delta_1(k, \delta) + k^2\sqrt{\frac{\log\bigl(10T^2/\delta'\bigr)}{n}} \biggr\}^{1/2} = 2C_{\mathrm{SE}} k \cdot \Delta_1(k+1, \delta),
\]
which is by assumption larger than the right-hand side of the inequalities in Lemma~\ref{lem:apx-concentration-useful}.  We thus have, on $\Omega' \cap \Omega_0'$, 
\begin{align*}
	\mathbb{D}_d^{c}(\tau^{\psi_d, k+1}) &= \Bigl\{\bv \in \mathbb{B}_2(c_{\mathrm{SE}}):\; \bigl \lvert \tau^{\psi_d, k+1}(\bv, \bV_k, \bg_{k+1}(\bv), \bG_k) - \mathbb{E}\bigl[\tau^{\psi_d, k+1}\bigl(\vse_{k+1}, \Vse_{k}, \gse_{k+1}, \Gse_{k}\bigr)\bigr] \bigr \rvert > \Delta \Bigr\}\\
	&\overset{\1}{\subseteq} \Bigl\{\bv \in \mathbb{B}_2(c_{\mathrm{SE}}):\; \bigl \lvert \tau^{\psi_d, k+1}(\bv, \bV_k, \bg_{k+1}(\bv), \bG_k) - \tau^{\psi_d, k+1}\bigl(\vapx_{k+1}, \bV_{k}, \bg_{k+1}(\vapx_{k+1}), \bG_{k}\bigr) \bigr \rvert > \Delta/2 \Bigr\} \\
	&\overset{\2}{\subseteq} \Bigl\{\bv\in \mathbb{B}_2(c_{\mathrm{SE}}): \; \| \bv- \vapx_{k+1} \|_2 > \frac{\Delta}{k c_{\mathrm{SE}}} \Bigr\},
\end{align*}
where the inclusion $\1$ follows from Lemma~\ref{lem:apx-concentration-useful} (in particular the inequality~\eqref{eq:apx-ind-events-new-four-psin}) and the inclusion $\2$ follows from the order-$2$ pseudo-Lipschitz nature of $\tau^{\psi_n, k+1}$.  Consequently, on $\Omega' \cap \Omega_0'$, we have
\begin{align*}
	\max_{\bu \in \mathbb{B}_2(c_{\mathrm{SE}})} &\min_{\bv \in \mathbb{D}_d^{c}} \bigl\{\mathfrak{L}^{\mathrm{AO}}_{k+1}(\bu, \bv)\bigr\} \geq \min_{\bv \in \mathbb{D}_d^{c}} \bigl\{\mathfrak{L}^{\mathrm{AO}}_{k+1}(\uapx_{k+1}, \bv)\bigr\} \geq \min_{\bv\in \mathbb{B}_2(c_{\mathrm{SE}}): \; \| \bv - \vapx_{k+1} \|_2 > \frac{c_{\mathrm{SE}}\Delta}{k}} \bigl\{\mathfrak{L}^{\mathrm{AO}}_{k+1}(\uapx_{k+1}, \bv)\bigr\}  \\
	&\overset{\2}{=} \min_{\bv\in \mathbb{B}_2(c_{\mathrm{SE}}): \; \| \bv - \vapx_{k+1} \|_2 >\frac{c_{\mathrm{SE}}\Delta}{k}} \Bigl\{G(\bv) - \Bigl( \bigl \| \proj^{\perp}_{\bV_k} \bv \|_2 \cdot \langle \bxi_h, \bu^{\mathrm{SE}}_{k+1} \rangle_{L^2} -  \bigl \| \proj^{\perp}_{\bV_k} \bv \|_2 \cdot \frac{\bgamma_n^{\top}\uapx_{k+1}}{\sqrt{d}}\Bigr) \Bigr\} \\
	&\overset{\2}{\geq} \max_{\bv \in \bv\in \mathbb{B}_2(c_{\mathrm{SE}}): \; \| \bv - \vapx_{k+1} \|_2 > \frac{c_{\mathrm{SE}}\Delta}{k}} \bigl\{V(\bv)\bigr\} + \sqrt{k} C_{\mathrm{SE}} \Delta_1(k, \delta) + \sqrt{k} C_{\mathrm{SE}} \sqrt{\frac{\log(10T^2/\delta')}{n}}, \numberthis \label{eq:bound-minmax-local-stability-new}
\end{align*}
where step $\1$ follows from Eq.~\eqref{eq:aux-loss-G} and step $\2$ from an argument identical to the one above but using Eq.~\eqref{eq:SE-facts-saddle2}.
Invoking the strong convexity of $G$ in conjunction with the Cauchy--Schwarz inequality, we deduce that 
\begin{align*}
	G(\bv) &\geq G(\bv^{\mathrm{apx}}_{k+1}) + \partial G(\bv^{\mathrm{apx}}_{k+1})^{\top} (\bv - \bv^{\mathrm{apx}}_{k+1}) + \frac{\mu}{2} \| \bv - \bv^{\mathrm{apx}}_{k+1} \|_2^2 \\
	&\overset{\mathrm{Lemma}~\ref{lem:approximate-stationarity-F-func}}{\geq} \mathrm{OPT}_{k+1}  -   C_{\mathrm{SE}} k^2\Delta_1(k, \delta) - C_{\mathrm{SE}} k^2\sqrt{\frac{\log(10T^2/\delta')}{n}} + \frac{\mu}{2} \| \bv - \bv^{\mathrm{apx}}_{k+1} \|_2^2 \geq  \mathrm{OPT}_{k+1} +C_{\mathrm{SE}}\frac{\Delta^2}{k^2},
\end{align*}
where the final inequality follows upon taking $C_{\mathrm{SE}}'$ in the definition of $\Delta$ large enough.  Consequently, we find that 
\begin{align} \label{ineq:conclusion-step2b}
\mathbb{P}\Bigl(\min_{\bv \in \mathbb{D}_d} \max_{\bu \in \mathbb{B}_2(c_{\mathrm{SE}})} \{\mathfrak{L}_{n}^{(k+1)}(\bu, \bv)\} \leq &\mathrm{OPT}_{k+1} + C_{\mathrm{SE}} \frac{\Delta_1(k+1, \delta)^2}{k^2}\Bigr) \leq \mathbb{P}\left\{ (\Omega' \cap \Omega_0')^{c}\right\} \leq \delta' + \mathbb{P}(\Omega_k^{c}). 
\end{align}

\paragraph{Step 4: Putting the pieces together.} Combining the inequalities~\eqref{ineq:AO-Omega1'},~\eqref{ineq:AO-Omega2'},~\eqref{ineq:conclusion-step2a}, and~\eqref{ineq:conclusion-step2b}, we find that
$\mathbb{P}(\Omega^{\mathsf{PO}}_1 \cap \Omega^{\mathsf{PO}}_2) \geq 1 - \delta' - 3\mathbb{P}(\Omega_k^c)$.
Taking
\[
\delta' = \frac{2\delta \cdot 3^{k-1}}{10T^2 \cdot (3^{T} - 1)},
\]
we conclude by specifying every function $\tau$ in~\eqref{eq:tau-definitions} and repeating the previous step (all on the event $\Omega' \cap \Omega_0'$) so that
\[
\PP(\Omega_{k+1}^{c}) \leq 2\delta' + \PP(\Omega_k^{c}) \leq \frac{2\delta \cdot 3^k}{3^{T} - 1},
\]
as desired. \hfill \qed

\subsection{Proofs of helper lemmas} \label{sec:helper-thm2}

In this section, we collect proofs of Lemmas~\ref{lem:conditional-expectation-PL} and~\ref{lem:X-decomp}, which were used in the proof of Theorem~\ref{thm:exact-asymptotics} before the inductions steps.

\subsubsection{Proof of Lemma~\ref{lem:conditional-expectation-PL}} \label{sec:proof-PL-SE}
	The first part follows from the definition of conditional expectation, so we restrict ourselves to proving pseudo-Lipschitzness.  Note that the case $k=0$ holds by assumption, so we begin with the case $k=1$.  We have two cases, when the final step is a first-order update and when the final step is a saddle update.
	
	\medskip
	\noindent \underline{Case 1: The final step is first-order.}  In this case, we have
	\[
	\vse_{T} = f_T^{v}\Bigl(\gse_{T-1} - \sum_{\ell=1}^{T-1} c_{\ell} \vse_{\ell}\Bigr) \quad \text{ and } \quad \gse_{T} = \alpha_T \widetilde{\bg} + \sum_{\ell=1}^{T-1} \alpha_{\ell} \gse_{\ell},
	\]
	for some constants $\alpha_1, \ldots, \alpha_T$ and $c_{1}, \ldots, c_{T - 1}$, where $\widetilde{\bg} \sim \mathsf{N}(0, I_d/d)$ is independent of everything else.  Hence, 
	\begin{align*}
		\zeta_1^{v}(\bA, \bB) &= \mathbb{E}\bigl[ \tau^{v}(\vse_T, \Vse_{T-1}, \gse_{T}, \gse_{T-1}) \mid \Vse_{T-1} =\bA, \Gse_{T-1} = \bB\bigr] \\
		&= \mathbb{E}\Bigl[ \tau^{v}\Bigl( f_T^{v}\Bigl(\bb_{T-1} - \sum_{\ell=1}^{T-1}c_{\ell} \ba_{\ell}\Bigr), \bA, \alpha_T \widetilde{\bg} + \sum_{\ell=1}^{T-1} \alpha_{\ell} \bb_{T-1}\Bigr)\Bigr],
	\end{align*}
	where the expectation is taken with respect to $\widetilde{\bg}$.  Invoking the pseudo-Lipschitz nature of $\tau^{v}$ in conjunction with the Lipschitz nature of $f_{T}^v$, it follows that
	\begin{align*}
		\zeta_1^{v}(\bA, \bB) - &\zeta_1^{v}(\bA', \bB') =  \mathbb{E}\Bigl[ \tau^{v}\Bigl( f_T^{v}\Bigl(\bb_{T-1} - \sum_{\ell=1}^{T-1} c_{\ell} \ba_{\ell}\Bigr), \bA, \alpha_T \widetilde{\bg} + \sum_{\ell=1}^{T-1} \alpha_{\ell} \bb_{T-1}\Bigr) \\
		& \hspace{6cm}-  \tau^{v}\Bigl( f_T^{v}\Bigl(\bb_{T-1}' - \sum_{\ell=1}^{T-1} c_{\ell} \ba_{\ell}'\Bigr), \bA', \alpha_T \widetilde{\bg} + \sum_{\ell=1}^{T-1} \alpha_{\ell} \bb_{T-1}'\Bigr) \Bigr]\\
		&\lesssim \mathbb{E}\Bigl[ \Bigl\{ 1 + \Bigl \| f_T^{v}\Bigl(\bb_{T-1} - \sum_{\ell=1}^{T-1} c_{\ell} \ba_{\ell}\Bigr) \Bigr \|_2 + \Bigl \| f_T^{v}\Bigl(\bb_{T-1'} - \sum_{\ell=1}^{T-1} c_{\ell} \ba_{\ell}'\Bigr) \Bigr\|_2 + \| \widetilde{\bg}\|_2 +\\
		&\hspace{5cm} \sum_{\ell=1}^{T-1} \|\ba_{\ell} \|_2 + \| \bb_{\ell}\|_2 + \| \ba_{\ell}' \|_2 + \|\bb_{\ell}' \|_2 \Bigr\} \cdot  \sum_{\ell=1}^{T-1} \| \ba_{\ell} - \ba_{\ell}' \|_2 + \| \bb_{\ell} - \bb_{\ell}' \|_2 \Bigr]\\
		&\lesssim \Bigl\{1 + \sum_{\ell=1}^{T-1} \|\ba_{\ell} \|_2 + \| \bb_{\ell}\|_2 + \| \ba_{\ell}' \|_2 + \|\bb_{\ell}' \|_2 \Bigr\} \cdot  \sum_{\ell=1}^{T-1} \| \ba_{\ell} - \ba_{\ell}' \|_2 + \| \bb_{\ell} - \bb_{\ell}' \|_2,
	\end{align*}
	where in the final inequality we used the fact that $\mathbb{E}[\| \widetilde{\bg} \|_2] \lesssim 1$.  This concludes the proof in the case when the last step is first-order.  
	
	\medskip
	\noindent \underline{Case 2: The final step is saddle update.}  By definition, there exist scalars $c_1, \ldots, c_T$ such that $\vhatse_T = \sum_{\ell=1}^{T} c_{\ell} \vse_T$ and $\alpha_1, \ldots, \alpha_T$ such that $\gse_T = \alpha_T \widetilde{\bg} + \sum_{\ell=1}^{T-1} \alpha_{\ell} \gse_{\ell}$.  It follows from~\eqref{eq:se-saddle} that $\vse_{T}$ admits the equivalent representation
	\[
	\vse_T = \argmin_{\bv \in \mathbb{R}^d}\; \Bigl\{\frac{c_T}{2} \cdot \Bigl \| \bv - \frac{1}{c_T}\cdot \Bigl(\alpha_T \widetilde{\bg} + \sum_{\ell=1}^{T-1} \beta_{\ell} \gse_{\ell}  - \sum_{\ell=1}^{T-1} c_{\ell} \vse_{\ell}\Bigr) \Bigr\|_2^2 + \phi^v_{T}(\bv) \Bigr\}.
	\]
	Hence, by~\citet[Proposition 12.28]{bauschke2017convex}, $\vse_T$ is a Lipschitz function of $\widetilde{\bg}, \gse_1, \ldots, \gse_{T-1}$ and $\vse_{1}, \ldots, \vse_{T-1}$.  For notational convenience, we let $\eta: \mathbb{R}^d \times \mathbb{R}^{d \times (T - 1)} \times \mathbb{R}^{d \times (T-1)}$ denote this function so that $\vse_T = \eta(\widetilde{\bg}, \Vse_{T-1}, \Gse_{T-1})$.  We thus have
	\begin{align*}
	\zeta^v_{1}(\bA, \bB) &= \mathbb{E}\bigl[\tau^v(\Vse_T, \Gse_T) \mid \Vse_{T-1} = \bA, \Gse_{T-1} = \bB\bigr] \\
	&= \mathbb{E}\Bigl[\tau^v\Bigl(\eta(\widetilde{\bg}, \Vse_{T-1}, \Gse_{T-1}), \Vse_{T-1}, \alpha_T \widetilde{\bg} + \sum_{\ell=1}^{T-1} \alpha_{\ell} \gse_{\ell}, \Gse_{T-1}\Bigr) \mid \Vse_{T-1} = \bA, \Gse_{T-1} = \bB\Bigr]\\
	&= \mathbb{E}\Bigl[\tau^v\Bigl(\eta(\widetilde{\bg}, \bA, \bB), \bA, \alpha_T \widetilde{\bg} + \sum_{\ell=1}^{T-1} \alpha_{\ell} \bb_{\ell}, \bB\Bigr)\Bigr],
	\end{align*}
where the first equality is almost sure and the expectation in the final relation is only with respect to $\widetilde{\bg}$.  Invoking the pseudo-Lipschitz nature of $\tau^{v}$ in conjunction with the Lipschitz nature of $\eta$ yields
\begin{align*}
\zeta^v_{1}(\bA, \bB) - &\zeta^v_{1}(\bA', \bB') \\
&= \mathbb{E}\Bigl[\tau^v\Bigl(\eta(\widetilde{\bg}, \bA, \bB), \bA, \alpha_T \widetilde{\bg} + \sum_{\ell=1}^{T-1} \alpha_{\ell} \bb_{\ell}, \bB\Bigr) - \tau^v\Bigl(\eta(\widetilde{\bg}, \bA', \bB'), \bA', \alpha_T \widetilde{\bg} + \sum_{\ell=1}^{T-1} \alpha_{\ell} \bb_{\ell}', \bB'\Bigr)\Bigr]\\
&\lesssim \mathbb{E}\Bigl[\Bigl\{1 + \| \eta(\widetilde{\bg}, \bA, \bB) \|_2 + \| \eta(\widetilde{\bg}, \bA', \bB') \|_2 +  \| \widetilde{\bg}\|_2 + \sum_{\ell=1}^{T-1} \| \ba_{\ell} \|_2 + \| \ba_{\ell}' \|_2 + \| \bb_{\ell} \|_2 + \| \bb_{\ell}' \|_2 \Bigr\} \\
& \qquad \qquad \cdot \Bigl\{\| \eta(\widetilde{\bg}, \bA, \bB) - \eta(\widetilde{\bg}, \bA', \bB') \|_2 + \sum_{\ell=1}^{T-1} \| \ba_{\ell} - \ba_{\ell}' \|_2 + \| \bb_{\ell} - \bb_{\ell}' \|_2 \Bigr\}\Bigr]\\
&\lesssim \Bigl\{1 + \sum_{\ell=1}^{T-1} \| \ba_{\ell} \_2 + \| \ba_{\ell}' \|_2 + \| \bb_{\ell} \|_2 + \| \bb_{\ell}' \|_2 \Bigr\} \cdot \sum_{\ell=1}^{T-1} \| \ba_{\ell} - \ba_{\ell}' \|_2 + \| \bb_{\ell} - \bb_{\ell}' \|_2,
\end{align*}
where the final inequality again uses $\mathbb{E}[\| \widetilde{\bg} \|_2] \lesssim 1$. 
	
\medskip
\noindent \underline{Conclusion for $k > 1$.} To conclude for $k > 2$, note that
	\begin{align*}
		\zeta_2^{v}(\bA, \bB) &= \mathbb{E}\bigl[ \tau(\Vse_T, \Gse_T) \mid \Vse_{T-2} = \bA, \Gse_{T-2} = \bB\bigr]\\
		&= \mathbb{E}\bigl[ \zeta_1^{v}(\vse_{T-1}, \Vse_{T-2}, \gse_{T-1}, \Gse_{T-2}) \mid \Vse_{T-2} = \bA, \Gse_{T-2} = \bB\bigr],
	\end{align*}
	and proceed inductively since $\zeta_1^{v}$ is pseudo-Lipschitz. 
\qed

\subsubsection{Proof of Lemma~\ref{lem:X-decomp}}

In parallel with the SE notation, define
\[
\widehat{\bv}_k = \sum_{\ell = 1}^{k} \bigl(L^{v}_k\bigr)_{k, \ell} \bv_{\ell} \quad \text{and} \quad \widehat{\bu}_k = \sum_{\ell = 1}^{k} \bigl(L^{u}_k\bigr)_{k, \ell} \bu_{\ell},
\]
and stack these as columns to form the matrices $\widehat{\bV}_k$ and $\widehat{\bU}_k$, respectively.

Recall that $\bX^{\parallel} = \proj_{\bU_k} \bX + \bX \proj_{\bV_k} - \proj_{\bU_k} \bX \proj_{\bV_k}$, and that $\bT_g = \bG_k (\bK^g_k)^{-1} \bU_k^\top =  \widehat{\bV}_k (\bK^g_k)^{-1} \bU_k^\top - \bX^\top \bU_k (\bK^g_k)^{-1} \bU_k^{\top}$ and $\bT_h = \sqrt{\Lambda} \cdot \bH_k (\bK^h_k)^{-1} \bV_k^\top = \bX \bV_k (\bK^h_k)^{-1} \bV_k^{\top} - \widehat{\bU}_k (\bK^{h}_k)^{-1} \bV_k^{\top}$, where we have used the definitions~\eqref{eq:PO-g-h-def} of $\bG_k$ and $\bH_k$. We thus have 
\[
\bigl \| \bX^{\parallel} - (-\bT_g^{\top} + \bT_h) \bigr \|_{\mathrm{op}}  \leq I. + II. + III.,
\]
where 
\begin{align*}
	I. &= \bigl \| \proj_{\bU_k} \bX - \bU_k (\bK^g_k)^{-1} \bU_k^{\top} \bX\bigr\|_{\mathrm{op}}, \quad II. = \bigl\|  \bX \proj_{\bV_k} - \bX \bV_k (\bK^h_k)^{-1} \bV_k^{\top} \bigr\|_{\mathrm{op}} \\ &\quad \text{and} \quad
	III. = \bigl \| \proj_{\bU_k} \bX \proj_{\bV_k} - \bigl( \widehat{\bU}_k (\bK^{h}_k)^{-1} \bV_k^{\top} -  \bU_k^{\top} (\bK^g_k)^{-1} \widehat{\bV}_k \bigr) \bigr \|_{\mathrm{op}}.
\end{align*}
We bound each term in turn, beginning with term $I.$  To this end, we note that by the Woodbury matrix identity, for two invertible matrices $\bA, \bB \in \mathbb{R}^{k \times k}$, the following holds
\begin{align} 
	\| \bA^{-1} - \bB^{-1} \|_{\mathrm{op}} = \| \bB^{-1} (\bA- \bB) \bA^{-1} \|_{\mathrm{op}} &\leq \frac{1}{\lambda_{\min} (\bA) \lambda_{\min} (\bB)} \bigl\| \bA - \bB \bigr\|_{\mathrm{op}} \notag \\
	&\leq \frac{k}{\lambda_{\min} (\bA) \lambda_{\min} (\bB)} \| \bA - \bB \|_{\infty}. \label{eq:woodbury-consequence}
\end{align}
Hence, since $\proj_{\bU_k} = \bU_k ( \bU_k^{\top} \bU_k)^{-1} \bU_k^{\top}$, we deduce that on $\Omega_k$,
\begin{align*}
	\bigl \| \proj_{\bU_k} \bX - \bU_k \bK_g^{-1} (\bU_k)^{\top} \bX\bigr\|_{\mathrm{op}} 
	&\leq \Delta_0(k, \delta) \cdot \| \bU_k \|_{\mathrm{op}}^{2} \bigl \| (\bU_k^{\top} \bU_k)^{-1} - (\bK^{g}_k)^{-1} \bigr\|_{\mathrm{op}}\\
	&\overset{\1}{\leq} C_{\mathrm{SE}}\frac{k \cdot \| \bU_k \|^2_{\mathrm{op}} \cdot \Delta_0(k,\delta) \cdot \Delta_1(k, \delta)}{\lambda_{\min}\bigl( \bigl\{\bigl(\bU_k\bigr)^{\top} \bU_k\bigr\}^{-1}\bigr) \lambda_{\min}(\bK^g_k)} \\
	&\overset{\2}{\leq} C_{\mathrm{SE}}\frac{k \cdot \| \bK^g_k \|_{\mathrm{op}} \Delta_0(k,\delta) \cdot \Delta_1(k, \delta)}{ \lambda_{\min}^2(\bK^g_k)}  \\
	&\leq C_{\mathrm{SE}} \cdot k  \cdot \Delta_0(k,\delta) \cdot \Delta_1(k, \delta),
\end{align*}
where step $\1$ follows from Eq.~\eqref{eq:woodbury-consequence} and step $\2$ follows by taking $c_{\mathrm{SE}}$ in the assumed lower bound on $\delta$ small enough and applying the eigenvalue stability inequality~\citep[e.g.,][Eq. (1.63)]{tao2012topics}.  Term $II.$ can be bounded in an identical manner, so we omit the details for brevity.  

Turning to term $III.$, we expand 
\[
\proj_{\bU_k} \bX \proj_{\bV_k} = \bU_k (\bU_k^{\top} \bU_k)^{-1} \bU_k^{\top} \bX \bV_k (\bV_k^{\top} \bV_k)^{-1} \bV_k^{\top}.
\]
Since by definition~\eqref{eq:PO-g-h-def}, $\bX \bV_k = \bH_k - \bU_k \bigl(\bL_k^{u}\bigr)^{\top}$, we see that on $\Omega_k$, 
\[
\bigl \| \bU_k^{\top} \bX \bV_k - \bigl\{  \bigl(\bL_k^{v}\bigr)^{\top} \bK^h_k - \bK^g_k \bL_k^{u}\bigr\} \bigr \|_{\mathrm{op}} \leq C_{\mathrm{SE}} \cdot k \cdot \Delta_1(k, \delta).
\]
Further, following the same steps as in the bound on term $I.$ yields the pair of inequalities 
\begin{align*}
	\bigl\| (\bU_k^{\top} \bU_k)^{-1} - (\bK^g_h)^{-1} \|_{\mathrm{op}} &\leq C_{\mathrm{SE}} \cdot k \cdot \Delta_1(k, \delta) \quad \text{and} \\
	\bigl\| (\bV_k^{\top} \bV_k)^{-1} - (\bK^h_k)^{-1} \|_{\mathrm{op}} &\leq C_{\mathrm{SE}} \cdot k  \cdot \Delta_1(k, \delta).
\end{align*}
Combining these three inequalities along with the assumption $\Delta_1(k, \delta) \leq T^{-1}$ yields the inequality 
\[
III. \leq C_{\mathrm{SE}} \cdot k \cdot \Delta_1(k, \delta). 
\]
Combining the pieces yields the desired result.   The second part of the lemma follows from Lemma~\ref{lem:conditioning}. 
\qed

\section{Discussion} \label{sec:discussion}

In this paper, we provided a rigorous state evolution for the class of saddle point updates (possibly interleaved with first-order updates) as well as finite-sample guarantees. We hope that the tools introduced in this paper both for proving uniqueness of state evolution in nonseparable settings and finite-sample concentration results will find broader application in studying the dynamics of learning algorithms. Multiple intriguing open questions remain and we detail a few here.

First, we made simplifying assumptions in this paper that we believe can be weakened.  For instance, it would be of interest to consider penalty functions in our saddle updates that are not smooth and strongly convex. Second, while we gave an explicit description of the state evolution for a large class of methods, we did not use the state evolution to derive any optimization-theoretic consequences. 
This has been the focus of some previous work in the sample-split setting~\citep{chandrasekher2023sharp}, and it would be interesting to provide similar characterizations without sample-splitting. Indeed, our companion paper makes progress in this direction for some concrete iterative algorithms.
Third, we restricted our attention to the setting in which the entries of the data matrix $\bX$ are standard Gaussian.  It is known in the literature on AMP and first-order methods that these predictions are to some extent universal~\citep{chen2021universality,han2024entrywise,lovig2025universality}.  Given the universality of $M$-estimators~\citep[see, e.g.,][]{han2023universality} under delocalization assumptions (see also~\citet{oymak2018universality}), we expect such results to hold for the broader class of iterations considered here and it would be interesting to show such a result rigorously.  
Fourth, and as previously mentioned, non-asymptotic guarantees for saddle updates suffer from a known drawback of working directly with loss functions in the CGMT~\citep{miolane2021distribution}.  It was shown in, e.g.,~\citet{chandrasekher2024alternating} that, for a class of sample-split algorithms, the leave-one-out method of~\citet{bean2013optimal,karoui2013asymptotic} can be utilized to provide sharper non-asymptotic guarantees. It would be interesting to prove such results in the present case, without sample-splitting. The finite-sample deviation rate also currently scales as $C_{\mathrm{SE}} (T!)^{2}$. While we argued that this may be impossible to substantially improve for general algorithms, it would be interesting to establish faster concentration results for specific first-order and saddle algorithms.
Fifth, in the context of mean-field spin glasses, lower bounds on the performance of iterative methods  have been shown for large classes of algorithms~\citep{huang2025tight}. It would be interesting to study the relationship between this class of algorithms and our saddle updates, and to use this understanding to design optimal families of saddle updates. Finally, it is of interest to go beyond the saddle update formalism to encompass more methods, such as those in numerical optimization~\citep{conn2000trust} and semiparametric inference~\citep{kakade2011efficient,pananjady2021single}. This would likely require significantly different techniques from those introduced in the present paper.

\subsection*{Acknowledgements}
AP and KAV were supported in part by the National Science Foundation through grants CCF-2107455 and DMS-2210734, and by research awards from Adobe, Amazon, Google and Mathworks.  Part of this work was completed while MC was supported
by the Miller Institute for Basic Research in Science, University of California, Berkeley.  
AP thanks Johannes Milz for helpful discussions, and for pointing him to the book by~\cite{barbu2012convexity}.

\bibliographystyle{abbrvnat}
\bibliography{gordon} 

\begin{thebibliography}{97}
\providecommand{\natexlab}[1]{#1}
\providecommand{\url}[1]{\texttt{#1}}
\expandafter\ifx\csname urlstyle\endcsname\relax
  \providecommand{\doi}[1]{doi: #1}\else
  \providecommand{\doi}{doi: \begingroup \urlstyle{rm}\Url}\fi

\bibitem[Abbe et~al.(2023)Abbe, Adsera, and Misiakiewicz]{abbe2023sgd}
E.~Abbe, E.~B. Adsera, and T.~Misiakiewicz.
\newblock {SGD} learning on neural networks: {L}eap complexity and
  saddle-to-saddle dynamics.
\newblock In \emph{Conference on Learning Theory (COLT)}, 2023.

\bibitem[Agarwal et~al.(2012)Agarwal, Negahban, and
  Wainwright]{agarwal2012fast}
A.~Agarwal, S.~Negahban, and M.~J. Wainwright.
\newblock Fast global convergence of gradient methods for high-dimensional
  statistical recovery.
\newblock \emph{The Annals of Statistics}, 40\penalty0 (5):\penalty0
  2452--2482, 2012.

\bibitem[Arous et~al.(2021)Arous, Gheissari, and Jagannath]{arous2021online}
G.~B. Arous, R.~Gheissari, and A.~Jagannath.
\newblock Online stochastic gradient descent on non-convex losses from
  high-dimensional inference.
\newblock \emph{Journal of Machine Learning Research}, 22\penalty0
  (106):\penalty0 1--51, 2021.

\bibitem[Arous et~al.(2024)Arous, Gheissari, and Jagannath]{arous2024high}
G.~B. Arous, R.~Gheissari, and A.~Jagannath.
\newblock High-dimensional limit theorems for {SGD}: {E}ffective dynamics and
  critical scaling.
\newblock \emph{Communications on Pure and Applied Mathematics}, 77\penalty0
  (3):\penalty0 2030--2080, 2024.

\bibitem[Bai and Silverstein(2010)]{bai2010spectral}
Z.~Bai and J.~W. Silverstein.
\newblock \emph{Spectral Analysis of Large Dimensional Random Matrices}.
\newblock Springer, 2010.

\bibitem[Balakrishnan et~al.(2017)Balakrishnan, Wainwright, and
  Yu]{balakrishnan2017statistical}
S.~Balakrishnan, M.~J. Wainwright, and B.~Yu.
\newblock Statistical guarantees for the {EM} algorithm: {F}rom population to
  sample-based analysis.
\newblock \emph{The Annals of Statistics}, 45\penalty0 (1):\penalty0 77--120,
  2017.

\bibitem[Barbu and Precupanu(2012)]{barbu2012convexity}
V.~Barbu and T.~Precupanu.
\newblock \emph{Convexity and Optimization in Banach Spaces}.
\newblock Springer Science \& Business Media, 2012.

\bibitem[Bauschke and Combettes(2017)]{bauschke2017convex}
H.~H. Bauschke and P.~L. Combettes.
\newblock \emph{Convex Analysis and Monotone Operator Theory in {H}ilbert
  Spaces}.
\newblock Springer, 2017.

\bibitem[Bayati and Montanari(2011)]{bayati2011dynamics}
M.~Bayati and A.~Montanari.
\newblock The dynamics of message passing on dense graphs, with applications to
  compressed sensing.
\newblock \emph{IEEE Transactions on Information Theory}, 57\penalty0
  (2):\penalty0 764--785, 2011.

\bibitem[Bean et~al.(2013)Bean, Bickel, El~Karoui, and Yu]{bean2013optimal}
D.~Bean, P.~J. Bickel, N.~El~Karoui, and B.~Yu.
\newblock Optimal {$M$}-estimation in high-dimensional regression.
\newblock \emph{Proceedings of the National Academy of Sciences}, 110\penalty0
  (36):\penalty0 14563--14568, 2013.

\bibitem[Bellec and Koriyama(2023)]{bellec2023existence}
P.~C. Bellec and T.~Koriyama.
\newblock Existence of solutions to the nonlinear equations characterizing the
  precise error of {M}-estimators.
\newblock \emph{arXiv preprint arXiv:2312.13254}, 2023.

\bibitem[Bellec and Tan(2024)]{bellec2024uncertainty}
P.~C. Bellec and K.~Tan.
\newblock Uncertainty quantification for iterative algorithms in linear models
  with application to early stopping.
\newblock \emph{arXiv preprint arXiv:2404.17856}, 2024.

\bibitem[Berthier et~al.(2020)Berthier, Montanari, and
  Nguyen]{berthier2020state}
R.~Berthier, A.~Montanari, and P.-M. Nguyen.
\newblock State evolution for approximate message passing with non-separable
  functions.
\newblock \emph{Information and Inference: A Journal of the IMA}, 9\penalty0
  (1):\penalty0 33--79, 2020.

\bibitem[Blumensath and Davies(2008)]{blumensath2008iterative}
T.~Blumensath and M.~E. Davies.
\newblock Iterative thresholding for sparse approximations.
\newblock \emph{Journal of Fourier analysis and Applications}, 14\penalty0
  (5):\penalty0 629--654, 2008.

\bibitem[Bolthausen(2014)]{bolthausen2014iterative}
E.~Bolthausen.
\newblock An iterative construction of solutions of the tap equations for the
  {S}herrington--{K}irkpatrick model.
\newblock \emph{Communications in Mathematical Physics}, 325\penalty0
  (1):\penalty0 333--366, 2014.

\bibitem[Bordelon and Pehlevan(2023)]{bordelon2023dynamics}
B.~Bordelon and C.~Pehlevan.
\newblock Dynamics of finite width kernel and prediction fluctuations in mean
  field neural networks.
\newblock In \emph{Advances in Neural Information Processing Systems
  (NeurIPS)}, 2023.

\bibitem[Bottou et~al.(2018)Bottou, Curtis, and
  Nocedal]{bottou2018optimization}
L.~Bottou, F.~E. Curtis, and J.~Nocedal.
\newblock Optimization methods for large-scale machine learning.
\newblock \emph{SIAM Review}, 60\penalty0 (2):\penalty0 223--311, 2018.

\bibitem[Boyd and Vandenberghe(2004)]{boyd2004convex}
S.~P. Boyd and L.~Vandenberghe.
\newblock \emph{Convex {O}ptimization}.
\newblock Cambridge University Press, 2004.

\bibitem[Cademartori and Rush(2024)]{cademartori2024non}
C.~Cademartori and C.~Rush.
\newblock A non-asymptotic analysis of generalized vector approximate message
  passing algorithms with rotationally invariant designs.
\newblock \emph{IEEE Transactions on Information Theory}, 70\penalty0
  (8):\penalty0 5811--5856, 2024.

\bibitem[Celentano(2024)]{celentano2024sudakov}
M.~Celentano.
\newblock {S}udakov--{F}ernique post-{AMP}, and a new proof of the local
  convexity of the {TAP} free energy.
\newblock \emph{The Annals of Probability}, 52\penalty0 (3):\penalty0 923--954,
  2024.

\bibitem[Celentano and Montanari(2024)]{celentano2021cad}
M.~Celentano and A.~Montanari.
\newblock Correlation adjusted debiased {L}asso: Debiasing the lasso with
  inaccurate covariate model.
\newblock \emph{Journal of the Royal Statistical Society Series B: Statistical
  Methodology}, 86\penalty0 (5):\penalty0 1455--1482, 2024.

\bibitem[Celentano and Wainwright(2023)]{celentano2023challenges}
M.~Celentano and M.~J. Wainwright.
\newblock Challenges of the inconsistency regime: {N}ovel debiasing methods for
  missing data models.
\newblock \emph{arXiv preprint arXiv:2309.01362}, 2023.

\bibitem[Celentano et~al.(2020)Celentano, Montanari, and
  Wu]{celentano2020estimation}
M.~Celentano, A.~Montanari, and Y.~Wu.
\newblock The estimation error of general first order methods.
\newblock In \emph{Conference on Learning Theory (COLT)}, 2020.

\bibitem[Celentano et~al.(2021)Celentano, Cheng, and
  Montanari]{celentano2021high}
M.~Celentano, C.~Cheng, and A.~Montanari.
\newblock The high-dimensional asymptotics of first order methods with random
  data.
\newblock \emph{arXiv preprint arXiv:2112.07572}, 2021.

\bibitem[Celentano et~al.(2023)Celentano, Fan, and Mei]{celentano2023local}
M.~Celentano, Z.~Fan, and S.~Mei.
\newblock Local convexity of the {TAP} free energy and {AMP} convergence for
  $\mathbb{Z}_2$-synchronization.
\newblock \emph{The Annals of Statistics}, 51\penalty0 (2):\penalty0 519--546,
  2023.

\bibitem[Chandrasekher et~al.(2023)Chandrasekher, Pananjady, and
  Thrampoulidis]{chandrasekher2023sharp}
K.~A. Chandrasekher, A.~Pananjady, and C.~Thrampoulidis.
\newblock Sharp global convergence guarantees for iterative nonconvex
  optimization with random data.
\newblock \emph{The Annals of Statistics}, 51\penalty0 (1):\penalty0 179--210,
  2023.

\bibitem[Chandrasekher et~al.(2024)Chandrasekher, Lou, and
  Pananjady]{chandrasekher2024alternating}
K.~A. Chandrasekher, M.~Lou, and A.~Pananjady.
\newblock Alternating minimization for generalized rank-$1$ matrix sensing:
  {S}harp predictions from a random initialization.
\newblock \emph{Information and Inference: A Journal of the IMA}, 13\penalty0
  (3), 2024.

\bibitem[Chang and Pollard(1997)]{chang1997conditioning}
J.~T. Chang and D.~Pollard.
\newblock Conditioning as disintegration.
\newblock \emph{Statistica Neerlandica}, 51\penalty0 (3):\penalty0 287--317,
  1997.

\bibitem[Chen and Lam(2021)]{chen2021universality}
W.-K. Chen and W.-K. Lam.
\newblock Universality of approximate message passing algorithms.
\newblock \emph{Electronic Journal of Probability}, 26:\penalty0 1--44, 2021.

\bibitem[Chi et~al.(2019)Chi, Lu, and Chen]{chi2019nonconvex}
Y.~Chi, Y.~M. Lu, and Y.~Chen.
\newblock Nonconvex optimization meets low-rank matrix factorization: An
  overview.
\newblock \emph{IEEE Transactions on Signal Processing}, 67\penalty0
  (20):\penalty0 5239--5269, 2019.

\bibitem[Collins-Woodfin et~al.(2024)Collins-Woodfin, Paquette, Paquette, and
  Seroussi]{collins2024hitting}
E.~Collins-Woodfin, C.~Paquette, E.~Paquette, and I.~Seroussi.
\newblock Hitting the high-dimensional notes: An {ODE} for {SGD} learning
  dynamics on {GLM}s and multi-index models.
\newblock \emph{Information and Inference: A Journal of the IMA}, 13\penalty0
  (4), 2024.

\bibitem[Conn et~al.(2000)Conn, Gould, and Toint]{conn2000trust}
A.~R. Conn, N.~I. Gould, and P.~L. Toint.
\newblock \emph{Trust-region Methods}.
\newblock SIAM, 2000.

\bibitem[Couillet and Liao(2022)]{couillet2022random}
R.~Couillet and Z.~Liao.
\newblock \emph{Random Matrix Methods for Machine Learning}.
\newblock Cambridge University Press, 2022.

\bibitem[Dempster et~al.(1977)Dempster, Laird, and Rubin]{dempster1977maximum}
A.~P. Dempster, N.~M. Laird, and D.~B. Rubin.
\newblock Maximum likelihood from incomplete data via the {EM} algorithm.
\newblock \emph{Journal of the Royal Statistical Society: Series B
  (methodological)}, 39\penalty0 (1):\penalty0 1--22, 1977.

\bibitem[Donoho and Montanari(2016)]{donoho2016high}
D.~Donoho and A.~Montanari.
\newblock High dimensional robust {$M$}-estimation: {A}symptotic variance via
  approximate message passing.
\newblock \emph{Probability Theory and Related Fields}, 166\penalty0
  (3):\penalty0 935--969, 2016.

\bibitem[Donoho et~al.(2009)Donoho, Maleki, and Montanari]{donoho2009message}
D.~L. Donoho, A.~Maleki, and A.~Montanari.
\newblock Message-passing algorithms for compressed sensing.
\newblock \emph{Proceedings of the National Academy of Sciences}, 106\penalty0
  (45):\penalty0 18914--18919, 2009.

\bibitem[Drusvyatskiy(2018)]{drusvyatskiy2018proximal}
D.~Drusvyatskiy.
\newblock The proximal point method revisited.
\newblock \emph{SIAG/OPT Views and News}, 18\penalty0 (1), 2018.

\bibitem[Drusvyatskiy et~al.(2021)Drusvyatskiy, Ioffe, and
  Lewis]{drusvyatskiy2021nonsmooth}
D.~Drusvyatskiy, A.~D. Ioffe, and A.~S. Lewis.
\newblock Nonsmooth optimization using {T}aylor-like models: {E}rror bounds,
  convergence, and termination criteria.
\newblock \emph{Mathematical Programming}, 185:\penalty0 357--383, 2021.

\bibitem[El~Karoui(2013)]{karoui2013asymptotic}
N.~El~Karoui.
\newblock Asymptotic behavior of unregularized and ridge-regularized
  high-dimensional robust regression estimators: Rigorous results.
\newblock \emph{arXiv preprint arXiv:1311.2445}, 2013.

\bibitem[Fan et~al.(2025)Fan, Ko, Loureiro, Lu, and Shen]{fan2025dynamical}
Z.~Fan, J.~Ko, B.~Loureiro, Y.~M. Lu, and Y.~Shen.
\newblock Dynamical mean-field analysis of adaptive {L}angevin diffusions:
  Propagation-of-chaos and convergence of the linear response.
\newblock \emph{arXiv preprint arXiv:2504.15556}, 2025.

\bibitem[Feng et~al.(2022)Feng, Venkataramanan, Rush, and
  Samworth]{feng2022unifying}
O.~Y. Feng, R.~Venkataramanan, C.~Rush, and R.~J. Samworth.
\newblock A unifying tutorial on approximate message passing.
\newblock \emph{Foundations and Trends{\textregistered} in Machine Learning},
  15\penalty0 (4):\penalty0 335--536, 2022.

\bibitem[Ferbach et~al.(2025)Ferbach, Everett, Gidel, Paquette, and
  Paquette]{ferbach2025dimension}
D.~Ferbach, K.~Everett, G.~Gidel, E.~Paquette, and C.~Paquette.
\newblock Dimension-adapted momentum outscales {SGD}.
\newblock \emph{arXiv preprint arXiv:2505.16098}, 2025.

\bibitem[Fienup(1982)]{fienup1982phase}
J.~R. Fienup.
\newblock Phase retrieval algorithms: {A} comparison.
\newblock \emph{Applied Optics}, 21\penalty0 (15):\penalty0 2758--2769, 1982.

\bibitem[Fridovich-Keil et~al.(2024)Fridovich-Keil, Valdivia, Wetzstein, Recht,
  and Soltanolkotabi]{fridovich2023gradient}
S.~Fridovich-Keil, F.~Valdivia, G.~Wetzstein, B.~Recht, and M.~Soltanolkotabi.
\newblock {G}radient descent provably solves nonlinear tomographic
  reconstruction.
\newblock \emph{arXiv preprint arXiv:2310.03956}, 2024.

\bibitem[Gamarnik et~al.(2022)Gamarnik, Moore, and
  Zdeborov{\'a}]{gamarnik2022disordered}
D.~Gamarnik, C.~Moore, and L.~Zdeborov{\'a}.
\newblock Disordered systems insights on computational hardness.
\newblock \emph{Journal of Statistical Mechanics: Theory and Experiment},
  2022\penalty0 (11):\penalty0 114015, 2022.

\bibitem[Ge et~al.(2016)Ge, Lee, and Ma]{ge2016matrix}
R.~Ge, J.~D. Lee, and T.~Ma.
\newblock Matrix completion has no spurious local minimum.
\newblock In \emph{Advances in Neural Information Processing Systems
  (NeurIPS)}, 2016.

\bibitem[Ge et~al.(2017)Ge, Jin, and Zheng]{ge2017no}
R.~Ge, C.~Jin, and Y.~Zheng.
\newblock No spurious local minima in nonconvex low rank problems: {A} unified
  geometric analysis.
\newblock In \emph{International Conference on Machine Learning (ICML)}, 2017.

\bibitem[Gerbelot et~al.(2024)Gerbelot, Troiani, Mignacco, Krzakala, and
  Zdeborova]{gerbelot2024rigorous}
C.~Gerbelot, E.~Troiani, F.~Mignacco, F.~Krzakala, and L.~Zdeborova.
\newblock Rigorous dynamical mean-field theory for stochastic gradient descent
  methods.
\newblock \emph{SIAM Journal on Mathematics of Data Science}, 6\penalty0
  (2):\penalty0 400--427, 2024.

\bibitem[Gerchberg and Saxton(1972)]{gerchberg1972practical}
R.~W. Gerchberg and W.~O. Saxton.
\newblock A practical algorithm for the determination of plane from image and
  diffraction pictures.
\newblock \emph{Optik}, 35\penalty0 (2):\penalty0 237--246, 1972.

\bibitem[Gordon(1985)]{gordon1985some}
Y.~Gordon.
\newblock Some inequalities for {G}aussian processes and applications.
\newblock \emph{Israel Journal of Mathematics}, 50\penalty0 (4):\penalty0
  265--289, 1985.

\bibitem[Gordon(1988)]{gordon1988milman}
Y.~Gordon.
\newblock On {M}ilman's inequality and random subspaces which escape through a
  mesh in $\mathbb{R}^n$.
\newblock \emph{Geometric Aspects of Functional Analysis}, pages 84--106, 1988.

\bibitem[Han(2024)]{han2024entrywise}
Q.~Han.
\newblock Entrywise dynamics and universality of general first order methods.
\newblock \emph{arXiv preprint arXiv:2406.19061}, 2024.

\bibitem[Han and Shen(2023)]{han2023universality}
Q.~Han and Y.~Shen.
\newblock Universality of regularized regression estimators in high dimensions.
\newblock \emph{The Annals of Statistics}, 51\penalty0 (4):\penalty0
  1799--1823, 2023.

\bibitem[Han and Xu(2024)]{han2024gradient}
Q.~Han and X.~Xu.
\newblock Gradient descent inference in empirical risk minimization.
\newblock \emph{arXiv preprint arXiv:2412.09498}, 2024.

\bibitem[Huang and Sellke(2025)]{huang2025tight}
B.~Huang and M.~Sellke.
\newblock Tight {L}ipschitz hardness for optimizing mean field spin glasses.
\newblock \emph{Communications on Pure and Applied Mathematics}, 78\penalty0
  (1):\penalty0 60--119, 2025.

\bibitem[Jain and Kar(2017)]{jain2017non}
P.~Jain and P.~Kar.
\newblock Non-convex optimization for machine learning.
\newblock \emph{Foundations and Trends{\textregistered} in Machine Learning},
  10\penalty0 (3-4):\penalty0 142--363, 2017.

\bibitem[Javanmard and Montanari(2013)]{javanmard2013state}
A.~Javanmard and A.~Montanari.
\newblock State evolution for general approximate message passing algorithms,
  with applications to spatial coupling.
\newblock \emph{Information and Inference: A Journal of the IMA}, 2\penalty0
  (2):\penalty0 115--144, 2013.

\bibitem[Jekel et~al.(2025)Jekel, Sandhu, and Shi]{jekel2025potential}
D.~Jekel, J.~S. Sandhu, and J.~Shi.
\newblock Potential {H}essian ascent: The {S}herrington--{K}irkpatrick model.
\newblock In \emph{Symposium on Discrete Algorithms (SODA)}, 2025.

\bibitem[Kakade et~al.(2011)Kakade, Kanade, Shamir, and
  Kalai]{kakade2011efficient}
S.~M. Kakade, V.~Kanade, O.~Shamir, and A.~Kalai.
\newblock Efficient learning of generalized linear and single index models with
  isotonic regression.
\newblock In \emph{Advances in Neural Information Processing Systems
  (NeurIPS)}, 2011.

\bibitem[Kaushik et~al.(2024)Kaushik, Romberg, and
  Muthukumar]{kaushik2024precise}
C.~Kaushik, J.~Romberg, and V.~Muthukumar.
\newblock Precise asymptotics of reweighted least-squares algorithms for linear
  diagonal networks.
\newblock \emph{arXiv preprint arXiv:2406.02769}, 2024.

\bibitem[Lan(2020)]{lan2020first}
G.~Lan.
\newblock \emph{First-order and Stochastic Optimization Methods for Machine
  Learning}.
\newblock Springer, 2020.

\bibitem[Lewis and Wright(2016)]{lewis2016proximal}
A.~S. Lewis and S.~J. Wright.
\newblock A proximal method for composite minimization.
\newblock \emph{Mathematical Programming}, 158\penalty0 (1):\penalty0 501--546,
  2016.

\bibitem[Li and Wei(2022)]{li2022non}
G.~Li and Y.~Wei.
\newblock A non-asymptotic framework for approximate message passing in spiked
  models.
\newblock \emph{arXiv preprint arXiv:2208.03313}, 2022.

\bibitem[Li and Wei(2024)]{li2024non}
G.~Li and Y.~Wei.
\newblock A non-asymptotic distributional theory of approximate message passing
  for sparse and robust regression.
\newblock \emph{arXiv preprint arXiv:2401.03923}, 2024.

\bibitem[Liang et~al.(2023)Liang, Sen, and Sur]{liang2023high}
T.~Liang, S.~Sen, and P.~Sur.
\newblock High-dimensional asymptotics of {L}angevin dynamics in spiked matrix
  models.
\newblock \emph{Information and Inference: A Journal of the IMA}, 12\penalty0
  (4), 2023.

\bibitem[Lou et~al.(2024)Lou, Verchand, and Pananjady]{lou2024hyperparameter}
M.~Lou, K.~A. Verchand, and A.~Pananjady.
\newblock Hyperparameter tuning via trajectory predictions: {S}tochastic
  prox-linear methods in matrix sensing.
\newblock \emph{arXiv preprint arXiv:2402.01599}, 2024.

\bibitem[Lou et~al.(2025)Lou, Verchand, Fridovich-Keil, and
  Pananjady]{lou2025accurate}
M.~Lou, K.~A. Verchand, S.~Fridovich-Keil, and A.~Pananjady.
\newblock Accurate, provable, and fast nonlinear tomographic reconstruction: A
  variational inequality approach.
\newblock \emph{arXiv preprint arXiv:2503.19925}, 2025.

\bibitem[Lovig et~al.(2025)Lovig, Wang, and Fan]{lovig2025universality}
M.~Lovig, T.~Wang, and Z.~Fan.
\newblock On universality of non-separable approximate message passing
  algorithms.
\newblock \emph{arXiv preprint arXiv:2506.23010}, 2025.

\bibitem[Mannelli et~al.(2019)Mannelli, Krzakala, Urbani, and
  Zdeborova]{mannelli2019passed}
S.~S. Mannelli, F.~Krzakala, P.~Urbani, and L.~Zdeborova.
\newblock Passed \& spurious: Descent algorithms and local minima in spiked
  matrix-tensor models.
\newblock In \emph{International Conference on Machine Learning (ICML)}, 2019.

\bibitem[Mignacco et~al.(2020)Mignacco, Krzakala, Urbani, and
  Zdeborov{\'a}]{mignacco2020dynamical}
F.~Mignacco, F.~Krzakala, P.~Urbani, and L.~Zdeborov{\'a}.
\newblock Dynamical mean-field theory for stochastic gradient descent in
  {G}aussian mixture classification.
\newblock In \emph{Advances in Neural Information Processing Systems
  (NeurIPS)}, 2020.

\bibitem[Milman and Schechtman(1986)]{milman1986asymptotic}
V.~D. Milman and G.~Schechtman.
\newblock \emph{Asymptotic Theory of Finite Dimensional Normed Spaces}.
\newblock Springer, 1986.

\bibitem[Miolane and Montanari(2021)]{miolane2021distribution}
L.~Miolane and A.~Montanari.
\newblock The distribution of the {L}asso: {U}niform control over sparse balls
  and adaptive parameter tuning.
\newblock \emph{The Annals of Statistics}, 49:\penalty0 2313--2335, 2021.

\bibitem[Mondelli and Venkataramanan(2021)]{mondelli2021approximate}
M.~Mondelli and R.~Venkataramanan.
\newblock Approximate message passing with spectral initialization for
  generalized linear models.
\newblock In \emph{International Conference on Artificial Intelligence and
  Statistics (AISTATS)}, 2021.

\bibitem[Montanari and Subag(2025)]{montanari2025solving}
A.~Montanari and E.~Subag.
\newblock Solving systems of random equations via first and second-order
  optimization algorithms.
\newblock \emph{arXiv preprint arXiv:2306.13326}, 2025.

\bibitem[Montanari and Wu(2023)]{montanari2023adversarial}
A.~Montanari and Y.~Wu.
\newblock Adversarial examples in random neural networks with general
  activations.
\newblock \emph{Mathematical Statistics and Learning}, 6\penalty0 (1):\penalty0
  143--200, 2023.

\bibitem[Montanari et~al.(2025)Montanari, Ruan, Sohn, and
  Yan]{montanari2025generalization}
A.~Montanari, F.~Ruan, Y.~Sohn, and J.~Yan.
\newblock The generalization error of max-margin linear classifiers: {B}enign
  overfitting and high dimensional asymptotics in the overparametrized regime.
\newblock \emph{The Annals of Statistics}, 53\penalty0 (2):\penalty0 822--853,
  2025.

\bibitem[Nemirovski and Yudin(1983)]{nemirovski1983problem}
A.~Nemirovski and D.~Yudin.
\newblock \emph{Problem Complexity and Method Efficiency in Optimization}.
\newblock Wiley-Interscience, 1983.

\bibitem[Nesterov(2018)]{nesterov2018lectures}
Y.~Nesterov.
\newblock \emph{Lectures on convex optimization}, volume 137.
\newblock Springer, 2018.

\bibitem[Nocedal and Wright(2006)]{nocedal2006numerical}
J.~Nocedal and S.~J. Wright.
\newblock \emph{Numerical optimization}.
\newblock Springer, 2006.

\bibitem[Okajima and Takahashi(2025)]{okajima2025asymptotic}
K.~Okajima and T.~Takahashi.
\newblock Asymptotic dynamics of alternating minimization for bilinear
  regression.
\newblock \emph{Journal of Statistical Mechanics: Theory and Experiment},
  2025\penalty0 (5):\penalty0 053301, 2025.

\bibitem[Oymak and Tropp(2018)]{oymak2018universality}
S.~Oymak and J.~A. Tropp.
\newblock Universality laws for randomized dimension reduction, with
  applications.
\newblock \emph{Information and Inference: A Journal of the IMA}, 7\penalty0
  (3):\penalty0 337--446, 2018.

\bibitem[Pananjady and Foster(2021)]{pananjady2021single}
A.~Pananjady and D.~P. Foster.
\newblock Single-index models in the high signal regime.
\newblock \emph{IEEE Transactions on Information Theory}, 67\penalty0
  (6):\penalty0 4092--4124, 2021.

\bibitem[Rangan(2011)]{rangan2011generalized}
S.~Rangan.
\newblock Generalized approximate message passing for estimation with random
  linear mixing.
\newblock In \emph{International Symposium on Information Theory (ISIT)}, 2011.

\bibitem[Richardson and Urbanke(2008)]{richardson2008modern}
T.~Richardson and R.~Urbanke.
\newblock \emph{Modern {C}oding {T}heory}.
\newblock Cambridge University Press, 2008.

\bibitem[Rockafellar(1971)]{rockafellar1971saddle}
R.~T. Rockafellar.
\newblock Saddle-points and convex analysis.
\newblock \emph{Differential Games and Related Topics}, 109, 1971.

\bibitem[Rockafellar(1976)]{rockafellar1976monotone}
R.~T. Rockafellar.
\newblock Monotone operators and the proximal point algorithm.
\newblock \emph{SIAM Journal on Control and Optimization}, 14\penalty0
  (5):\penalty0 877--898, 1976.

\bibitem[Rockafellar(1991)]{rockafellar1991special}
R.~T. Rockafellar.
\newblock On a special class of convex functions.
\newblock \emph{Journal of Optimization Theory and Applications}, 70:\penalty0
  619--621, 1991.

\bibitem[Rush and Venkataramanan(2018)]{rush2018finite}
C.~Rush and R.~Venkataramanan.
\newblock Finite sample analysis of approximate message passing algorithms.
\newblock \emph{IEEE Transactions on Information Theory}, 64\penalty0
  (11):\penalty0 7264--7286, 2018.

\bibitem[Stojnic(2013)]{stojnic2013framework}
M.~Stojnic.
\newblock A framework to characterize performance of {L}asso algorithms.
\newblock \emph{arXiv preprint arXiv:1303.7291}, 2013.

\bibitem[Subag(2021)]{subag2021following}
E.~Subag.
\newblock Following the ground states of full-{RSB} spherical spin glasses.
\newblock \emph{Communications on Pure and Applied Mathematics}, 74\penalty0
  (5):\penalty0 1021--1044, 2021.

\bibitem[Tan and Vershynin(2023)]{tan2023online}
Y.~S. Tan and R.~Vershynin.
\newblock Online stochastic gradient descent with arbitrary initialization
  solves non-smooth, non-convex phase retrieval.
\newblock \emph{Journal of Machine Learning Research}, 24\penalty0
  (58):\penalty0 1--47, 2023.

\bibitem[Tao(2012)]{tao2012topics}
T.~Tao.
\newblock \emph{Topics in Random Matrix Theory}.
\newblock American Mathematical Society, 2012.

\bibitem[Thrampoulidis et~al.(2015)Thrampoulidis, Oymak, and
  Hassibi]{thrampoulidis2015regularized}
C.~Thrampoulidis, S.~Oymak, and B.~Hassibi.
\newblock Regularized linear regression: A precise analysis of the estimation
  error.
\newblock In \emph{Conference on Learning Theory (COLT)}, 2015.

\bibitem[Thrampoulidis et~al.(2018)Thrampoulidis, Abbasi, and
  Hassibi]{thrampoulidis2018precise}
C.~Thrampoulidis, E.~Abbasi, and B.~Hassibi.
\newblock Precise error analysis of regularized {$M$}-estimators in high
  dimensions.
\newblock \emph{IEEE Transactions on Information Theory}, 64\penalty0
  (8):\penalty0 5592--5628, 2018.

\bibitem[Tse and Viswanath(2005)]{tse2005fundamentals}
D.~Tse and P.~Viswanath.
\newblock \emph{Fundamentals of {W}ireless {C}ommunication}.
\newblock Cambridge University Press, 2005.

\bibitem[Vershynin(2018)]{vershynin2018high}
R.~Vershynin.
\newblock \emph{{High-dimensional Probability: An Introduction with
  Applications in Data Science}}.
\newblock Cambridge University Press, 2018.

\bibitem[Wainwright(2019)]{wainwright2019high}
M.~J. Wainwright.
\newblock \emph{High-Dimensional Statistics: A Non-Asymptotic Viewpoint}.
\newblock Cambridge University Press, 2019.

\end{thebibliography}

\vspace*{.1in}

\appendix
\begin{center}
{\bf{\Large{Appendix}}}
\end{center}

\section{Examples illustrating parameterization of state evolution} \label{sec:SE-scalars}

\subsection{Recovering AMP state evolution} \label{sec:SE-AMP-example}
Consider the expanded AMP dynamics in Eq.~\eqref{eq:AMP-linear-expanded}.  Note that the corresponding state evolution given by Eqs.~\eqref{eq:SE-base-case}--\eqref{eq:fix-pt} is
\begin{align*}
		&\vse_{2k} = \vse_{2k-1} \qquad &\text{ and } \qquad \qquad &\use_{2k} =  \by_{2k}^{\mathrm{SE}} + \uhatse_{2k-1} - \sqrt{\Lambda}\hse_{2k-1} + \lambda_{k-1} \use_{2k-1} , \nonumber\\
	&\vse_{2k+1} = \lambda_k \cdot \bigl(\vhatse_{2k} - \gse_{2k} + \Lambda \vse_{2k}\bigr) \qquad &\text{ and } \qquad \qquad &\use_{2k+1} = \use_{2k},
\end{align*}
where $\by_{2k}^{\mathrm{SE}} \sim \mathsf{N}(0, \sigma^2 \bI_n)$, independent of everything else with initialization $\vse_1 = \bv_1$, $\use_1 = \bu_1$, $\gse_1 = \mathsf{N}(0, \| \bv_1 \|_2^2 \bI_d/d)$, $\hse_1 = \mathsf{N}(0, \| \bu_1 \|_2^2 \bI_n/n)$, $\vhatse_1 = 0$, and $\uhatse_1 = 0$.  We claim that $\uhatse_{2k+2} = \uhatse_{2k+1} = - \lambda_{k} \use_{2k}$ and $\vhatse_{2k+1} = \vhatse_{2k} = -\Lambda \vse_{2k-1}$.  Since $\use_{2k-1} = \use_{2k-2}$ and $\vse_{2k} = \vse_{2k-1}$, this implies that 
\begin{align*}
	&\vse_{2k} = \vse_{2k-1} \qquad &\text{ and } \qquad \qquad &\use_{2k} =  \by^{\mathrm{SE}}_{2k}  - \sqrt{\Lambda}\hse_{2k-1}, \nonumber\\
	&\vse_{2k+1} = -\lambda_k \gse_{2k}  \qquad &\text{ and } \qquad \qquad &\use_{2k+1} = \use_{2k}.
\end{align*}
We additionally take $\gse_{2k} \sim \mathsf{N}(0, \tau_k^2)$, independent of all other randomness, and $\gse_{2k+1} = \gse_{2k}$ as well as $\hse_{2k+1} \sim \mathsf{N}\bigl(0, \lambda_k^2 \cdot \{ \sigma^2 + \Lambda \tau_{k-1}^2\}\bigr)$, independent of all other randomness, and $\hse_{2k} = \hse_{2k-1}$ where $\tau_k^2 = \sigma^2 + \Lambda \lambda_{k-1}^2\tau_{k-1}^2$.

Let us now verify that these choices verify the fixed point equations in Eq.~\eqref{eq:fix-pt}.  We proceed by induction.  Note that the base case holds by construction as the fixed point equations are indeed verified at the first iteration.  Suppose that it holds for all positive integers $\ell \leq 2k-1$.  From the induction hypothesis, since $\vse_{2k} = \vse_{2k-1}$ and $\hse_{2k} = \hse_{2k-1}$, we see that $\llangle \Vse_{2k}, \Vse_{2k} \rrangle_{L^2} = \llangle \Hse_{2k}, \Hse_{2k} \rrangle_{L^2}$.  Note further that
\[
\langle \use_{2k}, \use_{2k} \rangle_{L^2} = \langle \by^{\mathrm{SE}}_{2k} - \sqrt{\Lambda} \hse_{2k-1}, \by^{\mathrm{SE}}_{2k} - \sqrt{\Lambda} \hse_{2k-1} \rangle_{L^2} = \sigma^2 + \Lambda \| \hse_{2k-1} \|_{L^2}^2 = \sigma^2 + \Lambda \lambda_{k-1}^2 \tau_{k-1}^2 = \tau_{k}^2, 
\]
which is equivalent to $\langle \gse_{2k}, \gse_{2k} \rangle_{L^2}$ by construction.  Moreover, for any $\ell \leq 2k-1$, we have
\[
\langle \use_{2k}, \use_{\ell} \rangle_{L^2} = \langle \by^{\mathrm{SE}}_{2k} - \sqrt{\Lambda} \hse_{2k-1}, \use_{\ell} \rangle_{L^2} = \langle \by^{\mathrm{SE}}_{2k}, \use_{\ell}\rangle_{L^2} - \sqrt{\Lambda} \langle \hse_{2k-1}, \use_{\ell} \rangle = 0,
\]
where the final inequality follows from independence.  By construction, we see that $\langle \gse_{2k}, \gse_{\ell} \rangle_{L^2} = 0$, so that $\llangle \Use_{2k}, \Use_{2k} \rrangle_{L^2} = \llangle \Gse_{2k}, \Gse_{2k} \rrangle_{L^2}$.  Additionally, we have, for each $\ell \leq k$ that 
\[
\langle \gse_{2k}, \vse_{2\ell} \rangle_{L^2} = \langle \gse_{2k}, -\lambda_{\ell-1} \gse_{2\ell-2} \rangle_{L^2} = -\lambda_{\ell-1} \langle \use_{2k}, \use_{2\ell-2} \rangle_{L^2} = \langle \use_{2k}, \uhatse_{2\ell} \rangle_{L^2},
\]
and for each $\ell \leq k-1$ that
\[
\langle \gse_{2k}, \vse_{2\ell + 1} \rangle_{L^2} = \langle \gse_{2k}, -\lambda_{\ell} \gse_{2\ell} \rangle_{L^2} = -\lambda_{\ell} \langle \use_{2k}, \use_{2\ell} \rangle_{L^2} = \langle \use_{2k}, \uhatse_{2\ell+1} \rangle_{L^2}.
\]
Thus, $\llangle \Gse_{2k}, \Vse_{2k} \rrangle_{L^2} = \llangle \Use_{2k}, \Uhatse_{2k} \rrangle_{L^2}$.  It remains to verify that $\sqrt{\Lambda} \llangle \Hse_{2k}, \Use_{2k} \rrangle_{L^2} = \llangle \Vse_{2k}, \Vhatse_{2k} \rrangle_{L^2}$.  We have 
\[
\sqrt{\Lambda} \langle \hse_{2k}, \use_{2k} \rangle_{L^2} = \sqrt{\Lambda} \langle \hse_{2k-1}, \by^{\mathrm{SE}}_{2k} - \sqrt{\Lambda} \hse_{2k-1} \rangle_{L^2} = -\Lambda \| \hse_{2k-1} \|_{L^2}^2. 
\]
On the other hand,
\[
\langle \vse_{2k}, \vhatse_{2k} \rangle_{L^2} = \langle \vse_{2k-1}, -\Lambda \vse_{2k-1} \rangle_{L^2} = -\Lambda \| \vse_{2k-1} \|_{L^2}^2 = -\Lambda \| \hse_{2k-1} \|_{L^2}^2. 
\]
Moreover, for each $\ell \leq k-1$, we have
\[
\sqrt{\Lambda} \langle \hse_{2k}, \use_{2\ell} \rangle_{L^2} = 0 = \langle \vse_{2k}, \vhatse_{2\ell} \rangle_{L^2} \qquad \text{ and } \qquad \sqrt{\Lambda} \langle \hse_{2k}, \use_{2\ell + 1} \rangle_{L^2} = 0 = \langle \vse_{2k}, \vhatse_{2\ell + 1} \rangle_{L^2}.
\]
A similar set of calculations then verifies that $\sqrt{\Lambda} \llangle \Hse_{2k}, \Use_{2k} \rrangle_{L^2} = \llangle \Vse_{2k}, \Vhatse_{2k} \rrangle_{L^2}$, as desired.

\subsection{Rigorous state evolution for one step of CGMT}\label{sec:appendix-m-est-gordon}
In this section, we use the state evolution in Eqs.~\eqref{eq:se-guass-and-span}--\eqref{eq:fix-pt} to recover the nonlinear system in~\citet[Eq. (15)]{thrampoulidis2018precise} (see also~\citet[Eq. (1.11)]{bellec2023existence}). We are interested in regularized $M$-estimation problems where we minimize the loss
\begin{align*}
L(\bvtilde) &= \frac{1}{d}\rho(\by - \bX \bvtilde) + \frac{\lambda}{d} f(\bvtilde) = \frac{1}{d}\sum_{i=1}^{n} \rho(y_i - \bx_i^{\top} \bvtilde) +\frac{ \lambda}{d} \sum_{j=1}^{d} f(\widetilde{v}_j)
\end{align*}
where 
$\rho$ is smooth and separable and $f$ is strongly convex and separable.
For simplicity of exposition, we will consider the setting in which the observations $\by = 0$ (with true signal equal to $0$), whereby we obtain
\begin{align*}
L(\bvtilde) &= \overset{(d)}{=}\frac{1}{d}\rho(\bX \bvtilde) + \frac{\lambda}{d} f(\bvtilde) = \frac{1}{d}\sum_{i=1}^{n} \rho(\bx_i^{\top} \bvtilde) +\frac{ \lambda}{d} \sum_{j=1}^{d} f(\widetilde{v}_j).
\end{align*}
Note that our normalization matches that of~\citet[see Eq. (64)]{thrampoulidis2018precise}.  Dualizing and scaling the decision variables so that $\bv = \bvtilde / \sqrt{d}$, we obtain the equivalent representation
\[
L(\bv) = \sup_{\bu \in \mathbb{R}^n} \; \bu^{\top} \bX \bv - \frac{1}{d} \sum_{j=1}^{d} \partial \rho^{\star}\bigl( u_j \sqrt{d} \bigr) + \frac{\lambda}{d} \sum_{j=1}^{d} f\bigl(v_j \sqrt{d}\bigr),
\]
where $\rho^{\star}$ denotes the convex conjugate of $\rho$. Note that the minimizer $\bv^{\mathsf{opt}}$ of this loss can be written in conjunction with the maximizing $\bu^{\mathsf{opt}}$ as
\[
(\bu^{\mathsf{opt}}, \bv^{\mathsf{opt}}) = \underset{\bv \in \mathbb{R}^d, \bu \in \mathbb{R}^n}\saddle \Big\{ \bu^{\top} \bX \bv - \frac{1}{d} \sum_{j=1}^{d} \partial \rho^{\star}\bigl( u_j \sqrt{d} \bigr) + \frac{\lambda}{d} \sum_{j=1}^{d} f\bigl(v_j \sqrt{d}\bigr) \Big\}.
\]

The state evolution from Eqs.~\eqref{eq:se-guass-and-span}--\eqref{eq:fix-pt} now applies, and yields
\begin{subequations} \label{eq:Gordon-M-SE}
\begin{align}
	\sqrt{\Lambda} \hse - \uhatse
	&= \frac{1}{\sqrt{d}} \bigl[\partial \rho^{\star}\bigl(u^{\mathrm{SE}}_1 \sqrt{d}\bigr) \; \vert \; \cdots \; \vert \; \partial \rho^{\star}\bigl(u^{\mathrm{SE}}_n \sqrt{d}\bigr)\bigr]^{\top} \quad \text{ and } \label{eq:ex-gordon-use}\\ 
	\gse -\vhatse
	&=
	\frac{\lambda}{\sqrt{d}} \bigl[\partial f\bigl(v^{\mathrm{SE}}_1 \sqrt{d}\bigr) \; \vert \; \cdots \; \vert \; \partial f\bigl(v^{\mathrm{SE}}_d \sqrt{d}\bigr)\bigr]^{\top}. \label{eq:ex-gordon-vse}
\end{align}
Using the span constraints to take $\uhatse = \kappa \use$, $\vhatse = \nu \vse$ and letting $\| \use \|_{L^2}^2 = \beta^2$, and $\| \vse \|_{L^2}^2 = \alpha^2$ then simplifies the fixed point equations to
\begin{equation}
	\begin{aligned}
		 \beta^2 &= \| \gse \|_{L^2}^2 = \| \use \|_{L^2}^2,
		\qquad&
		\alpha^2 = \| \hse \|_{L^2}^2 &= \| \vse \|_{L^2}^2
		\\
		\langle \gse, \vse \rangle_{L^2} &= \langle \uhatse, \use \rangle_{L^2} = \kappa \beta^2,
		\qquad&
		\langle \sqrt{\Lambda} \hse, \use \rangle_{L^2} &= \langle \vse, \vhatse \rangle_{L^2} = \nu \alpha^2. \label{ex:gordon-fix-pt}
	\end{aligned}
\end{equation}
\end{subequations}
We next show that the system of equations~\eqref{eq:Gordon-M-SE} is equivalent to the system of equations in~\citet[Eq. (1.11)]{bellec2023existence}; see Eq.~\eqref{eq:BK-SE}.   
Let us show equivalence of Eqs.~\eqref{eq:BK-SE} with our state evolution~\eqref{eq:Gordon-M-SE} by considering each equation in turn. We begin with Eq.~\eqref{eq:example-gordon-norm-v} and define the random variable $V^{\mathrm{BK}} := \mathsf{prox}_{\lambda \nu^{-1} f}(\nu^{-1} \beta Z)$ so that $V^{\mathrm{BK}}$ satisfies the KKT condition
\[
\nu V^{\mathrm{BK}} -  \beta Z + \lambda  \partial f(V^{\mathrm{BK}}) = 0.
\]
We recognize from~\eqref{eq:ex-gordon-vse} that, since $\vhatse = \kappa \vse$,
\[
\nu v^{\mathrm{SE}}_j - g^{\mathrm{SE}}_j + \frac{\lambda}{\sqrt{d}} \partial f\bigl(v^{\mathrm{SE}}_j \sqrt{d}\bigr) = 0 \quad \text{ for all } j \in [d].
\]
It thus follows that the entries of $\sqrt{d} \vse$ are i.i.d. with the same law as $V^{\mathrm{BK}}$ so that Eq.~\eqref{eq:example-gordon-norm-v} is equivalent to $\alpha^2 = \| \vse \|_{L^2}^2$.  

We next turn to Eq.~\eqref{eq:example-gordon-norm-u}.  Note that by Moreau's identity,
\[
\alpha W - \mathsf{prox}_{\kappa \rho}(\alpha W) = \mathsf{prox}_{(\kappa \rho)^{\star}}(\alpha W)
\]
Define $U^{\mathrm{BK}} := \frac{1}{\kappa}\mathsf{prox}_{(\kappa \rho)^{\star}}(\alpha W)$,
so that Eq.~\eqref{eq:example-gordon-norm-u} is equivalent to $\beta^2 = \Lambda \mathbb{E}[(U^{\mathrm{BK}})^2]$.  Next, noting that $(\kappa \rho)^{\star}(x) = \kappa \rho^{\star}(x/\kappa)$, we note that $U^{\mathrm{BK}}$ satisfies the KKT condition
\[
\kappa U^{\mathrm{BK}} - \alpha W + \partial \rho^{\star} \bigl(U^{\mathrm{BK}}\bigr) = 0.
\]
On the other hand, from Eq.~\eqref{eq:ex-gordon-use}, we see that 
\[
\kappa u^{\mathrm{SE}}_i - \sqrt{\Lambda} h^{\mathrm{SE}}_i + \frac{1}{\sqrt{d}} \partial \rho^{\star}\bigl(u^{\mathrm{SE}}_i \sqrt{d}\bigr) = 0 \quad \text{ for all } i \in [n].
\]
It thus follows that the entries of $\sqrt{d} \use$ are i.i.d. with the same law as $U^{\mathrm{BK}}$ so that $\| \use \|_{L^2}^2 = \Lambda \mathbb{E}[(U^{\mathrm{BK}})^2]$.  Hence, Eq.~\eqref{eq:example-gordon-norm-u} is equivalent to $\beta^2 \| \use \|_{L^2}^2$.

We next turn to Eq.~\eqref{eq:example-gordon-uh}, which is equivalent to the equation
$\nu \alpha = \Lambda \mathbb{E}\bigl[W U^{\mathrm{BK}}]$.  On the other hand, our state evolution implies that
\begin{align}\label{eq:relation-penultimate-uh-example}
\nu \alpha^2 = \nu \| \vse \|_{L^2}^2 = \langle \vse, \vhatse \rangle_{L^2} = \sqrt{\Lambda} \langle \hse, \use\rangle_{L^2}.  
\end{align}
Note that the entries of $\sqrt{d\Lambda} \hse$ are i.i.d. with the same law as $\mathsf{N}(0, \alpha^2)$ and the entries of $\sqrt{d} \use$ are i.i.d. with the same law as $U^{\mathrm{BK}}$, so that
\[
\sqrt{\Lambda} \langle \hse, \use\rangle_{L^2} = \Lambda \alpha \mathbb{E}[W U^{\mathrm{BK}}].
\]
Re-arranging we thus see that Eq.~\eqref{eq:relation-penultimate-uh-example} is equivalent to Eq.~\eqref{eq:example-gordon-uh}.  

We finally consider Eq.~\eqref{eq:example-gordon-vg}.  By definition, we have $ \mathbb{E}\Bigl[Z \cdot \mathsf{prox}_{\lambda \nu^{-1} f}(\nu^{-1} \beta Z)\Bigr] = \mathbb{E}[Z V^{\mathrm{BK}}]$.  On the other hand, since the entries of $\sqrt{d} \vse$ are i.i.d. with the same law as $V^{\mathrm{BK}}$ and the entries of $\sqrt{d} \gse$ are i.i.d. $\mathsf{N}(0, \beta^2)$, and invoking the fixed point relation in Eq.~\eqref{ex:gordon-fix-pt}, we have
\[
\kappa \beta^2 = \langle \gse, \vse \rangle_{L^2} = \beta \mathbb{E} [Z V^{\mathrm{BK}}],
\]
which in turn is equivalent to Eq.~\eqref{eq:example-gordon-vg}, as desired.

\section{Proofs of technical lemmas from Section~\ref{sec:proof-exact-asymptotics}}

In this section, we prove the technical lemmas that were stated in our proofs of the finite-sample guarantees. We first prove Lemma~\ref{lem:facts}, which was used in the first-order induction step, and then Lemmas~\ref{lem:approximation-by-po},~\ref{lem:apx-concentration-useful}, and~\ref{lem:approximate-stationarity-F-func}, which were used in the saddle induction step.

\subsection{Proof of Lemma~\ref{lem:facts}}
We prove each of the statements in turn, restricting our proofs to inequalities (a) in the left column; inequalities (b) admit identical proofs.  Starting with inequality~\eqref{eq:facts1}(a), we expand
\[
\bT_{g}^{\top} \bv_{k+1} = \bU_k\bigl(\bK_k^{g}\bigr)^{-1} \bG_k^{\top} \bv_{k+1} = \sum_{\ell=1}^{k} \bigl\{\bigl(\bK_k^{g}\bigr)^{-1} \bG_k^{\top} \bv_{k+1}\bigr\}_{\ell} \bu_{\ell}.
\]
Since for first-order methods, the events $\{\Vse_{k+1} = \bV_{k+1}, \Gse_{k} = \bG_{k}, \Use_{k+1} = \bU_{k+1}, \Hse_{k} = \bH_k\}$ and $\{\Vse_k = \bV_k, \Gse_k = \bG_k, \Use_k = \bU_k, \Hse_k = \bH_k\}$ are equivalent, on $\Omega_k$, we have
\[
\bigl \| \bG_k^{\top} \bv_{k+1} - \llangle \Gse_k, \vse_{k+1} \rrangle_{L^2} \|_{\infty} \leq  \Delta_1(k, \delta).
\] 
As in the proof of Claim~\ref{claim:first-order-hat-value}, we expand the state evolution equations in~\eqref{eq:SE-FO-system} to deduce that 
\[
\llangle \Gse_k, \vse_{k+1} \rrangle_{L^2} = \bK^{g}_k [\bigl(L^{u}_{k+1}\bigr)_{k+1, 1} \; \vert \; \cdots \; \vert \; \bigl(L^{u}_{k+1}\bigr)_{k+1, k}].
\]
Putting the pieces together, we find that 
\[
\Bigl\| \bT_g^{\top} \bv_{k+1} - \sum_{\ell=1}^{k} \bigl(L^{u}_{k+1}\bigr)_{k+1, \ell} \bu_{\ell} \Bigr\|_2 \leq \sum_{\ell=1}^{k} \bigl \lvert \bigl\{\bigl(\bK_k^{g}\bigr)^{-1} \bG_k^{\top} \bv_{k+1}\bigr\}_{\ell} - \bigl(L^{u}_{k+1}\bigr)_{k+1, \ell} \bigr \rvert \leq C_{\mathrm{SE}} k \Delta_1(k, \delta),
\]
as desired. 

Turning to inequality~\eqref{eq:facts2}(a), we expand
\[
\bT_h \bv_{k+1} = \sqrt{\Lambda} \cdot \bH_k \bigl(\bK^h_k\bigr)^{-1} \bV_k^{\top} \bv_{k+1} = \sqrt{\Lambda} \sum_{\ell=1}^{k}\bigl\{\bigl(\bK_k^{h}\bigr)^{-1} \bV_k^{\top} \bv_{k+1}\bigr\}_{\ell} \bh_{\ell}.
\]
Similar to the previous argument, we note that on $\Omega_k$, 
\[
\bigl \| \bV_k^{\top} \bv_{k+1} - \llangle \Vse_k, \vse_{k+1} \rrangle_{L^2} \|_{\infty} \leq  \Delta_1(k, \delta).
\] 
Further, recall from the proof of Claim~\ref{claim:first-order-hat-value} that, for $ \ell \in [k]$,  $\beta_{k+1, \ell} = \bigl\{ \bigl(\bK^{h}_k\bigr)^{-1} \llangle \Vse_{k}, \vse_{k+1} \rrangle_{L^2}\bigr\}_{\ell}$.  Putting the pieces together, it thus follows that 
\[
\Bigl\| \bT_h \bv_{k+1} - \sqrt{\Lambda}\sum_{\ell=1}^{k} \beta_{k+1, \ell} \bh_{\ell} \Bigr\|_2 \leq \sqrt{\Lambda} \sum_{\ell=1}^{k} \bigl \lvert \bigl\{\bigl(\bK_k^{h}\bigr)^{-1} \bV_k^{\top} \bv_{k+1}\bigr\}_{\ell} - \beta_{k+1, \ell} \bigr \rvert \leq C_{\mathrm{SE}} k \Delta_1(k, \delta),
\]
as desired. 

Finally, we turn to inequality~\eqref{eq:facts3}(a).  We have
\[
\| \proj_{\bV_k}^{\perp} \bv_{k+1} \|_2^2 = \| \bv_{k+1} \|_2^2 - \| \proj_{\bV_k} \bv_{k+1}\|_2^2 = \| \bv_{k+1} \|_2^2 - \langle \bV_{k}^{\top} \bv_{k+1}, (\bV_k^{\top} \bV_k^{\top})^{-1} \bV_k^{\top} \bv_{k+1} \rangle. 
\]
Note that on $\Omega_k$, $\lvert \| \bv_{k+1} \|_2^2 - \| \vse_{k+1} \|_{L^2}^2 \rvert \leq \Delta_1(k, \delta)$.  We thus focus on the second term.  As in Eq.~\eqref{eq:woodbury-consequence}, we note that for two invertible matrices $\bA,\bB \in \mathbb{R}^{k \times k}$, it holds that
\[
\| \bA^{-1} - \bB^{-1} \|_{\mathrm{op}} \leq \frac{k}{\lambda_{\min}(\bA) \lambda_{\min}(\bB)} \| \bA- \bB \|_{\infty}.
\]
Hence on $\Omega_k$, 
\begin{align*}
\Bigl \lvert \langle \bV_{k}^{\top} \bv_{k+1}, (\bV_k^{\top} \bV_k^{\top})^{-1} \bV_k^{\top} \bv_{k+1} \rangle - \langle \bV_{k}^{\top} \bv_{k+1}, (\bK^{h}_k)^{-1} \bV_k^{\top} \bv_{k+1} \rangle \Bigr \rvert &\leq  \| \bV_k^{\top} \bv_{k+1} \|_2^2 \bigl \| (\bV_k^{\top} \bV_k^{\top})^{-1} - \bigl(\bK^h_k\bigr)^{-1} \bigr\|_{\mathrm{op}} \\
&\leq C_{\mathrm{SE}} k^2 \Delta_1(k, \delta).
\end{align*}
Further, applying the Cauchy--Schwartz inequality, we deduce that
\begin{align*}
	\Bigl \lvert \langle &\bV_{k}^{\top} \bv_{k+1}, (\bK^{h}_k)^{-1} \bV_k^{\top} \bv_{k+1} \rangle - \langle \llangle \Vse_k, \vse_{k+1} \rrangle, (\bK^{h}_k)^{-1} \llangle \Vse_{k}, \vse_{k+1} \rrangle \rangle \Bigr \rvert \\
	&\leq 2 \bigl \| (\bK^{h}_k)^{-1} \bigr \|_{\mathrm{op}} \bigl \| \bV_k^{\top} \bv_{k+1} - \llangle \Vse_{k}, \vse_{k+1} \rrangle \bigr \|_2 \cdot \Bigl\{ \bigl \| \bV_k^{\top} \bv_{k+1} \bigr\|_2 + \bigl \|  \llangle \Vse_{k}, \vse_{k+1} \rrangle \bigr\|_2 \Bigr\} \leq C_{\mathrm{SE}} k^2 \Delta_1(k, \delta)
\end{align*}
Putting the pieces together, we find that
\[
\Bigl \lvert \| \proj_{\bV_k}^{\perp} \bv_{k+1} \|_2^2 - \Bigl\{ \| \vse_{k+1} \|_{L^2}^2 - \langle \llangle \Vse_k, \vse_{k+1} \rrangle, (\bK^{h}_k)^{-1} \llangle \Vse_{k}, \vse_{k+1} \rrangle \Bigr \} \Bigr \rvert \leq C_{\mathrm{SE}} k^2 \Delta_1(k, \delta).
\]
Recognizing that 
\[
\beta_{k+1, k+1} = \bK^h_{k+1, k+1} - \bigl \langle \bigl[ \bigl(\bK^h_{k+1}\bigr)_{k+1, 1} \;\vert \; \cdots \; \vert \; \bigr(\bK^h_{k+1}\bigr)_{k+1, k}\bigr], (\bK^{h}_k)^{-1}\bigl[ \bigl(\bK^h_{k+1}\bigr)_{k+1, 1} \;\vert \; \cdots \; \vert \; \bigr(\bK^h_{k+1}\bigr)_{k+1, k}\bigr]\bigr \rangle,
\]
yields the desired result.
\qed


\subsection{Proof of Lemma~\ref{lem:approximation-by-po}} \label{sec:proof-lem-approximation-by-po}
Suppose as in the lemma that
\begin{align} \label{ineq:uniform-po-bound}
	\sup_{\bv \in \mathbb{B}_2(c_{\mathrm{SE}}), \bu \in \mathbb{B}_2(c_{\mathrm{SE}})}\; \Bigl \lvert \Lpo_{k+1}(\bu, \bv) - F_{k+1}(\bu, \bv) \Bigr \rvert \leq C_{\mathrm{SE}} \cdot k \cdot \Delta_1(k, \delta).
\end{align}
Note that we have the following sandwich relation:
\begin{align} \label{ineq:penultimate-step1-saddle}
	\frac{\mu}{2} \| \bv_{k+1} - \vpo_{k+1} \|_2^2 \overset{\1}{\leq} F_{k+1}(\bu_{k+1}, \vpo_{k+1}) - F(\bu_{k+1}, \bv_{k+1}) \overset{\2}{\leq} C_{\mathrm{SE}} \cdot k \cdot \Delta_1(k, \delta),
\end{align}
where the inequality $\1$ follows from strong convexity of the map $\bv \mapsto F_{k+1}(\bu_{k+1}, \bv)$ and the inequality $\2$ follows from combining the uniform bound~\eqref{ineq:uniform-po-bound} with the saddle relations
\begin{align*}
	&F_{k+1}(\upo_{k+1}, \bv_{k+1}) \leq F_{k+1}(\bu_{k+1}, \bv_{k+1}) \leq F_{k+1}(\bu_{k+1}, \vpo_{k+1}) \quad \text{ and }\\
	& \Lpo_{k+1}(\bu_{k+1}, \vpo_{k+1}) \leq \Lpo_{k+1}(\upo_{k+1}, \vpo_{k+1}) \leq \Lpo_{k+1}(\upo_{k+1}, \bv_{k+1}).
\end{align*}
Re-arranging inequality~\eqref{ineq:penultimate-step1-saddle} and repeating the same steps for $\bu_{k+1}$ thus yields the pair of claimed bounds
\begin{align}
	\label{eq:approximation-by-po-new}
	\| \vpo_{k+1} - \bv_{k+1} \|_2 \leq C_{\mathrm{SE}} \cdot\sqrt{ k \cdot \Delta_1(k, \delta)} \quad \text{ and } \quad \| \upo_{k+1} - \bu_{k+1} \|_2 \leq C_{\mathrm{SE}} \cdot\sqrt{ k \cdot \Delta_1(k, \delta)}.
\end{align}

It remains to prove that 
\begin{align}
	\label{ineq:boundedness-of-iterates}
	\| \bv_{k+1} \|_2 \vee \| \bu_{k+1} \|_2 \leq c_{\mathrm{SE}}.
\end{align}
Note that, since $\phi_{k+1}^{u}$ is strongly convex, the function $\gamma: \bv \mapsto \max_{\bu \in \mathbb{R}^n}\, F_{k+1}(\bu, \bv)$ is well-defined and strongly convex as it is the maximum of strongly convex functions.  Moreover, by Danskin's theorem~\cite[][Corollary 1]{rockafellar1971saddle}, $\nabla \gamma(\bv) = \bX^{\top} \argmax_{\bu \in \mathbb{R}^d} \{ \phi_{k+1}^{u}(\bu)\} + \nabla \phi_{k+1}^{v}(\bv)$.  Thus, since $\gamma$ is $\mu$-strongly convex, on $\Omega_k$, we have
\begin{align} \label{ineq:norm-bound-v}
	\bigl\| \bv_{k+1} \bigr\|_2 \overset{\1}{\leq} \frac{2}{\mu} \| \nabla \gamma(0) \|_2 &\leq \frac{2}{\mu} \cdot \bigl(\| \bX \|_{\mathrm{op}} \cdot \| \argmax_{\bu \in \mathbb{R}^d} \{ \phi_{k+1}^{u}(\bu)\} \|_2 + \| \nabla \phi_{k+1}^{v}(0) \|_2\bigr) \nonumber \\
	&\leq \frac{2}{\mu} \cdot \bigl(\Delta_0(k, \delta) + K\bigr) \leq c_{\mathrm{SE}}',
\end{align}
where in step $\1$ we have used the fact that for a differentiable, $\mu$-strongly convex function $f: \mathbb{R}^{m} \times \mathbb{R}$, $\| \argmin_{x \in \mathbb{R}^m}\, f(x) \|_2 \leq \frac{2}{\mu} \| \nabla f(0) \|_2$ and the final inequality follows from Assumption~\ref{asm:regularity}.  Continuing, since $(\bu_{k+1}, \bv_{k+1})$ is a saddle point of $F_{k+1}$, it satisfies the sandwich relation 
\[
F_{k+1}(\bu, \bv_{k+1}) \leq F(\bu_{k+1}, \bv_{k+1}) \leq F(\bu_{k+1}, \bv), \quad \text{ for all } \quad \bu \in \mathbb{R}^n, \bv \in \mathbb{R}^d. 
\]
Taking $\bu = 0$ in left-hand side of the sandwich relation above and invoking strong concavity of the map $\bu \mapsto F_{k+1}(\bu, \bv_{k+1})$ yields 
\begin{align*}
	0 &\geq F_{k+1}(0, \bv_{k+1}) - F_{k+1}(\bu_{k+1}, \bv_{k+1}) \geq - \nabla_{\bu} F_{k+1}(0, \bv_{k+1})^{\top} \bu_{k+1} + \frac{\mu}{2} \| \bu_{k+1} \|_2^2 \\
	&\geq - \| \nabla_{\bu} F_{k+1}(0, \bv_{k+1}) \|_2 \| \bu_{k+1} \|_2 + \frac{\mu}{2} \| \bu_{k+1}\|_2^2 \geq - \| \bu_{k+1} \|_2 \cdot \bigl( \| \bX \|_{\mathrm{op}} \| \bv_{k+1} \|_2 + \| \nabla \phi_{k+1}^{u}(0) \|_2 \bigr). 
\end{align*}
Re-arranging and invoking Assumption~\ref{asm:regularity} in conjunction with the bound on $\| \bv_{k+1}\|_2$~\eqref{ineq:norm-bound-v} then implies that $\| \bu_{k+1} \|_2 \leq c_{\mathrm{SE}}$, for a small enough constant $c_{\mathrm{SE}} > 0$.  We then take $(M, M')$ in the uniform bound~\eqref{ineq:uniform-po-bound} to be $(c_{\mathrm{SE}}, c_{\mathrm{SE}}')$, as desired. \hfill \qed

\subsection{Proof of Lemma~\ref{lem:apx-concentration-useful}}\label{sec:proof-lem-apx-concentration-useful}
Each of the proofs follows a similar template and proceeds by using a pseudo-Lipschitz property, so we restrict ourselves to the proof of~\eqref{eq:apx-ind-events-four-phiv}.  We will exploit Eq.~\eqref{eq:apx-se-equivalence}, reproduced below for convenience:
\begin{align*}
\vse_{k+1} = \vapx_{k+1}(\Vse_{k}, \Gse_k, \bgamma_d/\sqrt{d}) \quad \text{ and } \quad \use_{k+1} = \uapx_{k+1}(\Use_k, \Hse_k, \bgamma_n/\sqrt{n})
\end{align*}
for $\bgamma_d \sim \normal(0, \bI_d)$ and $\bgamma_d \sim \normal(0, \bI_n)$.

\paragraph{Proof of the inequality~\eqref{eq:apx-ind-events-four-phiv}.}
For $\ell \leq k$, on $\Omega_k$, 
\[
\bigl \lvert \mathbb{E}\bigl[\phi_{\ell}^{v}(\bv_{\ell}) \mid \Vse_{k} = \bV_k, \vse_{k+1} = \vapx_{k+1}, \Gse_{k} = \bG_k, \gse_{k+1} = \bg_{k+1}(\vapx_{k+1})\bigr] - \EE[\phi_{\ell}^{v}(\bv_{\ell})] \bigr \rvert \leq \Delta_1(\ell, \delta) \leq \Delta_1(k, \delta),
\]
so that it suffices to consider $\ell > k$.  In this case, Lemma~\ref{lem:conditional-expectation-PL} implies that the function $\tau_{\ell}: \mathbb{R}^{d} \times \mathbb{R}^{d \times k} \times \mathbb{R}^d \times \mathbb{R}^{d \times k}$, defined as
\[
\tau_{\ell}(\bc, \bC, \bd, \bD) := \mathbb{E}\bigl[\phi_{\ell}^{v}(\vse_{\ell}) \mid \Vse_{k} = \bC, \vse_{k+1} = \bc, \Gse_{k} = \bD, \gse_{k+1} = \bd\bigr]
\]
is order-$2$ pseudo-Lipschitz.  Recalling the definition of $\tau^{\phi_{\ell}^{v},k}$~\eqref{eq:tau-definitions}, we deduce that on $\Omega_k$
\begin{align*}
	\bigl \lvert \tau_{\ell}(\vapx_{k+1}, \bV_k, \bg_{k+1}(\vapx_{k+1}), \bG_k) - \EE\bigl[\phi_{\ell}^{v}(\vse_{\ell})\bigr] \bigr \rvert &\leq \bigl \lvert \tau_{\ell}(\vapx_{k+1}, \bV_k, \bg_{k+1}(\vapx_{k+1}), \bG_k) - \tau^{\phi_{\ell}^v, k}(\bV_k, \bG_k) \bigr \rvert \\
	&\hspace{4cm}+ \bigl \lvert \tau^{\phi_{\ell}^v, k}(\bV_k, \bG_k)  - \EE\bigl[\phi_{\ell}^{v}(\vse_{\ell})\bigr] \bigr \rvert \\
	&\leq \bigl \lvert \tau_{\ell}(\vapx_{k+1}, \bV_k, \bg_{k+1}(\vapx_{k+1}), \bG_k) - \tau^{\phi_{\ell}^v, k}(\bV_k, \bG_k) \bigr \rvert + \Delta_1(k, \delta).
\end{align*}
Now, recalling the definition of $\bg_{k+1}(\cdot)$ from~\eqref{eq:g-h-po}, we have
\begin{align*}
	\bg_{k+1}(\vapx_{k+1}) &= \bg_{k+1}(\vapx_{k+1}(\bV_k, \bG_k, \bgamma_d/\sqrt{d}))\\
	&= \bigl(L^{v}_{k+1}\bigr)_{k+1, k+1} \vapx_{k+1}(\bV_k, \bG_k, \bgamma_d/\sqrt{d}) + \sum_{\ell=1}^{k} \bigl(L^{v}_{k+1}\bigr)_{k+1, \ell} \bv_k + \nabla \phi_{k+1}^{v}(\vapx_{k+1}(\bV_k, \bG_k, \bgamma_d/\sqrt{d})).
\end{align*}
Since by assumption $\phi_{k+1}^{v}$ is smooth and $\vapx_{k+1}$ is a Lipschitz function of its arguments, it follows that $\bg_{k+1}(\vapx_{k+1})$ is a Lipschitz function of $\bV_k, \bG_k$, and $\bgamma_d/\sqrt{d}$.  Hence, $\tau_{\ell}$ is an order-$2$ pseudo-Lipschitz function of $\bV_k, \bG_k, \bgamma_{d}/\sqrt{d}$. Also note that due to the equivalence~\eqref{eq:apx-se-equivalence}, we have
\[
\tau^{\phi_{\ell}^v, k}(\bV_k, \bG_k)  = \EE_{\bgamma_d} [\tau_{\ell}(\vapx_{k+1}(\bV_k, \bG_k, \bgamma_d/\sqrt{d}), \bV_k, \bg_{k+1}(\vapx_{k+1}(\bV_k, \bG_k, \bgamma_d/\sqrt{d})), \bG_k) ].
\]
Thus, applying Lemma~\ref{lem:pseudo-Lipschitz-concentration} yields
\[
\bigl \lvert \tau_{\ell}(\vapx_{k+1}, \bV_k, \bg_{k+1}(\vapx_{k+1}), \bG_k) - \tau^{\phi_{\ell}^v, k}(\bV_k, \bG_k) \bigr \rvert \leq C_{\mathrm{SE}} \sqrt{\frac{\log(10T^2/\delta')}{n}}, \quad \text{ with probability } \quad \geq 1 - \frac{\delta'}{10T^2}.
\]
Putting the pieces together yields the desired result. 

To conclude, we note that there are $6T^2 + 2T + 2 \leq 10T^2$ events to be controlled, and applying a union bound completes the proof. \hfill \qed

\subsection{Proof of Lemma~\ref{lem:approximate-stationarity-F-func}}\label{sec:proof-lem-approximate-stationarity-F-func}
We first recall the Hilbert spaces $\mathcal{H}^{k+1}_n$ and $\mathcal{H}^{k + 1}_d$ from Eq.~\eqref{eq:Hilbert-defs} as well as the orthonormal collections $\{\vseperp_{\ell}\}_{\ell=1}^{k} \in \mathcal{H}^{k+1}_d, \{\useperp_{\ell}\}_{\ell=1}^{k} \in \mathcal{H}^{k+1}_n, \{\gseperp_{\ell}\}_{\ell=1}^{k} \in \mathcal{H}^{k+1}_d,$ and $\{\hseperp_{\ell}\}_{\ell=1}^{k} \in \mathcal{H}^{k+1}_n$ and collect them to form 
\begin{align*}
	\Vseperp_k &:= [\vseperp_1 \; \vert \; \cdots \; \vert \; \vseperp_k], \quad \Useperp_k := [\useperp_1 \; \vert \; \cdots \; \vert \; \useperp_k],\\
	\Gseperp_k &:= [\gseperp_1 \; \vert \; \cdots \; \vert \; \gseperp_k], \quad \Hseperp_k := [\hseperp_1 \; \vert \; \cdots \; \vert \; \hseperp_k].
\end{align*} 
Note further that by assumption, the matrices $\bK^g_k, \bK^h_k$ are invertible and act as whitening transformations, so that we have the identities 
\begin{align}\label{eq:SE-whitening}
	\Vseperp_k &= \Vse_k \bigl(\bK^h_k\bigr)^{-1/2}, \quad \Useperp_k = \Use_k \bigl(\bK^g_k\bigr)^{-1/2}, \nonumber\\
	\Gseperp_k &= \Gse_k \bigl(\bK^g_k\bigr)^{-1/2}, \quad \text{ and } \quad\Hseperp_k = \Hse_k \bigl(\bK^h_k\bigr)^{-1/2},
\end{align}
where these matrix multiplications should be understood as acting on the matrices realized by the random vectors.  We will additionally exploit Eq.~\eqref{eq:apx-se-equivalence}, i.e., the fact that 
\[
\vse_{k+1} = \vapx_{k+1}(\Vse_{k}, \Gse_k, \bgamma_d/\sqrt{d}) \quad \text{ and } \quad \use_{k+1} = \uapx_{k+1}(\Use_k, \Hse_k, \bgamma_n/\sqrt{n}).
\]  

Equipped with these relations, we prove each part of the lemma in turn.

\paragraph{Proof of the inequality~\eqref{ineq:gradient-bound-F-a}.}
We expand
\begin{align*}
H(\uapx_{k+1}) = \Bigl\{&\bigl(\uapx_{k+1}\bigr)^{\top} \bT_h \vapx_{k+1} + \sqrt{\Lambda}\| \proj^{\perp}_{\bV_k} \vapx_{k+1} \|_2 \frac{\bigl(\uapx_{k+1}\bigr)^{\top}\bgamma_n}{\sqrt{n}}\Bigr\} \\
&- \Bigl\{\bigl(\uapx_{k+1}\bigr)^{\top} \bT_g^{\top} \vapx_{k+1} + \bigl \| \proj^{\perp}_{\bU_k} \uapx_{k+1} \|_2 \cdot \langle \bxi_g, \bv^{\mathrm{SE}}_{k+1} \rangle_{L^2}\Bigr\} - \phi_{k+1}^{u} (\uapx_{k+1}) + \phi_{k+1}^{v}(\vapx_{k+1}),
\end{align*}
and 
\[
\mathfrak{L}_{k+1}(\use_{k+1}, \vse_{k+1}) = - \langle \gse_{k+1}, \vse_{k+1} \rangle_{L^2} + \sqrt{\Lambda} \langle \hse_{k+1}, \use_{k+1} \rangle_{L^2} - \EE[\phi_{k+1}^{u}(\use_{k+1})] + \EE[\phi_{k+1}^{v}(\vse_{k+1})].
\]
Hence, by the triangle inequality,
\[
\bigl \lvert H(\uapx_{k+1}) - \mathrm{OPT}_{k+1} \bigr \rvert = \bigl \lvert H(\uapx_{k+1}) - \mathfrak{L}(\use_{k+1}, \vse_{k+1}) \bigr \rvert \leq T_1 + T_2 + T_3 + T_4,
\]
where 
\begin{align*}
	T_1 &= \Bigl \lvert \bigl(\uapx_{k+1}\bigr)^{\top} \bT_h \vapx_{k+1} + \sqrt{\Lambda}\| \proj^{\perp}_{\bV_k} \vapx_{k+1} \|_2 \frac{\bigl(\uapx_{k+1}\bigr)^{\top}\bgamma_n}{\sqrt{n}} - \sqrt{\Lambda} \langle \hse_{k+1}, \use_{k+1} \rangle_{L^2} \Bigr \rvert,\\
	T_2 & = \Bigl \lvert \bigl(\uapx_{k+1}\bigr)^{\top} \bT_g^{\top} \vapx_{k+1} + \bigl \| \proj^{\perp}_{\bU_k} \uapx_{k+1} \|_2 \cdot \langle \bxi_g, \bv^{\mathrm{SE}}_{k+1} \rangle_{L^2} - \langle \gse_{k+1}, \vse_{k+1} \rangle_{L^2} \Bigr \rvert,\\
	T_3 &= \Bigl \lvert \phi_{k+1}^{u} (\uapx_{k+1}) - \EE[\phi_{k+1}^{u}(\use_{k+1})] \Bigr \rvert \quad \text{ and } \quad T_4 = \Bigl \lvert \phi_{k+1}^{v} (\vapx_{k+1}) - \EE[\phi_{k+1}^{v}(\vse_{k+1})] \Bigr \rvert.
\end{align*}
Let $\Omega'$ denote the event that each of the inequalities in Lemma~\ref{lem:apx-concentration-useful} hold for parameter $\delta'$, and for the remainder of the proof, work on the intersection $\Omega' \cap \Omega_k$.  Note that on this intersection, from the inequalities~\eqref{eq:apx-ind-events-four-phiu} and~\eqref{eq:apx-ind-events-four-phiv}, we have
\[
T_3 \vee T_4 \leq \Delta_1(k, \delta) + C_{\mathrm{SE}}\sqrt{\frac{\log(10T^2/\delta')}{n}}.
\]
We claim that
\begin{subequations}
	\begin{align}
		T_1 &\leq C_{\mathrm{SE}} k^2 \Delta_1(k, \delta) + C_{\mathrm{SE}}k^2\sqrt{\frac{\log(10T^2/\delta')}{n}}, \quad \text{ and } \label{ineq:apx-useful-T1-bound}\\
		T_2 &\leq C_{\mathrm{SE}} k^2 \Delta_1(k, \delta) + C_{\mathrm{SE}}k^2\sqrt{\frac{\log(10T^2/\delta')}{n}} \label{ineq:apx-useful-T2-bound}.
	\end{align}
\end{subequations}
Combining these bounds on terms $T_1, T_2, T_3$, and $T_4$ yields the desired result.  It remains to establish the inequalities~\eqref{ineq:apx-useful-T1-bound} and~\eqref{ineq:apx-useful-T2-bound}.  The proofs are nearly identical, so we explicitly prove only inequality~\eqref{ineq:apx-useful-T1-bound}.

\medskip
\noindent \underline{Proof of the inequality~\eqref{ineq:apx-useful-T1-bound}.}
Recall from Eq.~\eqref{eq:SE-g-hilbert-h-hilbert} that $\hse_{k+1} = \mathsf{h}[\vse_{k+1}]$ and
\begin{align}
	\langle \use_{k+1}, \hse_{k+1}\rangle_{L^2} &= \sum_{\ell=1}^{k} \langle \vseperp_{\ell}, \vse_{k+1} \rangle_{L^2} \langle \use_{k+1}, \hseperp_{\ell}\rangle_{L^2} + \| \proj_{\Vse_{k}}^{\perp} \vse_{k+1}\|_{L^2} \langle \use_{k+1}, \bxi_h\rangle_{L^2} \notag \\
	&= \llangle \use_{k+1}, \Hseperp_k\rrangle_{L^2} \cdot  \llangle \Vseperp_k,  \vse_{k+1} \rrangle_{L^2} + \| \proj_{\Vse_{k}}^{\perp} \vse_{k+1}\|_{L^2} \langle \use_{k+1}, \bxi_h\rangle_{L^2} \notag \\
	&\overset{\1}{=} \llangle \use_{k+1}, \Hse_{k} \rrangle_{L^2} \bigl(\bK^{h}_k\bigr)^{-1} \llangle \Vse_{k}, \vse_{k+1} \rrangle_{L^2} +  \| \proj_{\Vse_{k}}^{\perp} \vse_{k+1}\|_{L^2} \langle \use_{k+1}, \bxi_h\rangle_{L^2}, \label{eq:SE-algebra}
\end{align}
where the final relation follows by applying the whitening transformations in Eq.~\eqref{eq:SE-whitening}.
On the other hand, recalling the definition of $\bT_h$, we expand
\[
\bigl(\uapx_{k+1}\bigr)^{\top} \bT_h \vapx_{k+1} + \bigl \| \proj_{\bV_k}^{\perp} \vapx_{k+1} \bigr\|_2 \frac{(\uapx_{k+1})^{\top} \bgamma_n}{\sqrt{d}} = \bigl(\uapx_{k+1}\bigr)^{\top} \sqrt{\Lambda} \bH_k \bigl(\bK^h_k\bigr)^{-1}\bV_k^{\top} \vapx_{k+1} + \bigl \| \proj_{\bV_k}^{\perp} \vapx_{k+1} \bigr\|_2 \frac{(\uapx_{k+1})^{\top} \bgamma_n}{\sqrt{d}}.
\]
It thus suffices to show that 
\begin{align*}
	\Bigl \lvert \bigl(\uapx_{k+1}\bigr)^{\top}\bH_k \bigl(\bK^h_k\bigr)^{-1}\bV_k^{\top} \vapx_{k+1} - \EE\bigl[\bigl(\use_{k+1}\bigr)^{\top} \Hse_k\bigr] \bigl(\bK^h_k\bigr)^{-1} \EE\bigl[\bigl(\Vse_k\bigr)^{\top} \vse_{k+1} \bigr] \Bigr \rvert  &\leq C_{\mathrm{SE}} k \Delta_1(k, \delta) + C_{\mathrm{SE}} k \sqrt{\frac{\log(10T^2/\delta')}{n}}, \\
	\Bigl \lvert \bigl \| \proj_{\bV_k}^{\perp} \vapx_{k+1} \bigr\|_2 \frac{(\uapx_{k+1})^{\top} \bgamma_n}{\sqrt{d}} -  \| \proj_{\Vse_{k}}^{\perp} \vse_{k+1}\|_{L^2} \langle \use_{k+1}, \bxi_h\rangle_{L^2}\Bigr \rvert &\leq C_{\mathrm{SE}} k \Delta_1(k, \delta) + C_{\mathrm{SE}} k \sqrt{\frac{\log(10T^2/\delta')}{n}}
\end{align*}
We handle each term in turn:
\begin{align*}
	\Bigl \lvert \bigl(\uapx_{k+1}\bigr)^{\top}\bH_k \bigl(\bK^h_k\bigr)^{-1}\bV_k^{\top} \vapx_{k+1} - \EE\bigl[\bigl(\use_{k+1}\bigr)^{\top} \Hse_k\bigr] \bigl(\bK^h_k\bigr)^{-1} \EE\bigl[\bigl(\Vse_k\bigr)^{\top} \vse_{k+1} \bigr] \Bigr \rvert \leq A + B, 
\end{align*}
where
\begin{align*}
	A &= \Bigl \lvert \bigl(\uapx_{k+1}\bigr)^{\top}\bH_k \bigl(\bK^h_k\bigr)^{-1}\bV_k^{\top} \vapx_{k+1} - \EE\bigl[\bigl(\use_{k+1}\bigr)^{\top} \Hse_k\bigr] \bigl(\bK^h_k\bigr)^{-1} \bV_k^{\top} \vapx_{k+1}\Bigr \rvert, \quad \text{ and }\\
	B &= \Bigl \lvert \EE\bigl[\bigl(\use_{k+1}\bigr)^{\top} \Hse_k\bigr] \bigl(\bK^h_k\bigr)^{-1} \bV_k^{\top} \vapx_{k+1} - \EE\bigl[\bigl(\use_{k+1}\bigr)^{\top} \Hse_k\bigr] \bigl(\bK^h_k\bigr)^{-1} \EE\bigl[\bigl(\Vse_k\bigr)^{\top} \vse_{k+1} \bigr]  \Bigr \rvert. 
\end{align*}
On $\Omega' \cap \Omega_k$, we invoke the inequalities~\eqref{eq:apx-ind-events-new-four-1} and~\eqref{eq:apx-ind-events-new-four-4} in conjunction with the triangle inequality to deduce that both
\begin{align*}
	\bigl \| \bigl(\uapx_{k+1}\bigr)^{\top} \bH_k -  \EE\bigl[\bigl(\use_{k+1}\bigr)^{\top} \Hse_k\bigr] \bigr \|_2 &\leq C_{\mathrm{SE}}\sqrt{k} \Delta_1(k, \delta) + C_{\mathrm{SE}} \sqrt{k} \sqrt{\frac{\log(10T^2/\delta')}{n}}, \quad \text{ and }\\
	\bigl \| \bigl(\vapx_{k+1}\bigr)^{\top} \bV_k -  \EE\bigl[\bigl(\vse_{k+1}\bigr)^{\top} \Vse_k\bigr] \bigr \|_2 &\leq C_{\mathrm{SE}} \sqrt{k} \Delta_1(k, \delta) + C_{\mathrm{SE}} \sqrt{k} \sqrt{\frac{\log(10T^2/\delta')}{n}},
\end{align*}
so that
\[
A \leq C_{\mathrm{SE}} \sqrt{k} \Delta_1(k, \delta) + C_{\mathrm{SE}} \sqrt{k} \sqrt{\frac{\log(10T^2/\delta')}{n}} \quad \text{ and } \quad B \leq C_{\mathrm{SE}} \sqrt{k} \Delta_1(k, \delta) + C_{\mathrm{SE}} \sqrt{k} \sqrt{\frac{\log(10T^2/\delta')}{n}}.
\]
Hence, 
\begin{align*}
	\Bigl \lvert \bigl(\uapx_{k+1}\bigr)^{\top}\bH_k \bigl(\bK^h_k\bigr)^{-1}\bV_k^{\top} \vapx_{k+1} - \EE\bigl[\bigl(\use_{k+1}\bigr)^{\top} \Hse_k\bigr] \bigl(\bK^h_k\bigr)^{-1} \EE\bigl[\bigl(\Vse_k\bigr)^{\top} \vse_{k+1} \bigr] \Bigr \rvert \leq C_{\mathrm{SE}} \sqrt{k} \Delta_1(k, \delta) + C_{\mathrm{SE}} \sqrt{k} \sqrt{\frac{\log(10T^2/\delta')}{n}}. 
\end{align*}
Turning to the second term, we have 
\begin{align*}
	\Bigl \lvert &\bigl \| \proj_{\bV_k}^{\perp} \vapx_{k+1} \bigr\|_2 \frac{(\uapx_{k+1})^{\top} \bgamma_n}{\sqrt{d}} -  \| \proj_{\Vse_{k}}^{\perp} \vse_{k+1}\|_{L^2} \langle \use_{k+1}, \bxi_h\rangle_{L^2}\Bigr \rvert \leq  \| \proj_{\Vse_{k}}^{\perp} \vse_{k+1}\|_{L^2} \cdot \Bigl \lvert \frac{(\uapx_{k+1})^{\top} \bgamma_n}{\sqrt{d}} -\langle \use_{k+1}, \bxi_h\rangle_{L^2} \Bigr \rvert \\
	&+ \langle \use_{k+1}, \bxi_h\rangle_{L^2} \cdot \bigl \lvert  \bigl \| \proj_{\bV_k}^{\perp} \vapx_{k+1} \bigr\|_2 - \| \proj_{\Vse_{k}}^{\perp} \vse_{k+1}\|_{L^2} \bigr \rvert + \Bigl \lvert \frac{(\uapx_{k+1})^{\top} \bgamma_n}{\sqrt{d}} -\langle \use_{k+1}, \bxi_h\rangle_{L^2} \Bigr \rvert \cdot \bigl \lvert  \bigl \| \proj_{\bV_k}^{\perp} \vapx_{k+1} \bigr\|_2 - \| \proj_{\Vse_{k}}^{\perp} \vse_{k+1}\|_{L^2} \bigr \rvert\\
	&\leq  C_{\mathrm{SE}} \Delta_1(k, \delta) + C_{\mathrm{SE}}  \sqrt{\frac{\log(10T^2/\delta')}{n}} + C_{\mathrm{SE}} \bigl \lvert  \bigl \| \proj_{\bV_k}^{\perp} \vapx_{k+1} \bigr\|_2 - \| \proj_{\Vse_{k}}^{\perp} \vse_{k+1}\|_{L^2} \bigr \rvert
\end{align*}
where the final inequality follows with probability at least $1 - \delta'$ by the inequality~\eqref{eq:corollary-73a-bound}.  It remains to bound the quantity $\bigl \lvert  \bigl \| \proj_{\bV_k}^{\perp} \vapx_{k+1} \bigr\|_2 - \| \proj_{\Vse_{k}}^{\perp} \vse_{k+1}\|_{L^2} \bigr \rvert$.  Following similar steps to establishing the bounds on $A$ and $B$, we deduce that
\begin{align} \label{ineq:vapx-perp-concentration}
\bigl \lvert  \bigl \| \proj_{\bV_k}^{\perp} \vapx_{k+1} \bigr\|_2 - \| \proj_{\Vse_{k}}^{\perp} \vse_{k+1}\|_{L^2} \bigr \rvert \leq C_{\mathrm{SE}} k^2 \Delta_1(k, \delta) + C_{\mathrm{SE}} k^2 \sqrt{\frac{\log(10T^2/\delta')}{n}},
\end{align}
Putting the pieces together yields the desired result. 


\paragraph{Proof of the inequality~\eqref{ineq:gradient-bound-F-b}.}
We consider two cases: when $\useperp_{k+1} = 0$ (almost surely) and otherwise.  

\medskip
\noindent \emph{Case 1: $\useperp_{k+1}= 0$ almost surely:} In this case, by Claim~\ref{clm:chains}, we see that $L^{u}_{k+1, k+1} = 0$ so that the stationary condition defining $\uapx_{k+1}$reads
\[
\Bigl\{ \sqrt{\Lambda} \hapx_{k+1}(\bH_k, \bgamma_n) - \uhatapx_{k+1}(\bU_k)\Bigr\} + \nabla \phi_{k+1}^u\bigl(\uapx_{k+1}(\bU_k, \bH_k, \bgamma_n)\bigr) = 0.
\]
Moreover, since by Claim~\ref{clm:chains}, $\langle \bxi_g, \vse_{k+1} \rangle_{L^2} = 0$, we have
\[
\partial H(\uapx_{k+1}) = \bT_h \bv_{k+1}^{\mathrm{apx}} - \bT_g^{\top} \bv_{k+1}^{\mathrm{apx}} + \| \proj^{\perp}_{\bV_k} \bv_{k+1}^{\mathrm{apx}} \|_2 \frac{\bgamma_n}{\sqrt{d}} - \nabla \phi_{k+1}^{u}(\uapx_{k+1}).
\]
Combining the previous displays and applying the triangle inequality yields
\[
	\| \partial H(\uapx_{k+1}) \|_2 \leq \Bigl \| \bT_h \vapx_{k+1} + \| \proj^{\perp}_{\bV_k} \bv_{k+1}^{\mathrm{apx}} \|_2 \frac{\bgamma_n}{\sqrt{d}} - \sqrt{\Lambda} \hapx_{k+1}(\bH_k, \bgamma_n/\sqrt{n}) \Bigr\|_2 + \Bigl \| \uhatapx_{k+1}(\bU_k) - \bT_g^{\top} \vapx_{k+1} \Bigr\|_2.
\]
We bound each of the terms on the RHS of the display above in turn.  Expanding
\[
\bT_h \vapx_{k+1}= \sqrt{\Lambda} \bH_k \bigl(\bK^h_k\bigr)^{-1} \langle \! \langle \bV_k, \vapx_{k+1} \rangle \! \rangle \quad \text{ and } \quad \hapx_{k+1}(\bH_k, \bgamma_n/\sqrt{n}) = \bH_k \bigl(\bK^h_k\bigr)^{-1} \langle \! \langle \Vse_k, \vse_{k+1} \rangle \! \rangle_{L^2} +  \| \proj_{\Vse_k}^{\perp} \vse_{k+1} \|_2  \cdot \frac{\bgamma_n }{\sqrt{n}},
\]
we thus have, on $\Omega_k \cap \Omega'$ that
\begin{align*}
	\Bigl \| \bT_h &\vapx_{k+1} + \| \proj^{\perp}_{\bV_k} \bv_{k+1}^{\mathrm{apx}} \|_2 \frac{\bgamma_n}{\sqrt{d}} - \sqrt{\Lambda} \hapx_{k+1}(\bH_k, \bgamma_n/\sqrt{n}) \Bigr\|_2\\
	&\leq \sqrt{\Lambda} \Bigl \| \bH_k \bigl(\bK^h_k\bigr)^{-1} \langle \! \langle \bV_k, \vapx_{k+1} \rangle \! \rangle - \bH_k \bigl(\bK^h_k\bigr)^{-1} \langle \! \langle \Vse_k, \vse_{k+1} \rangle \! \rangle_{L^2} \Bigr \|_2 +  \Bigl \lvert \| \proj_{\bV_k}^{\perp} \vapx_{k+1} \|_2 - \| \proj_{\Vse_k}^{\perp} \vse_{k+1} \|_2 \Bigr \rvert \cdot \frac{\| \bgamma_n \|_2}{\sqrt{d}}\\
	&\leq \sqrt{k \Lambda} \| \bH_k \|_{\mathrm{op}} \bigl \| \bigl(\bK^h_k\bigr)^{-1} \bigr\|_{\mathrm{op}} \bigl\|  \langle \! \langle \bV_k, \vapx_{k+1} \rangle \! \rangle -   \langle \! \langle \Vse_k, \vse_{k+1} \rangle \! \rangle_{L^2} \bigr \|_{\infty} + \Bigl \lvert \| \proj_{\bV_k}^{\perp} \vapx_{k+1} \|_2 - \| \proj_{\Vse_k}^{\perp} \vse_{k+1} \|_2 \Bigr \rvert \cdot \frac{\| \bgamma_n \|_2}{\sqrt{d}}\\
	&\leq C_{\mathrm{SE}} k^2 \Delta_1(k, \delta) + C_{\mathrm{SE}}k^2 \sqrt{\frac{\log(10T^2/\delta')}{n}} \numberthis \label{ineq:H-l2-bound},
\end{align*}
where in the final inequality, we invoked the bounds~\eqref{eq:apx-ind-events-new-four-1},~\eqref{eq:apx-ind-events-new-four-hh}, and~\eqref{ineq:vapx-perp-concentration}, which hold on $\Omega_k \cap \Omega'$.  On the other hand, 
\begin{align*}
\Bigl \| \uhatapx_{k+1}(\bU_k) - \bT_g^{\top} \vapx_{k+1} \Bigr\|_2 
	&= \Bigl \| \bU_k \bigl(\bK^g_k\bigr)^{-1} \langle \! \langle \bG_k, \vapx_{k+1} \rangle \! \rangle - \bU_k \bigl(\bK^g_k\bigr)^{-1} \langle \! \langle \Gse_k, \vse_{k+1} \rangle \! \rangle_{L^2} \Bigr \|_2\\
	&\leq \sqrt{k} \| \bG_k \|_{\mathrm{op}} \bigl \| \bigl(\bK^g_k\bigr)^{-1} \bigr\|_{\mathrm{op}} \bigl \| \llangle \bG_k, \vapx_{k+1} \rrangle - \llangle \Gse_k, \vse_{k+1} \rrangle_{L^2} \bigr \|_{\infty}\\
	&\leq C_{\mathrm{SE}} k^2 \Delta_1(k, \delta) + C_{\mathrm{SE}}k^2 \sqrt{\frac{\log(10T^2/\delta')}{n}} \numberthis \label{eq:bound-partial-F},
\end{align*}
where in the final inequality, we invoked the bounds~\eqref{eq:apx-ind-events-new-four-2} and~\eqref{eq:apx-ind-events-new-four-gg}, which hold on the event $\Omega' \cap \Omega_k$.  Combining the elements establishes the claim in the first case. 

\medskip
\noindent \emph{Case 2: $\| \useperp_{k+1} \|_2 > 0$, almost surely:}
Note that on $\Omega_k \cap \Omega'$, $\| \uapx_{k+1}\|_2 > 0$.  It thus follows that $F$ is differentiable at $\uapx_{k+1}$, and we compute
\[
\nabla H(\uapx_{k+1}) = \bT_h \vapx_{k+1} - \bT_g^{\top} \vapx_{k+1} + \| \proj_{\bV_k}^{\perp} \vapx_{k+1} \|_2 \frac{\bgamma_n}{\sqrt{d}} - \frac{\proj_{\bU_k}^{\perp} \uapx_{k+1}}{\| \proj_{\bU_k}^{\perp} \uapx_{k+1} \|_2} \cdot \langle \bxi_g, \vse_{k+1} \rangle_{L^2} - \nabla \phi_{k+1}^{u}(\uapx_{k+1}).
\]
On the other hand, from the KKT condition defining $\uapx_{k+1}$, we have 
\begin{align*}
	L_{k+1, k+1}^{u} \proj_{\bU_k} \uapx_{k+1} &+ L_{k+1, k+1}^u \proj^{\perp}_{\bU_k}\uapx_{k+1} -\Bigl\{ \sqrt{\Lambda}\hapx_{k+1}(\bH_k, \bgamma_n) - \uhatapx_{k+1}(\bU_k)\Bigr\} + \nabla \phi_{k+1}^u\bigl(\uapx_{k+1}\bigr) = 0.
\end{align*}
Adding the two preceding displays yields
\begin{align*}
	\| \nabla H(\uapx_{k+1}) \|_2 &\leq \Bigl \| \bT_h \vapx_{k+1} + \| \proj^{\perp}_{\bV_k} \bv_{k+1}^{\mathrm{apx}} \|_2 \frac{\bgamma_n}{\sqrt{d}} - \sqrt{\Lambda} \hapx_{k+1}(\bH_k, \bgamma_n) \Bigr\|_2 \\
	&\quad + \Bigl \| \uhatapx_{k+1}(\bU_k) + L_{k+1, k+1}^{u} \proj_{\bU_k} \uapx_{k+1} - \bT_g^{\top} \vapx_{k+1} + L_{k+1, k+1}^{u}\proj_{\bU_k}^{\perp} \uapx_{k+1} -  \frac{\proj_{\bU_k}^{\perp} \uapx_{k+1}}{\| \proj_{\bU_k}^{\perp} \uapx_{k+1} \|_2} \cdot \langle \bxi_g, \vse_{k+1} \rangle_{L^2}\Bigr\|_2 \\
	&\leq C_{\mathrm{SE}} k^2 \Delta_1(k, \delta) + C_{\mathrm{SE}}k^2 \sqrt{\frac{\log(10T^2/\delta')}{n}} + \Bigl \| \uhatapx_{k+1}(\bU_k) + L_{k+1, k+1}^{u} \proj_{\bU_k} \uapx_{k+1}  - \bT_g^{\top} \vapx_{k+1} \Bigr \|_2\\
	&\hspace{8cm}+\Bigl \lvert L_{k+1, k+1}^{u} -  \frac{\langle \bxi_g, \vse_{k+1} \rangle_{L^2}}{\| \proj_{\bU_k}^{\perp} \uapx_{k+1} \|_2}\Bigr \rvert \cdot \bigl \| \proj_{\bU_k}^{\perp} \uapx_{k+1}\bigr\|_2,
\end{align*}
where the final inequality follows from the bound~\eqref{ineq:H-l2-bound}.  We next bound each of the latter two terms.  Expanding, we have
\begin{align*}
	\Bigl \| \uhatapx_{k+1}(\bU_k) + L_{k+1, k+1}^{u} \proj_{\bU_k} \uapx_{k+1}  - \bT_g^{\top} \vapx_{k+1} \Bigr \|_2 &\leq \Bigl \|\bU_k \bigl(\bK_g\bigr)^{-1} \llangle \Gse_k, \vse_{k+1}\rrangle_{L^2} - \bT_g^{\top} \vapx_{k+1} \Bigr \|_2 \\
	&\quad + \Bigl \| L_{k+1, k+1}^{u} \proj_{\bU_k} \uapx_{k+1} - L_{k+1, k+1}^{u} \bU_k \bigl(\bK^g_k\bigr)^{-1} \llangle \Use_k, \use_{k+1} \rrangle_{L^2} \Bigr \|_2\\
	&\overset{\1}{\leq} C_{\mathrm{SE}} k^2 \Delta_1(k, \delta) + C_{\mathrm{SE}}k^2 \sqrt{\frac{\log(10T^2/\delta')}{n}} \\
	&\quad + L_{k+1, k+1}^{u} \Bigr \| \bU_k \bigl(\bU_k^{\top} \bU_k\bigr)^{-1} \bU_k^{\top} \uapx_{k+1} -  \bU_k \bigl(\bK^g_k\bigr)^{-1} \llangle \Use_k, \use_{k+1} \rrangle_{L^2}\Bigr \|_2\\
	&\leq C_{\mathrm{SE}} k^2 \Delta_1(k, \delta) + C_{\mathrm{SE}}k^2 \sqrt{\frac{\log(10T^2/\delta')}{n}},
\end{align*}
where step $\1$ follows from the bound~\eqref{eq:bound-partial-F} and the final inequality follows on $\Omega_k \cap \Omega'$.  It remains to bound
\begin{align*}
\Bigl \lvert L_{k+1, k+1}^{u} -  \frac{\langle \bxi_g, \vse_{k+1} \rangle_{L^2}}{\| \proj_{\bU_k}^{\perp} \uapx_{k+1} \|_2}\Bigr \rvert \cdot \bigl \| \proj_{\bU_k}^{\perp} \uapx_{k+1}\bigr\|_2 &\leq \Bigl \lvert \frac{\langle \bxi_g, \vse_{k+1} \rangle_{L^2}}{\| \proj_{\Use_{k}}^{\perp} \use_{k+1}\|_{L^2}} -  \frac{\langle \bxi_g, \vse_{k+1} \rangle_{L^2}}{\| \proj_{\bU_k}^{\perp} \uapx_{k+1} \|_2} \Bigr \rvert \cdot \| \uapx_{k+1}\|_2\\
&\leq C_{\mathrm{SE}} k^2 \Delta_1(k, \delta) + C_{\mathrm{SE}}k^2 \sqrt{\frac{\log(10T^2/\delta')}{n}},
\end{align*}
where the first inequality follows from non-expansiveness of projections as well as the representation $L_{k+1,k+1}^{u} =  \frac{\langle \bxi_g, \vse_{k+1} \rangle_{L^2}}{\| \proj_{\Use_{k}}^{\perp} \use_{k+1}\|_{L^2}}$, and the final inequality follows on $\Omega_k \cap \Omega'$.  Putting the pieces together yields the desired result. \hfill \qed 

\section{Ancillary results}\label{sec:aux-results}

We organize this appendix into three sections.

\subsection{Optimization problems on Hilbert spaces}

Let $\mathcal{L}(\mathbb{X}, \mathbb{Y})$ denote the vector space of linear maps from $\mathbb{X} \rightarrow \mathbb{Y}$. 
\begin{lemma}
	Let $\mathbb{X}$ and $\mathbb{Y}$ be separable Hilbert spaces.  Further, let $f: \mathbb{X}\rightarrow \mathbb{R}$ and $g: \mathbb{Y} \rightarrow \mathbb{R}$ be Fr{\'e}chet differentiable and strongly convex.  For a linear map $A\in \mathcal{L}(\mathbb{X}, \mathbb{Y})$, define the function $L: \mathbb{X} \times \mathbb{Y} \times \mathcal{L}(\mathbb{X}, \mathbb{Y}) \rightarrow \mathbb{R}$ as $L(x, y, A) := \langle Ax, y \rangle + f(x) - g(y)$.  Then, for each $A \in \mathcal{L}(\mathbb{X}, \mathbb{Y})$, there is a unique pair $(\bar{x}, \bar{y}) \in \mathbb{X} \times \mathbb{Y}$ such that $L(\bar{x}, \bar{y}) = \min_{x \in \mathbb{X}} \max_{y \in \mathbb{Y}}\, L(x, y;A)$.
\end{lemma}
\begin{proof}
Fix $A \in \mathcal{L}(\mathbb{X}, \mathbb{Y})$.  Since $y \mapsto -L(x, y)$ is strongly convex, it admits a unique minimizer~\citep[][Corollary 11.16]{bauschke2017convex}.  We thus define the function $\psi: \mathbb{X} \mapsto \mathbb{R}$ as $\psi(x) := \max_{y \in \mathbb{Y}}\, L(x, y)$.  Then, since $\psi$ is the maximum of a collection of strongly convex functions, it is strongly convex over $\mathbb{X}$ and hence admits a unique minimizer $\bar{x} := \argmin_{x \in \mathbb{X}}\, \psi(x)$.  Next, take $\bar{y} := \argmax_{y \in \mathbb{Y}}\, L(\bar{x}, y)$.
\end{proof}

\begin{lemma} \label{lem:prox-like}
Let $\mathbb{X}$ be a Hilbert space equipped with the inner product $\langle \cdot, \cdot \rangle$ that in turn induces the norm $\| \cdot \|$. Let $f: \mathbb{X} \to \mathbb{R}$ be some $C^1$ function with $L$-Lipschitz gradients, in that
\begin{align*}
\| \nabla f(x) - \nabla f(x') \| \leq L \| x - x' \| \quad \text{ for all } x, x' \in \mathbb{X}.
\end{align*}
Suppose we have a non-expansive linear map $A \in \mathcal{L}(\mathbb{X}, \mathbb{X})$ and some fixed $x_0 \in \mathbb{X} \setminus \{0\}$. Let $c_1 \geq 0$. Define the following optimization problems:
\begin{align}
\text{minimize }& c_1 \| A x \| + f(x)  \qquad &\text{ such that } x \in \mathbb{X} \tag{P1} \label{eq:prob1} \\
\text{minimize }& c_1 \| A x \| + f(x) - \inprod{x}{x_0} \quad &\text{ such that } x \in \mathbb{X} \text{ and } \inprod{x}{x_0} \geq 0 \tag{P2} \label{eq:prob2}.
\end{align}
Suppose there exists a unique optimal solution $x_1$ to problem~\eqref{eq:prob1}, and that this solution satisfies $\inprod{x_1}{x_0} = 0$.
If there exists an optimal solution $x_2$ to problem~\eqref{eq:prob2} satisfying $Ax_2 \neq 0$,
then $x_2$ must satisfy $\inprod{x_2}{x_0} > 0$.
\end{lemma}

\begin{proof}
It suffices to show that $x_2 \neq x_1$. Indeed, if $x_1 \neq x_2$, then we have
\begin{align*}
\| A x_2 \| + f(x_2) - \inprod{x_2}{x_0} \overset{\1}{\leq} \| A x_1 \| + f(x_1) \overset{\2}{<} \| A x_2 \| + f(x_2),
\end{align*}
and so $\inprod{x_2}{x_0} > 0$.
Here step $\1$ follows because $x_2$ is an optimal solution of problem~\eqref{eq:prob2} and $x_1$ is feasible for problem~\eqref{eq:prob2} with $\inprod{x_1}{x_0} = 0$. Step $\2$ follows since $x_1$ is the unique optimal solution of problem~\eqref{eq:prob1} and $x_2$ is feasible for problem~\eqref{eq:prob1}.
To show that $x_2 \neq x_1$, we split the proof into two cases.

\paragraph{Case 1: $Ax_1 = 0$:} By assumption, we have $Ax_2 \neq 0$, so we must have that $x_2 \neq x_1$. This completes the proof for this case.

\paragraph{Case 2: $Ax_1 \neq 0$.} 
First consider the subcase $c_1 > 0$. It suffices to show that there exists a feasible solution $\widetilde{x}$ to the constrained problem~\eqref{eq:prob2} such that its objective value is strictly less than the optimal objective value of the unconstrained problem~\eqref{eq:prob1}, i.e.,
\begin{align}
c_1 \| A \widetilde{x} \| + f(\widetilde{x}) - \inprod{\widetilde{x}}{x_0} < c_1 \| Ax_1 \| + f(x_1). \label{eq:sufficient-cond}
\end{align}
We claim that for any positive scalar $\delta < \min\Bigl\{\frac{\| Ax_1 \|}{2\| x_0 \|}, 2  \bigl( \frac{2}{c_1 \| Ax_1 \|} + L \bigr)^{-1} \Bigr\}$, the point $\widetilde{x} = x_1 + \delta x_0$ is feasible for the constrained problem~\eqref{eq:prob2} and satisfies Eq.~\eqref{eq:sufficient-cond}. Indeed, feasibility follows immediately, since 
\[
\inprod{\widetilde{x}}{x_0} = \inprod{x_1}{x_0} + \delta \inprod{x_0}{x_0} =  \delta \| x_0 \|^2  > 0.
\]

Now note that for all $x'$ such that $\| Ax' \| \neq 0$, we have
\[
\left\| \nabla (\| Ax' \|) -  \nabla (\| Ax_1 \|) \right\| = \left\| \frac{Ax'}{\| Ax' \|} -  \frac{Ax_1}{\| Ax_1 \|} \right\| \leq \frac{1}{\| Ax_1 \| \land \| Ax' \|} \cdot \| Ax' - Ax_1 \| \leq \frac{1}{\| Ax_1 \| \land \| Ax' \|} \cdot \| x' - x_1 \|,
\]
where the final inequality follows since $A$ is non-expansive.
Also note that since $x_1$ is the optimal solution to the unconstrained problem~\eqref{eq:prob1} and $Ax_1 \neq 0$, we have 
\[
\nabla \bigl(\| Ax \| + f(x) \bigr) \mid_{x = x_1} = 0.
\]
Furthermore, note that $\| A\widetilde{x} \| \geq \| A x_1 \| - \delta \| x_0 \| \geq \frac{1}{2} \| A x_1 \|$, where the final inequality follows since $\delta < \frac{\| Ax_1 \|}{2 \|x_0 \|}$. Thus, the function $x \mapsto \| A x \| + f(x)$ is $\{2/(\| A x_1 \|_2) + L\}$-smooth on the line segment connecting $x_1$ to $\widetilde{x}$. Using these properties in conjunction with Taylor's theorem yields
\[
c_1 \| A\widetilde{x} \| + f(\widetilde{x}) - \bigl\{c_1 \| Ax_1 \| + f(x_1)\bigr\} \leq \left( \frac{2}{c_1 \| Ax_1 \|} + L \right) \cdot \frac{\| \widetilde{x} - x_1 \|^2}{2} = \frac{\delta^2 \| x_0 \|^2}{2} \cdot \left( \frac{2}{c_1 \| Ax_1 \|} + L \right)
\]

To verify Eq.~\eqref{eq:sufficient-cond}, we may write
\begin{align*}
c_1 \| A\widetilde{x} \| + f(\widetilde{x}) -  \inprod{\widetilde{x}}{x_0} - \bigl\{c_1 \| Ax_1 \| + f(x_1)\bigr\} &\leq \frac{\delta^2 \| x_0 \|^2}{2} \cdot \left( \frac{2}{c_1 \| Ax_1 \|} + L \right) -  \delta \| x_0 \|^2 \overset{\1}{<} 0,
\end{align*}
where step $\1$ follows since $0 < \delta < 2  \bigl( \frac{2}{c_1 \| Ax_1 \|} + L \bigr)^{-1}$.

To handle the subcase $c_1 = 0$ note that $f$ is differentiable at $x_1$ and $L$-smooth. Thus, we have, for the same choice $\widetilde{x} = x_1 + \delta x_0$, the inequality
\[
f(\widetilde{x}) -  \inprod{\widetilde{x}}{x_0} - f(x_1) \leq \frac{L}{2} \| \widetilde{x} - x_1 \|^2 - \delta \| x_0 \|^2 = \delta \| x_0 \|^2 \left( \frac{L}{2} - 1 \right), 
\]
and the RHS is less than $0$ for any $\delta \in (0, 2/L)$.
\end{proof}

\subsection{Saddle updates and Gaussian conditioning}

Given a matrix $A \in \mathbb{R}^{n \times d}$ and differentiable functions $f: \mathbb{R}^{d} \rightarrow \mathbb{R}, g: \mathbb{R}^n \rightarrow \mathbb{R}$, we let $I_A$ denote the solutions to the following non-linear system of equations
\begin{align} \label{eq:def-image-A}
I_{A} := \bigl\{(x, y) \in \mathbb{R}^{d} \times \mathbb{R}^n: \, Ax = \nabla g(y) \text{ and } A^{\top} y = -\nabla f(x)\bigr\}.
\end{align}
\begin{lemma} \label{lem:measurable-saddle}
	Let $f: \mathbb{R}^d \rightarrow \mathbb{R}$ and $g: \mathbb{R}^n \rightarrow \mathbb{R}$ be differentiable and strongly convex.  For each $A \in \mathbb{R}^{n \times d}$, the non-linear system in~\eqref{eq:def-image-A} admits a unique solution $(\bar{x}(A), \bar{y}(A))$.  Moreover, the map $A \mapsto (\bar{x}(A), \bar{y}(A))$ is continuous.  
\end{lemma}
\begin{proof}
	We first establish that the non-linear system admits a unique solution.  To this end, define the function $L: \mathbb{R}^{d} \times \mathbb{R}^{n} \times \mathbb{R}^{n \times d} \rightarrow \mathbb{R}$ as $L(x, y, A) := \langle Ax, y \rangle + f(x) - g(y)$.  Note that, for every $x \in \mathbb{R}^d$ and $A \in \mathbb{R}^{n \times d}$, the function $\phi_x : y \mapsto L(x, y, A)$ is strongly concave and hence admits a unique maximizer so that the function $\psi: x \mapsto \max_{y \in \mathbb{R}^d} \phi(y)$ is well-defined.  Since $\psi$ is the maximum of strongly convex functions, it is also strongly convex.  Further, by Danskin's theorem~\citep[see, e.g.,][Corollary 1]{rockafellar1991special}, $\psi$ is differentiable with derivative $\nabla \psi(x) = A^{\top} \cdot \argmax_{y \in \mathbb{R}^d} \phi_x(y) + \nabla f(x)$.  Thus,  since $\psi$ is strongly convex and differentiable, it admits a unique minimizer $\bar{x}(A)$ that satisfies $\nabla \psi\bigl(\bar{x}(A)\bigr) = 0$.  Similarly, $\phi_{\bar{x}(A)}$ admits a unique maximizer $\bar{y}(A)$ that satisfies $\nabla \phi_{\bar{x}(A)}\bigl(\bar{y}(A)\bigr) = 0$.  Expanding each of these first order conditions yields the non-linear system in~\eqref{eq:def-image-A}, which by construction admits a unique solution.  
	
	We turn now to establishing continuity of the map $M: A \mapsto \bigl(\bar{x}(A), \bar{y}(A)\bigr)$.  To this end, consider a sequence of matrices $(A_{k}) \rightarrow A$.  It suffices to show that $M(A_{\ell}) \rightarrow (\bar{x}, \bar{y})$.  Next, let $L_k(x,y) := \langle A_k x, y \rangle + f(x) - g(y)$ and $L(x, y) := \langle Ax, y \rangle + f(x) - g(y)$ and note that $L_k(x, y) = L(x, y) +\bigl \langle (A_k - A)x, y \bigr\rangle$.  Further, set $(\tilde{x}, \tilde{y}) := M(A_k)$.  By definition, the following pair of inequalities hold for all $x \in \mathbb{R}^d$ and $y \in \mathbb{R}^n$,
	\begin{align}\label{ineq:saddle}
		L_k(\tilde{x}, y) \overset{(a)}{\leq} L_k(\tilde{x}, \tilde{y}) \overset{(b)}{\leq} L_k(x, \tilde{y}) \qquad \text{ and } \qquad L(\bar{x}, y) \overset{(c)}{\leq} L(\bar{x}, \bar{y}) \overset{(d)}{\leq} L(x, \bar{y}).
	\end{align}
	Consequently, taking $y = \bar{y}$ in~\eqref{ineq:saddle}(a), $y = \tilde{y}$ in~\eqref{ineq:saddle}(c) and expanding we find that 
	\[
	L(\tilde{x}, \bar{y}) + \bigl \langle (A_k - A)\tilde{x}, \bar{y} \bigr\rangle \leq L(\bar{x}, \bar{y})  + \bigl \langle (A_k - A)\bar{x}, \tilde{y} \bigr\rangle.
	\]
	For each fixed $y \in \mathbb{R}^n$, $L(\cdot, y)$ is strongly convex, so there exists a constant $C > 0$ such that 
	\[
	\frac{C}{2} \| \tilde{x} - \bar{x} \|_2^2 \leq \bigl \langle (A_k - A)\bar{x}, \tilde{y} \bigr\rangle -  \bigl \langle (A_k - A)\tilde{x}, \bar{y} \bigr\rangle = \bigl \langle (A_k - A)(\tilde{x} - \bar{x}), \tilde{y} \bigr\rangle + \bigl \langle \tilde{x}, (A_k - A)^{\top} (\tilde{y} - \bar{y})\bigr \rangle.
	\]
	Re-arranging and applying Cauchy--Schwarz, we find that
	\[
	\| \tilde{x} - \bar{x} \|_2^2 \leq C \epsilon \cdot \bigl( \| \tilde{x} - \bar{x} \|_2 + \| \tilde{x} - \bar{x} \|_2 \cdot \| \tilde{y} - \bar{y} \|_2\bigr).
	\]
	Repeating these steps, using the inequalities~\eqref{ineq:saddle}(b,d) instead yields a similar inequality for $\| \tilde{y} - \bar{y} \|_2^2$, so that 
	\[
	\| \tilde{x} - \bar{x} \|_2^2 + \| \tilde{y} - \bar{y} \|_2^2  \leq C \epsilon \cdot \bigl( \| \tilde{x} - \bar{x} \|_2 + \| \tilde{y} - \bar{y} \|_2 +2 \| \tilde{x} - \bar{x} \|_2 \cdot \| \tilde{y} - \bar{y} \|_2\bigr).
	\]
	Taking $\epsilon$ small enough to ensure that $C \epsilon \leq 1/2$ and applying Jensen's inequality yields
	\[
	\frac{1}{4} \bigl(\| \tilde{x} - \bar{x} \|_2 + \| \tilde{y} - \bar{y} \|_2\bigr)^2 \leq C \epsilon \cdot \bigl(\| \tilde{x} - \bar{x} \|_2 + \| \tilde{y} - \bar{y} \|_2\bigr), 
	\]
	from which the conclusion follows.
\end{proof}
Lemma~\ref{lem:measurable-saddle} implies that, for each saddle-point update~\eqref{eq:saddle-update}, $(\bu_{k +1}, \bv_{k + 1})$ is a measurable function of the data $\bX$.  In turn, this implies that if we interleave first-order updates and saddle-point updates to obtain the collection of iterates $\bU_k = [\bu_1 \vert \ldots \vert \bu_k] \in \mathbb{R}^{n \times k}$ and $\bV_k = [\bv_1 \vert \ldots \vert \bv_k] \in \mathbb{R}^{d \times k}$, there exists an almost surely unique disintegration (see, e.g.,~\citet{chang1997conditioning}) of the joint distribution $\mathsf{Law}\bigl(\bX, (\bU_k, \bV_k)\bigr)$ into conditional distributions $\mathsf{Law}(\bX \, \vert \, \bU_k, \bV_k)$, which we determine in the next lemma.  To this end, let $F \subseteq [k]$ denote the indices for which a first-order update~\eqref{eq:first-order} was taken and $S \subseteq [k]$ with $S \cap F = \emptyset$ denote the indices for which a saddle-point update~\eqref{eq:saddle-update} was taken.  Further, let $\tilde{F} = \{\ell \in F: \ell + 1 \in F\}$ and define the pair of subspaces
\[
\mathbb{U}_k := \mathrm{span}\bigl(\{\bu_{\ell}:\, \ell \in \tilde{F} \cup S\}\bigr) \qquad \text{ and } \qquad \mathbb{V}_k := \mathrm{span}\bigl(\{\bv_{\ell}:\, \ell \in \tilde{F} \cup S\}\bigr).
\]
Equipped with these definitions, we next have our main conditioning lemma.
\begin{lemma} \label{lem:conditioning}
	Let $\bX = (X_{ij})_{i \in [n], j \in [d]}$ and suppose that $X_{ij} \overset{\mathrm{i.i.d.}}{\sim} \mathsf{N}(0, 1)$.  Let $k \in \mathbb{N}$ and for $\ell \leq k$, let $(\bu_{\ell}, \bv_{\ell})$ be constructed from the updates~\eqref{eq:initialization}-\eqref{eq:updates-general}.  Writing $\bU_k = [\bu_1 \vert \ldots \vert \bu_k] \in \mathbb{R}^{n \times k}$ and $\bV_k = [\bv_1 \vert \ldots \vert \bv_k] \in \mathbb{R}^{d \times k}$, let $\mathcal{F}_k = \sigma(\bU_k, \bV_k)$.  Then, with $\widetilde{\bX} \in \mathbb{R}^{n \times d}$ an independent copy of $\bX$,
	\[
	\bX \mid \mathcal{F}_k \overset{d}{=} \bP_{\mathbb{U}_k} \bX \bP_{\mathbb{V}_k} + \bP_{\mathbb{U}_k}^{\perp} \bX\bP_{\mathbb{V}_k} + \bP_{\mathbb{U}_k} \bX\bP_{\mathbb{V}_k}^{\perp} + \bP_{\mathbb{U}_k}^{\perp} \widetilde{\bX} \bP_{\mathbb{V}_k}^{\perp}.
	\]
\end{lemma}
\begin{proof}
Let $F, \tilde{F}, S$ be as defined before the statement of the lemma.  For each $\ell \in S \cup \tilde{F}$, let $\bh_{\ell} = \bX^{\top} \bu_{\ell}$ and $\bb_{\ell} = \bX\bv_{\ell}$.  Note that, by Lemma~\ref{lem:X-decomp}, it holds that $\bX \bv_{\ell} = \nabla \phi_{\ell}^{u}(\bu_{\ell}) = \bb_{\ell}$ and $\bX^{\top} \bu_{\ell} = - \nabla \phi_{\ell}^{v}(\bv_{\ell}) = \bh_{\ell}$.  Next, form $\widetilde{\bU}_k \in \mathbb{R}^{n \times \lvert S \cup \tilde{F} \rvert}$ and  $\widetilde{\bV}_k \in \mathbb{R}^{d \times \lvert S \cup \tilde{F} \rvert}$ by collecting the vectors $\{\bu_{\ell}: \, \ell \in S \cup \tilde{F}\}$ and $\{\bv_{\ell}: \, \ell \in S \cup \tilde{F} \}$ into matrices.  We then claim that, for all $U \in \mathbb{R}^{n \times k}, V \in \mathbb{R}^{n \times k}, H \in \mathbb{R}^{d \times \lvert S \cup \tilde{F} \rvert }$, and $B \in \mathbb{R}^{n \times \lvert S \cup \tilde{F} \rvert }$, the following events are equivalent.
\[
\Bigl\{\bU_k = U,\, \bV_k =  V, \bH_k = H, \bB_k = B \Bigr\} = \bigcup_{\ell \in S \cup \tilde{F}} \Bigl\{\bX v_{\ell} = b_{\ell},\, -\bX^{\top} u_{\ell} = h_{\ell}\Bigr\}.
\]
Note that for first-order updates, this equivalence is by definition, and for saddle-point updates this follows from Lemma~\ref{lem:measurable-saddle}.  Thus, $\mathcal{F}_k$ is equivalent to a set of linear constrains on $\bX$.  The desired claim then follows from standard properties of Gaussian matrices~\citep[see, e.g.,][Lemma 3.1]{montanari2023adversarial}.
\end{proof}

\subsection{Sub-exponential concentration for pseudo-Lipschitz functions}

Our last ancillary lemma is a concentration inequality for order-$2$ pseudo-Lipschitz functions of (normalized) Gaussian vectors.  The proof is a straightforward modification of the Gaussian Lipschitz concentration inequality~\citep[see, e.g.,][]{tao2012topics,wainwright2019high} and builds upon the Maurey--Pisier interpolation technique (see~\citet{milman1986asymptotic}). 
\begin{lemma} \label{lem:pseudo-Lipschitz-concentration}
	Suppose that $f: \mathbb{R}^d \rightarrow \mathbb{R}$ is an order-$2$ pseudo-Lipschitz function with constant $L$.  If $X \sim \mathsf{N}(0, I_d)$, then 
	\[
	\bigl \| f(X/\sqrt{d}) - \mathbb{E}[f(X/\sqrt{d})] \bigr\|_{\psi_1} \leq  \frac{\pi}{2} \sqrt{\frac{5}{\ln(2)}} \cdot \frac{L}{\sqrt{d}}.
	\]
\end{lemma}
\begin{proof}
	First, note that since $f$ is order-$2$ pseudo-Lipschitz, it is locally Lipschitz so that by Rademacher's theorem it is differentiable almost everywhere with respect to the Lebesgue measure.  Hence, it suffices to prove the claim for functions $f$ which are both differentiable as well as pseudo-Lipschitz.  
	
	Next, let $g: \mathbb{R}^d \rightarrow \mathbb{R}$ be such that $g(x) = f(x/\sqrt{d})$.  Fix $\lambda > 0$, and apply~\citet[][Lemma 2.1]{wainwright2019high} to the convex function $t \mapsto e^{\lambda t}$ to obtain the bound
	\[
	\mathbb{E} \Bigl[ \exp\bigl\{\lambda \bigl(g(X) - \mathbb{E}[g(X)]\bigr)\bigr\} \Bigr] \leq \mathbb{E}\Bigl[\exp\Bigl\{\frac{\lambda \pi}{2} \nabla g(X)^{\top} Y\Bigr\}\Bigr] = \mathbb{E}\Bigl[\exp\Bigl\{\frac{\lambda^2 \pi^2}{8} \| \nabla g(X) \|_2^2\Bigr\}\Bigr],
	\]
	where $Y \sim \mathsf{N}(0, I_d)$ and is independent of $X$.  From the pseudo-Lipschitz nature of $f$, we deduce that $\| \nabla g(X) \|_2 \leq \frac{L}{\sqrt{d}}(1 + 2 \| X/\sqrt{d} \|_2)$.  Consequently, 
	\[
	\mathbb{E} \Bigl[ \exp\bigl\{\lambda \bigl(g(X) - \mathbb{E}[g(X)]\bigr)\bigr\} \Bigr] \leq  \mathbb{E}\Bigl[\exp\Bigl\{\frac{\lambda^2 L^2 \pi^2}{4d} \Bigl(1 + \frac{4}{d} \| X \|_2^2\Bigr) \Bigr\}\Bigr]  \leq \exp\Bigl\{\frac{5\lambda^2 L^2 \pi^2}{4d} \Bigr\},
	\]
	where the final inequality holds for all $\lambda < d/(L\pi\sqrt{2})$.  We thus deduce that 
	\[
	\bigl\| g(X) - \mathbb{E}[g(X)]\bigr\|_{\psi_1} \leq \frac{\pi}{2} \sqrt{\frac{5}{\ln(2)}} \cdot \frac{L}{\sqrt{d}},
	\]
	which concludes the proof. 
\end{proof}
\end{document}